%% file: RDSmain.tex
\documentclass{amsbook}
\usepackage{dsfont}
\usepackage{upref}
\usepackage{mathrsfs}
 \usepackage{type1cm}

\usepackage{graphicx}        % standard LaTeX graphics tool
\usepackage{multicol}        % used for the two-column index
\usepackage[bottom]{footmisc}% places footnotes at page bottom

\usepackage{diagrams}
\usepackage{amsmath}%, amstext, amsthm}
\usepackage{amsfonts, latexsym, color}

\allowdisplaybreaks

\newtheorem{thm}{Theorem}[chapter]
\newtheorem{prop}[thm]{Proposition}
\newtheorem{lem}[thm]{Lemma}

\newtheorem{cor}[thm]{Corollary}

\newtheorem{fact}[thm]{Fact}

\newtheorem{exe}[thm]{Example}

\newtheorem{defprop}[thm]{Definition-Proposition}
\newtheorem{dfn}[thm]{Definition}

\newtheorem{rem}[thm]{Remark}

\numberwithin{section}{chapter}
\numberwithin{equation}{section}

\newcommand{\ben}{\begin{enumerate}}
\newcommand{\een}{\end{enumerate}}

\newcommand{\bea}{\begin{eqnarray}}
\newcommand{\ba}{\begin{array}}
\newcommand{\bean}{\begin{eqnarray*}}
\newcommand{\ea}{\end{array}}
\newcommand{\eea}{\end{eqnarray}}
\newcommand{\eean}{\end{eqnarray*}}
\newcommand{\beq}{\begin{equation}}
\newcommand{\eeq}{\end{equation}}
\newcommand{\bthm}{\begin{thm}}
\newcommand{\ethm}{\end{thm}}
\newcommand{\blem}{\begin{lem}}
\newcommand{\elem}{\end{lem}}
\newcommand{\bprop}{\begin{prop}}
\newcommand{\eprop}{\end{prop}}
\newcommand{\bcor}{\begin{cor}}
\newcommand{\ecor}{\end{cor}}
\newcommand{\bdfn}{\begin{dfn}}
\newcommand{\edfn}{\end{dfn}}
\newcommand{\brem}{\begin{rem}}
\newcommand{\erem}{\end{rem}}
\newcommand{\bpf}{\begin{proof}}
\newcommand{\epf}{\end{proof}}
\newcommand{\bfact}{\begin{fact}}
\newcommand{\efact}{\end{fact}}
\newcommand{\bdefprop}{\begin{defprop}}
\newcommand{\edefprop}{\end{defprop}}

\newcommand{\cE}{\mathcal{E}\!}

\alph{enumii} \roman{enumiii}

\newcommand{\cD}{\mathcal{D}}
\def\cA{\mathcal A}             \def\cB{\mathcal B}       \def\cC{\mathcal C}
\def\cH{\mathcal H}             \def\cF{\mathcal F}
\def\cL{{\mathcal L}}           \def\cM{\mathcal M}        \def\cP{{\mathcal P}}
                    \def\cJ{\mathcal J}
                    \def\cO{\mathcal O}
                     \def\cI{\mathcal I}
\def\cG{\mathcal G}             \def\cK{\mathcal K}       \def\cR{\mathcal R}

\def\endpf{\qed}

\def\N{{\mathbb N}}                \def\Z{{\mathbb Z}}      \def\R{{\mathbb R}}
\def\C{{\mathbb C}}                      \def\oc{\hat \C}
\def\Q{{\mathbb Q}}
\def\and{\text{ and }}

        \def\diam{\text{\rm {diam}}}

\def\H{\text{{\rm H}}}     %\def\HD{\text{{\rm HD}}}

         \def\P{\text{{\rm P}}}     \def\Id{\text{{\rm Id}}}

\def\Ha{\mathcal H}
        \def\Pa{{\mathcal P}}

\def\a{\alpha}                \def\b{\beta}             \def\d{\delta}
\def\De{\Delta}               \def\e{\varepsilon}          
\def\g{\gamma}                \def\Ga{\Gamma}
              \def\om{\omega}           
               \def\sg{\sigma}
               \def\th{\theta}           
\def\ka{\kappa}               \def\vp{\varphi}          \def\vr{\varrho}

\def\bi{\bigcap}              \def\bu{\bigcup}
\def\({\bigl(}                \def\){\bigr)}
\def\lt{\left}                \def\rt{\right}

\def\ld{\ldots}               \def\bd{\partial}         \def\^{\tilde}

\def\es{\emptyset}            \def\sms{\setminus}
\def\sbt{\subset}             \def\spt{\supset}

\def\sp{\medskip}             \def\fr{\noindent}        

\def\ov{\overline}

\def\om{\omega}

\def\supp{\text{{\rm supp}}}
\def\endpf{{\hfill $\square$}}
\def\Fa{\mathcal F}

\newcommand{\1}{\mathds{1}}
\newcommand{\Bc}{{\mathcal B}}

\newcommand{\lam}{\lambda}

\newcommand{\ph}{\varphi}
\newcommand{\al}{\alpha}
\newcommand{\ga}{\gamma}

\newcommand{\pfm}{\mathcal{L}}
\newcommand{\hcL}{\hat{\mathcal{L}}}
\newcommand{\tcL}{\tilde{\mathcal{L}}}

\DeclareMathOperator{\HD}{HD}
\DeclareMathOperator{\Arg}{Arg}
\DeclareMathOperator{\Range}{Range}
\DeclareMathOperator{\Dom}{Dom}
\DeclareMathOperator{\Leg}{L}
\DeclareMathOperator{\Legc}{{\bf L}}
\DeclareMathOperator{\cl}{cl}
\DeclareMathOperator{\graph}{graph}
\DeclareMathOperator*{\slim}{s-lim}
\DeclareMathOperator{\pdist}{Pdist}

\newcommand{\bex}{\begin{exe}}
\newcommand{\eex}{\end{exe}}

\DeclareMathOperator{\re}{Re}
\DeclareMathOperator{\im}{Im}
\DeclareMathOperator{\Mod}{mod}

\newcounter{assumei}

\newcounter{assume}

\newcommand\J{{\mathcal J}}

\def\shift{\theta}

\def\orb{{\cO_{x_0}}}

\def\holder{\cH^\al _m (\cJ )}
\def\holderx{\cH^\al (\cJ_x )}

\def\cone{\Lambda  ^s}
\def\conex{\Lambda _x ^s}

\def\hvarphi{\hat{\vp}}

\title{Distance Expanding Random Mappings, Thermodynamic Formalism,
   Gibbs Measures  and  Fractal Geometry}
\author{Volker Mayer}
 \address{Universit\'e de Lille I, UFR de
   Math\'ematiques, UMR 8524 du CNRS, 59655 Villeneuve d'Ascq Cedex,
   France} \email{volker.mayer@math.univ-lille1.fr \newline
   \hspace*{0.42cm} \it Web: \rm math.univ-lille1.fr/$\sim$mayer}

\author{Bart{\l}omiej Skorulski} 
\address{Departamento de
  Matem{\'a}ticas, Universidad Cat{\'o}lica del Norte, Avenida Angamos
  0610, Antofagasta, Chile} \email{bskorulski@ucn.cl}

\author{Mariusz Urba\'nski}
\address{Department of
  Mathematics, University of North Texas, Denton, TX 76203-1430, USA}
\email{urbanski@unt.edu \newline \hspace*{0.42cm} \it Web: \rm
  www.math.unt.edu/$\sim$urbanski}

\begin{document}

\begin{abstract}
  In this paper we introduce measurable expanding random systems,
  develop the thermodynamical formalism and establish, in particular,
  exponential decay of correlations and analyticity of the expected
  pressure although the spectral gap property does not hold.  This
  theory is then used to investigate fractal properties of conformal
  random systems.  We prove a Bowen's formula and develop the
  multifractal formalism of the Gibbs states. Depending on the
  behavior of the Birkhoff sums of the pressure function we get a
  natural classifications of the systems into two classes:
  \emph{quasi-deterministic} systems which share many properties of
  deterministic ones and \emph{essential} random systems which are
  rather generic and never bilipschitz equivalent to deterministic systems.
   We show in the essential case that the Hausdorff
  measure vanishes which refutes a conjecture of Bogensch\"utz and
  Ochs.  We finally give applications of our results to various
  specific conformal random systems and positively answer a question
  of Br\"uck and B\"urger concerning the Hausdorff dimension of random
  Julia sets.
 \vspace{8cm}
\end{abstract}

\maketitle


%    Address
\thanks{The second author was supported
  by FONDECYT Grant No. 11060538, Chile and Research Network on Low
  Dimensional Dynamics, PBCT ACT 17, CONICYT, Chile. The research of
  the third author is supported in part by the NSF Grant DMS
  0700831. A part of his work has been done while visiting the Max
  Planck Institute in Bonn, Germany. He wishes to thank the institute
  for support.}

\frontmatter

%\setcounter{page}{2}

%\tableofcontents

\tableofcontents
%\mainmatter

%\setcounter{page}{5}
\mainmatter
\input{RDSIntro}

\input{RDSsection2}
\input{RDSsection3}

\input{RDSsection4}
\input{RDSsection5}
\input{RDSsection6}
\input{RDSsection7}
\input{RDSsection8}
\input{RDSappendix}

%%%%Spring
%\backmatter
%\bibliographystyle{spmpsci}
\bibliographystyle{plain}

\bibliography{random}

%\printindex

%\include{index}

%\see
\end{document}

%% file: RDSIntro.tex
\chapter{Introduction}

In this manuscript we develop the thermodynamical formalism for
\emph{measurable expanding random mappings}. This theory is then 
applied in the context of conformal expanding random mappings where we
deal with the fractal geometry of fibers.

Distance expanding maps have been introduced for the first time in
Ruelle's monograph \cite{Rue78}. A systematic account of the dynamics
of such maps, including the thermodynamical formalism and the
multifractal analysis, can be found in \cite{PrzUrbXX}. One of the
main features of this class of maps is that their definition does not
require any differentiability or smoothness condition. Distance
expanding maps comprise symbol systems and expanding maps of smooth
manifolds but go far beyond. This is also a characteristic feature of
our approach.

In this manuscript we define measurable expanding random maps. \rm The
randomness is modeled by an invertible ergodic transformation $\theta$
of a probability space $(X,\cB ,m)$.  We investigate the dynamics of
compositions
\begin{displaymath}
T^n_x = T_{\theta ^{n-1} (x)}\circ ...\circ T_x \;\; , \; n\geq 1,
\end{displaymath}
where the $T_x:\cJ_x \to \cJ_{\theta (x)}$ ($x\in X$) is a distance
expanding mapping.  These maps are only supposed to be measurably
expanding in the sense that their expanding constant is measurable and a.e. $\gamma _x >1$
 or $\int \log \ga _x \, dm(x)>0$.

In so general setting we first build the thermodynamical formalism for
arbitrary H\"older continuous potentials $\varphi_x$. We show, in
particular, the existence, uniqueness and ergodicity of a family of
\emph{Gibbs measures} $\{\nu_x\}_{x\in X}$.  Following ideas of Kifer
\cite{Kif92}, these measures are first produced in a pointwise manner
and then we carefully check their measurability.
% In Remark 3.3 and 3.4 of \cite{KhaKif96} Kifer and Khanin indicated a
% possibility to build thermodynamic formalism for random distance
% expanding maps. They saw the obstacles to build such formalism in the
% lack of Markov partitions since they needed them in order to
% use symbolic representation from \cite{BogGun95}. 
% In our approach we
% do not need any Markov partitions or (even auxiliary) symbol dynamics.
% Moreover, in \cite{BogGun95} and \cite{KhaKif96} all fibers are
% required to be contained in the same compact metric space. We do
% not even need these fibers to be contained in one metric space.
Often in the literature all fibres are contained in one and the same
compact metric space and symbolic dynamics plays a prominet role. Our
approach does not require the fibres to be contained in one metric
space neither we need any Markov partitions or, even auxiliary,
symbol dynamics.

Our results
contain those in \cite{BogGun95} and in \cite{Kif92} (see also the
expository article \cite{KifPei06}).
%Kifer actually works under weaker assumption in this respect.
Throughout the entire manuscript where it is possible we avoid, in
hypotheses, absolute constants.  Our feeling is that in the context of
random systems all (or at least as many as possible) absolute
constants appearing in deterministic systems should become measurable
functions.  With this respect the thermodynamical formalism developed
in here represents also, up to our knowledge, new achievements in the
theory of random symbol dynamics or smooth expanding random maps
acting on Riemannian manifolds.

% Our approach to the thermodynamical formalism stems primarily from the
% classical method presented by Bowen in \cite{Bow75} and undertaken by
% Kifer in \cite{Kif92}. 
Unlike recent trends aiming to employ the method
of Hilbert metric (as for example in \cite{DenGor99}, \cite{Kif08},
\cite{Rug05}, \cite{Rug07}) our approach to the thermodynamical formalism
stems primarily from the classical method presented by Bowen in \cite{Bow75}
and undertaken by Kifer \cite{Kif92}.  Developing it in the context of random
dynamical systems we demonstrate that it works well and does not lead
to too complicated (at least to our taste) technicalities. The
measurability issue mentioned above results from convergence of the
Perron-Frobenius operators. We show that this convergence is
exponential, which implies exponential decay of correlations. These
results precede investigations of a pressure function $x\mapsto
P_x(\ph )$ which satisfies the property
$$\nu_{\theta (x)} (T_x(A)) =e^{P_x(\ph )} \int _A e^{-\ph _x}d\nu_x $$
where $A$ is any measurable set such that $T_x{|_A}$ is injective. The
integral, against the measure $m$ on the base $X$, of this function is
a central parameter $\cE P(\varphi)$ of random systems called the
\emph{expected pressure}. If the potential $\ph$ depends analytically
on parameters, we show that the expected pressure also behaves real
analytically.  We would like to mention that, contrary to the
deterministic case, the spectral gap methods do not work in the random
setting. Our proof utilizes the concept of complex cones introduced by
Rugh in \cite{Rug07}, and this is the only place, where
we use the projective metric.

We then apply the above results mainly to investigate fractal
properties of fibers of \emph{conformal random systems}. They include
Hausdorff dimension, Hausdorff and packing measures, as well as
multifractal analysis.  First, we establish a version of Bowen's
formula (obtained in a somewhat different context in \cite{BogOch99})
showing that the Hausdorff dimension of almost every fiber $\cJ_x$ is
equal to $h$, the only zero of the expected pressure $\cE P(\ph_t)$,
where $\ph_t=-t\log |f'|$ and $t\in \R$. Then we analyze the behavior
of $h$--dimensional Hausdorff and packing measures. It turned out that
the random dynamical systems split into two categories. Systems
from the first category, rather exceptional, behave like deterministic
systems. We call them, therefore, \emph{quasi-deterministic}. For them
the Hausdorff and packing measures are finite and positive. Other
systems, called \emph{essentially random}, are rather generic. For
them the $h$--dimensional Hausdorff measure vanishes while the
$h$-packing measure is infinite. This, in particular, refutes the
conjecture stated by Bogensch\"utz and Ochs in \cite{BogOch99} that
the $h$--dimensional Hausdorff measure of fibers is always positive
and finite. In fact, the distinction between the quasi-deterministic
and the essentially random systems is determined by the behavior of
the Birkhoff sums
\begin{displaymath}
  P_x^n (\ph) = P_{\theta
    ^{n-1}(x)}(\ph )+...+ P_x(\ph )
\end{displaymath}
of the pressure function for potential $\ph_h = -h\log |f'|$. If these
sums stay bounded then we are in the quasi-deterministic case.  On the
other hand, if these sums are neither bounded below nor above, the
system is called essentially random. The behavior of $P_x^n$, being
random variables defined on $X$, the base map for our skew product
map, is often governed by stochastic theorems such as the law of the
iterated logarithm whenever it holds.  This is the case for our
primary examples, namely conformal DG-systems and classical conformal
random systems.  We are then in position to state that the
quasi-deterministic systems correspond to rather exceptional case
where the asymptotic variance $\sigma^2 =0$. Otherwise the system is
essential.

The fact that Hausdorff measures in the Hausdorff dimension vanish has
further striking geometric consequences.  Namely, almost all fibers of
an essential conformal random system are not bi-Lipschitz equivalent
to any fiber of any quasi-deterministic or deterministic conformal
expanding system. In consequence almost every fiber of an essentially
random system is not a geometric circle nor even a piecewise analytic
curve.  We then show that these results do hold for many explicit
random dynamical systems, such as conformal DG-systems, classical
conformal random systems, and, perhaps most importantly, Br\"uck and
B\"uger polynomial systems. As a consequence of the techniques we
have developed, we positively answer the question of Br\"uck and
B\"uger (see \cite{BruBug03} and Question 5.4 in \cite{Bru01}) of
whether the Hausdorff dimension of almost all naturally defined random
Julia set is strictly larger than 1. We also show that in this same
setting the Hausdorff dimension of almost all Julia sets is strictly
less than 2.

Concerning the multifractal spectrum of Gibbs measures on fibers, we
show that the multifractal formalism is valid, i.e. the multifractal
spectrum is Legendre conjugated to a temperature function. As usual,
the temperature function is implicitly given in terms of the expected
pressure.  Here, the most important, although perhaps not most
strikingly visible, issue is to make sure that there exists a set
$X_{ma}$ of full measure in the base such that the multifractal
formalism works for all $x\in X_{ma}$. 

If the system is in addition uniformly expanding then we provide real analyticity
of the pressure function. This part is based on work by Rugh \cite{Rug05} and 
it is the only place where we work with the Hilbert metric. As a consequence and via Legendre transformation
we obtain real analyticity of the multifractal spectrum.

We would like to thank Yuri I. Kifer for his remarks which improved the final version of this manuscript.

%%% Local Variables:
%%% mode: latex
%%% TeX-master: "RDSmain"
%%% End:

%% file: RDSsection2.tex
\chapter{Expanding Random Maps}
\label{ch:GREM}
For the convenience of the reader, we first give some introductory examples.
In the remaining part of this chapter we present
 the general framework of \emph{expanding random maps.}
 
 \section{Introductory examples}
 \label{sec:Introductory examples}
Before giving the formal definitions of expanding random maps, let us now
consider some typical examples. 

The first one is a known random version of the Sierpinski gasket (see
for example \cite{Fal97}). Let $\Delta=\Delta (A,B,C)$ be a
triangle with vertexes $A,B,C$ and choose $a\in (A,B)$, $b\in (B,C)$
and $c\in (C,A)$.  Then we can associate to $x=(a,b,c)$ a map
$$f_x: \Delta (A,a,c)\cup \Delta (a,B,b)\cup \Delta (b,C,a) \rightarrow \Delta$$
such that the restriction of $f_x$ to each one of the three
subtriangles is a affine map onto $\Delta$. The map $f_x$ is nothing
else than the generator of a deterministic Sierpinski gasket. Note
that this map can me made continuous by identifying the vertices
$A,B,C$.
 
\begin{figure}[!h]
 \begin{center}
  \includegraphics[scale=0.4]{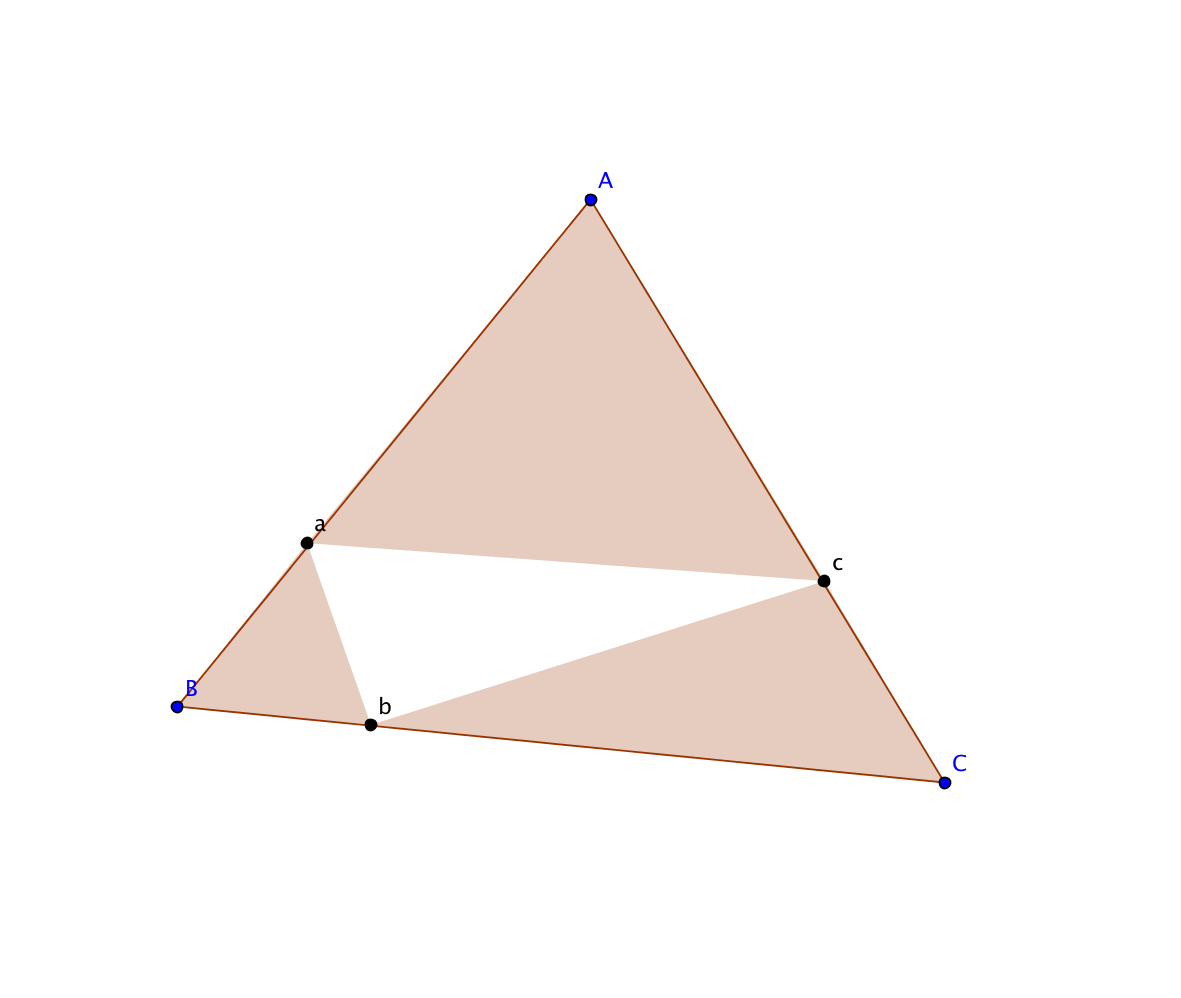}
       \includegraphics[scale=0.4]{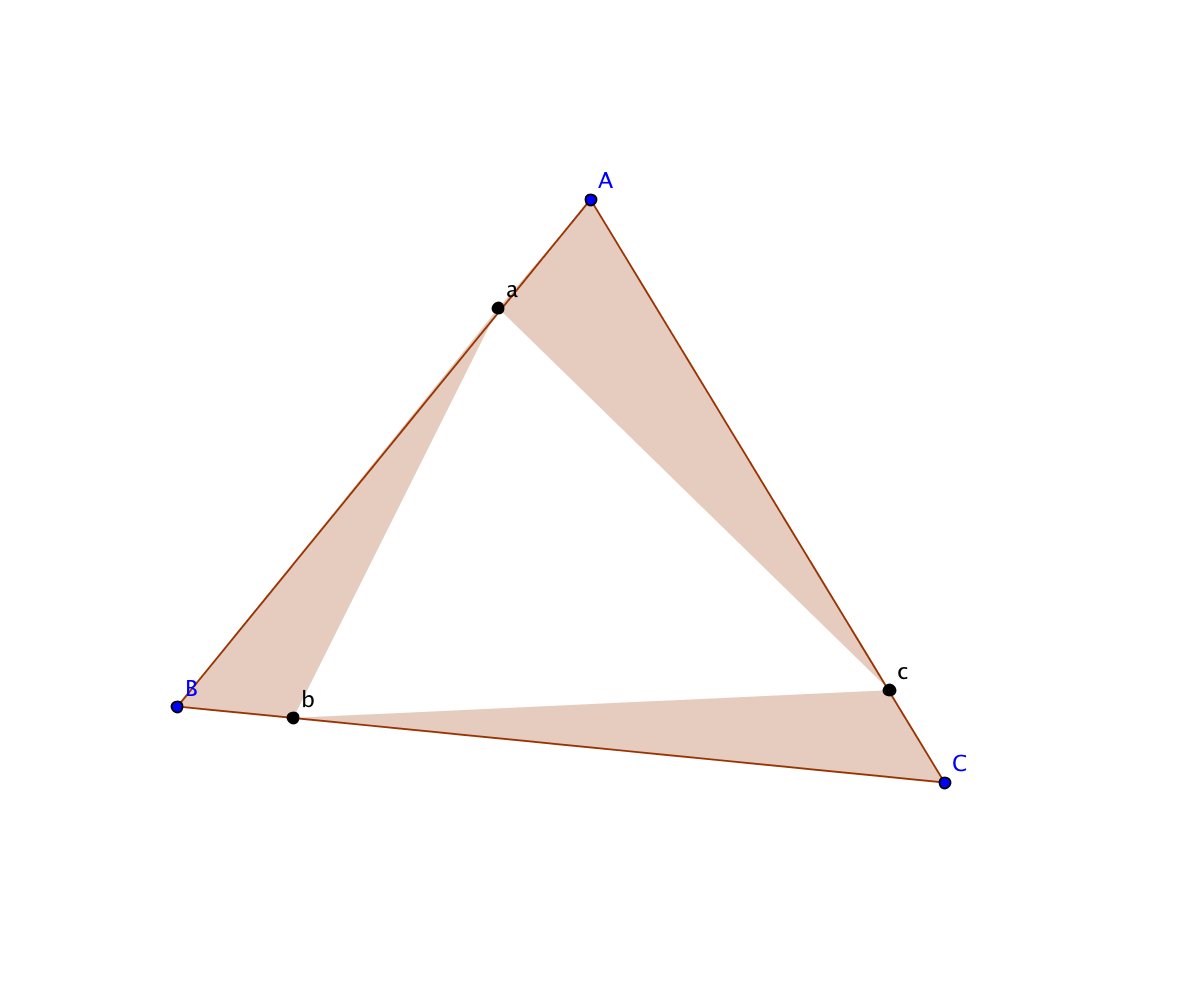}
         \caption{Two different generators of Sierpinski gaskets.}
      \end{center}
   \label{Figure 1}
\end{figure}

Now, suppose $x_1=(a_1,b_1,c_1), x_2=(a_2,b_2,c_2),...$ are chosen
randomly which, for example, may mean that
they form sequences of three dimensional independent and
identically distributed (i.i.d.) random variables. Then they generate compact sets
$$ \J_{x_1,x_2,x_3,...}= \bigcap_{n\geq 1} (f_{x_n}\circ ...\circ f_{x_1})^{-1}(\Delta) $$
called \emph{random Sierpi\'nski gaskets}\index{random Sierpi\'nski
  gasket} having the invariance property $
f_{x_1}^{-1}(\J_{x_2,x_3,...})= \J_{x_1,x_2,x_3,...}$.  For a little bit
simpler example of random Cantor sets we refer the reader to
Section~\ref{sec:random-cantor-set}. In that example we provide a more
detailed analysis of such random sets.

Such examples admit far going generalizations. First of all, we will
consider much more general random choices than i.i.d. ones. 
We model randomness by taking a probability space $(X,\cB ,m)$
along with an invariant ergodic transformation $\shift:X\to X$. 
This point of view was up to our knowledge introduced by the Bremen group
(see \cite{Arnold98}).

\begin{figure}[h]
\centering
  \includegraphics[scale=0.5]{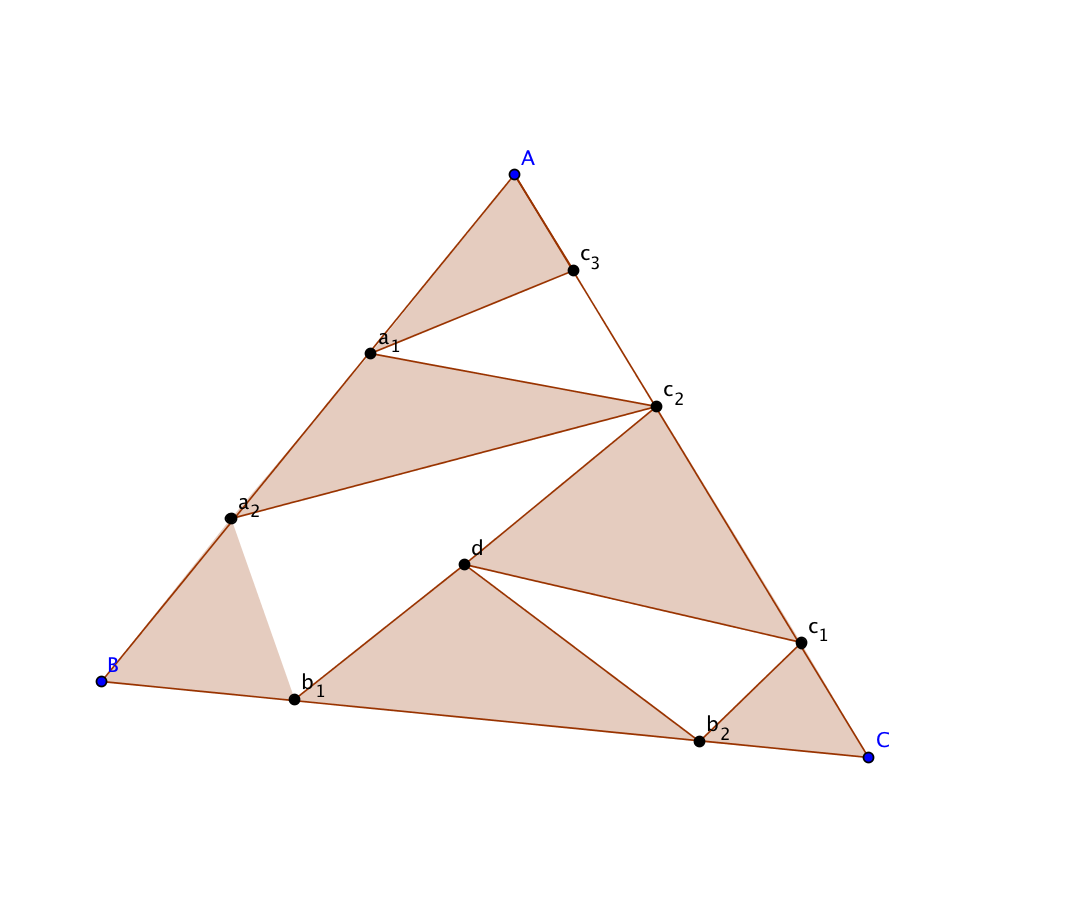}
   \caption{A generator of degree $6$.}
   \label{Figure 2}
\end{figure}

Another point is that the maps $f_x$ generating random Sierpinski
gasket have degree $3$. We will allow the
degree $d_x$ of all maps to be different (see Figure~\ref{Figure 2})
 and only require that the
function $x\mapsto \log (d_x)$ is measurable.

Finally, the above examples are all expanding with an expanding
constant $\gamma_x \geq \gamma >1$.  As already explained in the
introduction, the present manuscript concerns random maps for which
the expanding constants $\gamma_x$ can be arbitrarily close to
one. Furthermore, using an inducing procedure, we will even weaken this
to the maps that are only expanding in the mean (see
Chapter~\ref{sec:expanding-mean}).

The example of random Sierpi\'nski gasket is not conformal. Random
iterations of rational functions or of holomorphic repellers are
typical examples of conformal random dynamical systems.  Random
iterations of the quadratic family $f_c(z)=z^2+c$ have been considered,
for example, by Br\"uck and B\"uger among others (see \cite{Bru01} and \cite{BruBug03}).
%Fornaes-Sibony, Sumi, Sester, Rugh ... ????  {\bf add Fornaes-Sibony,
%  Sumi, Sester 1,2, others ?}

In this case, one chooses randomly a sequence of parameters $c=(c_1,c_2,...)$ and considers 
the dynamics of the family 
$$F_{c_1,...,c_n}=  f_{c_n}\circ f_{c_{n_1}}\circ ... \circ f_{c_1}\quad , \quad n\geq 1\;.$$
This leads to the dynamical invariant sets
$${\mathcal K}_c =\{ z\in \C \; ; \;\; F_{c_1,...,c_n} (z) \to \infty \} \quad \text{and}
\quad {\mathcal J}_c = \partial {\mathcal K}_c \;.$$
The set ${\mathcal K}_c$ is the filled in Julia set and ${\mathcal J}_c$ the Julia set 
associated to the sequence $c$.

The simplest case is certainly the one when we consider just two
polynomials $z\mapsto z^2+\lam _1$ and $z\mapsto z^2+\lam _2$ and we
build a random sequence out of them. Julia sets that come out of
such a choice are presented in Figure~\ref{Figure 4}.  Such random
Julia sets are differnet objects as compared to the Julia sets for
deterministic iteration of quadratic polynomials.  But not only the
pictures are different and intriguing, we will see in Chapter~\ref{ch:section 5}
that also generically the fractal properties of such Julia sets are
fairly different as compared with the deterministic case even if the
dynamics is uniformly expanding. In
Chapter~\ref{cha:class-expand-rand} we present more general class of
examples and we explain their dynamical and fractal features.

\begin{figure}[here]
\centering
   \includegraphics[scale=0.3]{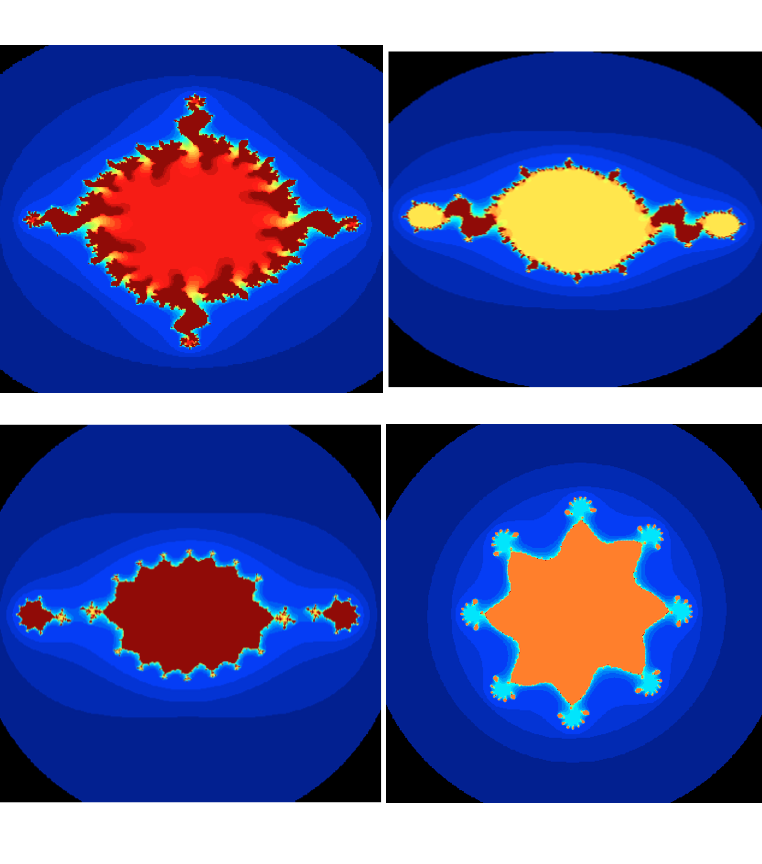}
   \caption{Some more quadratic random Julia sets}
   \label{Figure 4}
\end{figure}

\section{Preliminaries}
\label{sec:preliminaries}
Suppose $(X,\mathcal{F},m, \shift )$ is a measure preserving dynamical
system with invertible and ergodic map $\shift :X\to X$ which is
referred to as \index{base map}\emph{the base map}.  Assume further
that $(\J_x,\varrho_x)$, $x\in X$,
are compact metric spaces normalized in size by
$\diam_{\varrho_x}(\J_x)\leq 1$.
Let
\begin{equation}
  \label{2} \cJ= \bigcup _{x\in X} \{x\}\times \cJ_x\; .
\end{equation}
We will denote by $B_x(z,r)$ the ball in the space $(\J_x,\varrho_x )$
centered at $z\in\J_x$ and with radius $r$. Frequently, for ease
of notation, we will write $B(y,r)$ for $B_x(z,r)$, where $y=(x,z)$.
Let
$$
T_x:\cJ_x\to \cJ_{\shift(x)}\quad \;\; , \; x\in X\, ,
$$
be continuous mappings and let $T:\cJ\to\cJ$ be the associated skew-product defined by
 \beq\label{1} T(x,z)=(\shift(x),T_x(z)).  \eeq
For every
$n\ge 0$ we denote
$
T^n_x:=T_{\shift ^{n-1}(x )} \circ ...\circ T_x:\cJ_x\to \cJ_{\shift^n(x)}
$. With this notation one has $T^n(x,y)=(\shift^n(x),T^{n}_x(y))$. We will
frequently use the notation
\begin{displaymath}
  x_n=\shift^n(x) \quad , \;\; n\in \Z\, .
\end{displaymath}
If it does not lead to misunderstanding we will identify
$\cJ_x$ and $\{x\}\times \cJ_x$.

\section{Expanding Random Maps}\label{exp rds def}
A map $T:\J\to\J$ is called a \emph{expanding random
  map}\index{expanding random map} if the mappings $T_x:\cJ_x\to
\cJ_{\shift (x)}$ are continuous, open, and surjective, and if there
exist a function $\eta:X\to\R_+$, $x\mapsto \eta_x$,
and a real number $\xi>0$ such that following conditions hold.\\

\fr \emph{Uniform Openness.} \index{uniform openness}
 $T_x(B_x(z,\eta_x))\supset B_{\shift (x)}\big(T_x(z),\xi\big)$ for every $(x,z)\in \cJ$.

 \

\fr \emph{Measurably Expanding.} \index{measurably expanding}
There exists a measurable function $\gamma:X\to (1,+\infty)$,
$x\mapsto \gamma_x$ such that
   $\varrho_{\shift(x)}(T_x(z_1),T_x(z_2))\geq \gamma_x
      \varrho_x(z_1,z_2)$ whenever $ \varrho(z_1,z_2)
<\eta_x, \,\, z_1,z_2 \in \cJ_x $ holds $m$-a.e.

\

\fr\emph{Measurability of the Degree.} \index{measurability!of
  the degree} The map $
  x\mapsto\deg(T_x):=\sup_{y\in \cJ_{\shift(x)}}\#\, T_{x}^{-1}(\{y\})$ is measurable.

\

\fr\emph{Topological Exactness.} \index{topological exactness} There
exists a measurable function $x\mapsto n_\xi(x)$ such that
\begin{equation}
  \label{it:b''} T_x^{n_\xi(x)}(B_x(z,\xi))=
  \cJ_{\shift^{n_\xi(x)}(x)}\quad \text{ for every}\;\; z\in \cJ_x \;\; \text{and a.e.} \;\;  {x\in X}\;.
\end{equation}

Note that the measurably expanding condition implies that
$T_x|_{B(z,\eta_x)}$ is injective for every $(x,z)\in \cJ$.
Together with the compactness of the spaces $\cJ_x$ it yields the
numbers $\deg(T_x)$ to be finite. Therefore the supremum in the
condition of measurability of the degree is in fact a maximum.

In this work we consider two other classes of random maps.  The first
one consists of the \emph{uniform expanding} maps defined below. These
are expanding random maps with uniform control of measurable
``constants''.  The other class we consider is composed of maps that
are only \emph{expanding in the mean}.  These maps are defined like
the expanding random maps above excepted that the uniform openness and
the measurable expanding conditions are replaced by the following (see
Chapter~\ref{sec:expanding-mean} for detailed definition).
\begin{enumerate}
\item There exists all local inverse branches.
\item The function $\ga$ in the measurable expanding condition is allowed to 
have values in $(0,\infty )$ but subjects only the the condition that 
$$\int_X \log \, \ga _x \; dm >0 \;.$$
\end{enumerate}
We employ an inducing procedure to expanding in the mean random maps
in order to reduce then the case of random expanding maps. This is the 
content of Chapter ~\ref{sec:expanding-mean} and the conclusion is that
all the results of the present work valid for expanding random maps do also
hold for expanding in the mean random maps.

%********************************************************************* Uniform RDS
\section{Uniformly Expanding Random Maps}
\label{ch:CUERM}

Most of this paper and, in particular, the whole thermodynamical
formalism is devoted to measurable expanding systems.  The study of
fractal and geometric properties (which starts with
Chapter~\ref{ch:section 5}), somewhat against our general philosophy,
but with agreement with the existing tradition (see for example
\cite{BogGun95}, \cite{Kif92} and \cite{DenGor99}), we will work
mostly with \emph{uniform} and \emph{conformal} systems (the later are
introduced in Chapter~\ref{ch:section 5}).

%\bdfn\label{def uniform systems}
A expanding random map $T:\J\to\J$ is called
\index{uniformly expanding random map} \emph{uniformly
  expanding} if
\begin{itemize}
\item[\bf - ] $\gamma_*:=\inf_{x\in X}\gamma_x>1$,
\item[\bf - ] $\deg(T):=\sup_{x\in X}\deg(T_x)<\infty$,
\item[\bf - ] $n_{\xi*}:=\sup_{x\in X}n_{\xi}(x)<\infty$.
\end{itemize}
%\edfn

%****************************************************************** End Uniform RDS

\section{Remarks on Expanding Random Mappings}
\label{sec:GR}
The conditions of uniform openness and measurably expanding imply
that, for every $y=(x,z)\in \cJ$ there exists a unique continuous
inverse branch 
\begin{displaymath}
  T^{-1}_y:B_{\shift(x)}(T(y),\xi)\to B_x(y,\eta_x)
\end{displaymath}
 of
$T_x$ sending $T_x(z)$ to $z$.  By the measurably expanding property
we have
\begin{equation}
  \label{eq:gamma}
  \varrho(T_y^{-1}(z_1),T_y^{-1}(z_2))\leq
  \gamma_x^{-1} \varrho (z_1,z_2) \quad
  \text{for} \quad z_1,z_2\in B_{\shift(x)}\big(T(y),\xi\big)
\end{equation}
 and
 \begin{displaymath}
  T_y^{-1}(B_{\shift(x)}(T(y),\xi))\subset
  B_x(y,\gamma_x^{-1}\xi)\subset B_x(y,\xi). 
 \end{displaymath}
Hence, for every $n\geq 0$, the composition
\begin{equation}
  \label{eq:rds11l2}
  T^{-n}_y=T^{-1}_y\circ T^{-1}_{T(y)}\circ \ldots \circ T^{-1}_{T^{n-1}(y)}:
  B_{\shift^n(x)}(T^n(y),\xi)\to B_x(y,\xi)
\end{equation}
is well-defined and has the following properties:
\begin{displaymath}
  T^{-n}_y:B_{\shift^n(x)}(T^n(y),\xi)\to B_x(y,\xi)
\end{displaymath}
is continuous,
\begin{displaymath}
  T^n\circ T_y^{-n}=\Id |_{B_{\shift^n(x)}(T^n(y),\xi)}
\textrm{, }\,\,\,
  T_y^{-n}(T_x^n(z))=z
\end{displaymath}
  and, for every $z_1,z_2 \in B_{\shift^n(x)}\big(T^n(y),\xi\big)$,
\begin{equation}
  \label{eq:-gamma}
  \varrho (T^{-n}_y(z_1),T^{-n}_y(z_2))\leq
  (\gamma^n_x)^{-1}\varrho (z_1 ,z_2 ),
\end{equation}
 where
$
  \gamma_x^n=\gamma_x\gamma_{\shift (x) }\cdots \gamma_{\shift
    ^{n-1}(x)}.
$
 Moreover,
\begin{equation}
  \label{it:-TBB}
  T_x^{-n}(B_{\shift^n(x)}(T^n(y),\xi))\subset
  B_x(y,(\gamma_x^n)^{-1}\xi)\subset B_x(y,\xi).
\end{equation}

\begin{lem}
  \label{lem:1}
  For every ${r>0}$, there exists a measurable function $x\mapsto
  n_r(x)$ such that a.e.
  \beq \label{it:b}
  T_x^{n_r(x)}(B_x(z,r))= \cJ_{\shift^{n_r(x)}(x)}\quad \text{ for
    every}\;\; z\in \cJ_x \,.
 \eeq
  Moreover, there exists a measurable
  function $j: X\to \N$ such that a.e. we have
   \beq\label{it:b'} T_{x_{-j(x)}}^{j(x)}(B_{x_{-j(x)}}(z,\xi
  ))= \cJ_{x}\quad \text{ for every}\;\; z\in
  \cJ_{x_{-j(x)}}.\eeq
\end{lem}
\begin{proof}

  In order to prove the first statement, consider $\gamma_0>1$ and 
let $F$ be the set of $x\in X$ such that $\gamma _x \geq \gamma_0$.
If $\gamma_0$ is sufficiently close to $1$, then $m(F)>0$.
In the following section such a set will be called essential.
In that section we also associate to such an essential set
a set $X_{+F}'$ (see \eqref{return time}).
Then for
  $x\in X_{+F}'$, the limit
  $
    \lim_{n\to\infty}(\gamma_x^n)^{-1}=0
  $.
  Define
  \begin{displaymath}
    X_{+F,k}:=\{x\in X_{+F}':(\gamma_x^k)^{-1}\xi < r\}.
  \end{displaymath}
  Then $X_{+F,k}\subset X_{+F,k+1}$ and $\bigcup_{k\in\N}
  X_{+F,k}=X_{+F}'$ By measurability of $x\mapsto \gamma_x$,
  $X_{+F,k}$ is a measurable set. Hence the function
  \begin{displaymath}
    X'_{+F}\ni x\mapsto n_r(x):=\min\{k:x\in X_{+F,k}\}+n_{\xi}(x)
  \end{displaymath}
  is finite and measurable. By \eqref{it:-TBB} and \eqref{it:b''},
  \begin{displaymath}
    T_x^{n_r(x)}(B_x(z,r))= \cJ_{\shift^{n_r(x)}(x)}.
  \end{displaymath}

  In order to prove the existence of a measurable function $j:X\to \N$ define
  measurable sets 
  \begin{displaymath}
    X_{\xi,n}:=\{x\in X:n_\xi(x)\leq n\}\textrm{, }
    X_{\xi,n}':=\shift^n(X_{\xi,n})\textrm{ and }
    X_{\xi}'=\bigcup_{n\in\N}X_{\xi,n}'.
  \end{displaymath}
   Then the map
  \begin{displaymath}
    X_{\xi}'\ni
    x\mapsto j(x):=\min\{n\in\N:x\in X_{\xi,n}'\}
  \end{displaymath}
  satisfies (\ref{it:b'}) for $x\in X_{\xi}'$. Since
  $m(\shift^n(X_{\xi,n}))=m(X_{\xi,n})\nearrow 1$ as $n$ tends to
  $\infty$ we have $m(X_{\xi}')=1$.
   
\end{proof}

% \begin{rem}
%   \label{rem:IntGamma}
%   One could consider an allegedly weaker expanding condition that
%   $$\int \log \gamma_x dm(x)> 0.$$ However, this case can be reduced to
%   our settings by looking at an appropriate induced map. Indeed, it
%   follows from Birkhoff's Ergodic Theorem that there exist $n_0\geq 1$
%   and $X_0\subset X$ with $m(X_0)>0$ such that $\gamma_x^n\geq 2$ for
%   all $x\in X_0$ and all $n\geq n_0$. Now observe that $X_1$, the set
%   of all points in $X_0$ whose first return time to $X_1$ is larger
%   than $n_0$ has positive measure. Replacing $\shift:X\to X$ by
%   $\shift_{X_1}:X_1\to X_1$, the first return map to $X_1$, yields
%   $\hat{\gamma}_x\geq 2$ for all $x\in X_1$. Here $\hat{\gamma}_x$ is
%   the expanding factor for the map $\shift_{X_1}$.
% \end{rem}

\section{Visiting sequences}\index{visiting sequence}
Let $F\in \cF$ be a set with a positive measure.  Define the sets
\begin{displaymath}
  V_{+F}(x):=\{n\in\N:\shift^n(x)\in F\}
\quad \text{and}\quad
  V_{-F}(x):=\{n\in\N:\shift^{-n}(x)\in F\}.
\end{displaymath}
The set $V_{+F}(x)$ is called \emph{\index{visiting sequence}visiting
  sequence} (\emph{of $F$ at $x$}). Then the set $V_{-F}(x)$ is just a
visiting sequence for $\shift^{-1}$ and we also call it
\emph{\index{backward visiting sequence}backward visiting
  sequence}. By $n_{j}(x)$ we denote the $j$th-visit in $F$ at $x$.
Since $m(F)>0$, by Birkhoff's Ergodic Theorem we have that
\begin{displaymath}
  m(X_{+F}')=m(X_{-F}')=1
\end{displaymath}
where
\beq \label{return time}
  X_{+F}':=\Big\{x\in X: V_{+F}(x)\textrm{ is infinite and }
  \lim_{j\to \infty}\frac{n_{j+1}(x)}{n_j(x)}=1\Big\}
\eeq
and $X_{-F}'$ is defined analogously. The sets $X_{+F}'$ and $X_{-F}'$
are respectively called \emph{forward} and \emph{backward visiting for
  $F$}.

Let $\Psi(x,n)$ be a formula which depends on $x\in X$ and $n\in
\N$. We say that $\Psi(x,n)$ holds \emph{\index{visiting way}in a
  visiting way}, if there exists $F$ with $m(F)>0$ such that, for
$m$-a.e. $x\in X_{+F}'$ and all $j\in \N$, the formula
$\Psi(\shift^{n_j}(x),n_j(x))$ holds, where $(n_j(x))_{j=0}^{\infty}$
is the visiting sequence of $F$ at $x$.
We also say that $\Psi(x,n)$ holds \emph{\index{exhaustively
    visiting way}in a exhaustively visiting way}, if there exists a
family $F_k\in \cF$ with $\lim_{k\to \infty}m(F_k)=1$ such that, for
all $k$, $m$-a.e. $x\in X_{+F_k}'$, and all $j\in \N$, the formula
$\Psi(\shift^{n_j}(x),n_j(x))$ holds, where $(n_j(x))_{j=0}^{\infty}$
is the visiting sequence of $F_k$ at $x$.

Now, let $f_n:X\to \R$ be a sequence of measurable functions. We write
that
\begin{displaymath}
  \slim_{n\to\infty}f_n=f
\end{displaymath}
if that there exists a family $F_k\in \cF$ with $\lim_{k\to
  \infty}m(F_n)=1$ such that, for all $k$ and $m$-a.e. $x\in
X_{+F_k}'$ and all $j\in \N$,
\begin{displaymath}
  \lim_{j\to\infty}f_{n_j}(x)=f(x)
\end{displaymath}
where $(n_j)_{j=0}^{\infty}$ is the visiting sequence of $F_k$ at $x$.

Suppose that $g_1,\ldots,g_k:X\to \R$ is a finite collection of
measurable functions and let $b_1,\ldots, b_n$ be a collection of real
numbers.  Consider the set
\begin{displaymath}
  F:=\bigcap_{i=1}^kg_i^{-1}((-\infty,b_i]).
\end{displaymath}
If $m(F)>0$, then $F$ is called \emph{\index{essential}essential} for
$g_1,\ldots,g_k$ with constants $b_1,\ldots, b_n$ (or just
\emph{essential}, if we do not want explicitly specify functions and
numbers).  Note that by measurability of the functions
$g_1,\ldots,g_k$, for every $\varepsilon>0$ we can always find finite
numbers $b_1,\ldots, b_n$ such that the essential set $F$
for $g_1,\ldots,g_k$ with constants $b_1,\ldots, b_n$ has the measure
$m(F)\geq 1-\varepsilon$.

\section{Spaces of Continuous and H\"older Functions}
\label{sec:spacesHol}

We denote by $\cC(\J_x)$ the space of continuous functions
$g_x:\J_x\to \R$ and by $\cC(\J)$ the space of functions $g:\J\to \R$
such that, for a.e. $x\in X$, $x\mapsto g_x:={g} _{|\cJ_x}\in
\cC(\J_x)$.  The set $\cC(\J)$ contains the subspaces $\cC^0(\J)$ of
functions for which the function $x\mapsto \|g_x\|_\infty$ is
measurable, and $\cC^1(\J)$ for which the integral
$$ \|g\|_1:= \int _X \|g_x\|_\infty \, dm(x) <\infty .$$

Now, fix $\alpha\in (0,1]$. By $\holderx$ we denote the space of
H\"older continuous functions on $\cJ_x$ with an exponent
$\alpha$. This means that $\varphi_x\in \cH^\alpha(\cJ_x)$ if and only
if $\varphi_x\in \cC(\cJ_x)$ and $v(\varphi_x)<\infty$ where
\begin{equation}
  \label{eq:4}
  v_\alpha(\varphi_x):=\inf\{H_x:
    |\varphi(z_1)-\varphi(z_2)|\leq
  H_x\varrho^\alpha_x(z_1,z_2) \}.  
\end{equation}
where the infimum is taken over all $z_1,z_2\in \cJ_x $ with
$\varrho (z_1,z_2)\leq \eta$.

A function $\varphi\in \cC^1(\cJ )$ is called \emph{\index{H\"older
  continuous with an exponent $\alpha$}H\"older
  continuous with an exponent $\alpha$} provided that there exists a
measurable function $H: X\to [1,+\infty)$, $x\mapsto H_x$, such that
$\log H\in L^1(m)$ and such that $v_\alpha(\varphi_x)\leq H_x$ for
a.e. $x\in X$. We denote the space of all H\"older functions with
fixed $\al$ and $H$ by $\cH^\alpha(\J,H)$ and the space of all
$\al$--H\"older functions by $\cH^\alpha(\J) = \bigcup _{H\geq 1}
\cH^\alpha(\J,H)$.

\section{Transfer operator}\label{transfer}
For every function $g:\J\to\C$ and a.e. $x\in X$ let
\beq\label{3.3}
  S_ng_x=\sum_{j=0}^{n-1}g_x\circ T^j_x,
\eeq
 and, if $g:X\to \C$, then
$S_ng=\sum_{j=0}^{n-1}g\circ\shift^j$.
Let $\vp$ be a function in the H\"older space $\cH^\alpha(\J)$. For
every $x\in X$, we consider the \emph{transfer
  operator}\index{transfer!operator}
$\cL_x=\cL_{\vp,x}:\cC(\J_x) \to \cC(\J_{\shift(x)}) $
given by the formula \beq\label{3.1} \cL_xg_x (w) = \sum_{T_x(z)=w}
g_x(z) e^{\vp _x(z)}, \;\;w\in \cJ_{\shift (x)}.  \eeq It is obviously
a positive linear operator and it is bounded with the norm bounded
above by \beq
\label{3.4}\|\cL_x\|_\infty \leq \deg(T_x) \exp(\|\vp \|_\infty).\eeq
This family of operators gives rise to the global operator
$\cL:\cC(\J)\to\cC(J)$ defined as follows:
\begin{displaymath}
  \left(\cL g \right)_{x}(w)=\cL_{\shift^{-1}(x)}g_{\shift^{-1}(x)} (w).
\end{displaymath}
For every $n>1$ and a.e. $x\in X$, we denote
\begin{displaymath}
  \cL^n_x := \cL_{\shift^{n-1}(x)} \circ ... \circ \cL_x : \cC(\cJ _x ) \to \cC (\cJ _{\shift ^n (x)}).
\end{displaymath}
Note that
\begin{equation}\label{3.2}
  \cL_x^n g_x(w)=\sum_{z\in T^{-n}_x(w)}g_x(z)e^{S_n\varphi _x (z)}
  \textrm{ ,  }\, w\in \cJ_{\shift^n(x)}\, ,
\end{equation}
where $ S_n\varphi _x (z) $ has been defined in (\ref{3.3}).
The dual operator $\cL^*_x$ maps $C^*(\cJ_{\shift(x)})$ into
$C^*(\cJ_x)$.

\section{Distortion Properties}

\begin{lem}\label{lem:l1rds13} Let $\varphi\in \cH^\alpha(\J ,H)$, let
  $n\geq 1$ and let $y=(x,z)\in \J$. Then
  \begin{displaymath}
    |S_n\varphi_x(T^{-n}_y(w_1))-S_n\varphi_x(T_y^{-n}(w_2))| \leq
     \varrho^{\alpha}(w_1,w_2)
     \sum_{j=0}^{n-1}H_{\shift^j(x)}(\gamma_{\shift^j(x)}^{n-j})^{-\alpha}
  \end{displaymath}
  for all $w_1,w_2\in B(T_x^n(z),\xi)$.
\end{lem}

\begin{proof}
     
  We have by (\ref{eq:-gamma}) and H\"older continuity of $\vp$ that
  \begin{displaymath}
      \begin{aligned}
    |  S_n\varphi_x  (T^{-n}_y(w_1)) -S_n\varphi_x(T_y^{-n}(w_2))|
     &\leq
    \sum_{j=0}^{n-1}|\varphi_x(T_x^j(T_y^{-n}(w_1)))
    -\varphi_x(T_x^j(T_y^{-n}(w_2)))|\\
    &=\sum_{j=0}^{n-1}\Big|\varphi_x (T_{T_x^j(y)}^{-(n-j)}(w_1))-
    \varphi_x(T_{T_x^j(y)}^{-(n-j)}(w_2))\Big|\\
    &\leq
    \sum_{j=0}^{n-1}\varrho^{\alpha}\big(T_{T_x^j(x)}^{-(n-j)}(w_1),
    T_{T_x^j(x)}^{-(n-j)}(w_2)\big)
    H_{\shift^j(x)},
  \end{aligned}
  \end{displaymath}
  hence $|  S_n\varphi_x  (T^{-n}_y(w_1)) -S_n\varphi_x(T_y^{-n}(w_2))|\leq \varrho^{\alpha}(w_1,w_2)
    \sum_{j=0}^{n-1}H_{\shift^j(x)}(\gamma_{\shift^j(x)}^{n-j})^{-\alpha}$.
 \end{proof}
Set
\begin{equation}
  \label{eq:1.3}
  Q_x:=Q_{x}(H)=\sum_{j=1}^\infty H_{\shift^{-j}(x)}(\gamma_{\shift^{-j}(x)}^{j})^{-\alpha}.
\end{equation}

\begin{lem}
  \label{lem:l1rds13b}  The
  function $x\mapsto Q_x$ is measurable and $m$-a.e. finite. Moreover, for every
  $\varphi\in \cH^\alpha(\cJ ,H)$,
  \begin{equation*}
    |S_n\varphi_x(T^{-n}_y(w_1))-S_n\varphi_x(T_y^{-n}(w_2))|
    \leq Q_{\shift^n(x)}\varrho^{\alpha}(w_1,w_2)
  \end{equation*}
  for all $n\geq 1$, a.e. $x\in X$, every $z\in \J_x$ and $w_1,w_2\in B(T^n(z),\xi)$
  and where again $y=(x,z)$.
\end{lem}

\begin{proof} 
  The measurability of $Q_x$ follows directly form \eqref{eq:1.3}.
  Because of Lemma~\ref{lem:l1rds13} we are only left to show that
  $Q_x$ is $m$-a.e. finite. % Let $F$ be an essential set for $-\log
  % (\gamma-1)$. Then there exists $G>1$ such that $\gamma_x\geq G$ for
  % all $x\in F$.
  Let $\chi$ be a positive real number less or equal to $\int
  \log\gamma_xdm(x)$.  Then, using Birkhoff's Ergodic Theorem for
  $\shift^{-1}$, we get that
  \begin{equation*}
    \liminf_{j\to\infty}\frac{1}{j}\sum_{k=0}^{j-1}
    \log \gamma_{\shift^{-j}(x)}\geq\chi
  \end{equation*}
  for $m$-a.e. $x\in X$. Therefore, there exists a measurable function
  $C_\gamma:X\to [1,+\infty)$ $m$-a.e. finite such that
  $
    C_\gamma^{-1}(x)e^{j\chi/2}\leq
    \gamma_{\shift^{-j+1}(x)}^j
  $
  for all $j\geq 0$ and a.e. $x\in X$. Moreover, since $\log
  H\in L^1(m)$ it follows again from Birkhoff's Ergodic Theorem that
  \begin{equation*}
    \lim_{j\to\infty}\frac{1}{j}\log H_{\shift^{-j}(x)}=0 \qquad m-a.e.
  \end{equation*}
   There thus exists a measurable function
  $C_H:X\to [1,+\infty)$ such that
  \begin{equation}
    H_{\shift^j(x)}\leq C_H(x)e^{j\alpha\chi/4}\quad\textrm{ and }\quad
    H_{\shift^{-j}(x)}\leq C_H(x)e^{j\alpha\chi/4}
  \end{equation}
  for all $j\geq 0$ and a.e. $x\in X$. Then, for a.e. $x\in
  X$, all $n\geq 0$ and all $a\geq j\geq n-1$, we have
  \begin{equation*}
    H_{\shift^j(x)}=H_{\shift^{-(n-j)}(\shift^n(x))}\leq C_H(\shift^{n}(x))e^{(n-j)\alpha\chi/4}.
  \end{equation*}
  Therefore, still with $x_n =\shift ^n (x)$,
  \begin{multline*}
    Q_{x_n}=\sum_{j=0}^{n-1}H_{x_j}(\gamma_{x_j}^{n-j})^{-\alpha}
    \leq \sum_{j=0}^{n-1}C_H(x_n)e^{(n-j)\alpha\chi/4}
    C_\gamma^\alpha(x_{n-1})e^{-\alpha(n-j)\chi/2} \\
    \leq C_\gamma^\alpha(x_{n-1})
    C_H(x_n)\sum_{j=0}^{n-1}e^{-\alpha(n-j)\chi/4}
    \leq C_\gamma^\alpha(x_{n-1})
    C_H(x_n)(1-e^{-\alpha\chi/4})^{-1}.
  \end{multline*}
  Hence
  \begin{displaymath}
    Q_{x}\leq
    C_\gamma^\alpha(\shift^{-1}(x))
    C_H(x)(1-e^{-\alpha\chi/4})^{-1}<+\infty.
  \end{displaymath}
 \end{proof}

%%% Local Variables:
%%% mode: latex
%%% TeX-master: "RDSmain"
%%% End:

%% file: RDSsection3.tex
\chapter{The \index{RPF--theorem}RPF--theorem}

\label{ch:RPFmain}

We now establish a version of Ruelle-Perron-Frobenius
(RPF) Theorem along with a mixing property.
Notice that this quite substantial fact is proved
without any measurable structure on the space $\J$. In particular, we
do not address measurability issues of $\lambda_x$ and $q_x$. In order
to obtain this measurability we will need and we will impose a natural
measurable structure on the space $\J$. This will be done in the next section.

\section{Formulation of the Theorems}
\label{sec:formulation-theorems}

Let $T:\J\to\J$ be a expanding random map.  Denote by $\cM^1(\J_x)$ the
set of all Borel probability measures on $\J_x$.  A family of measures
$\{\mu_x\}_{x\in X}$ such that $\mu_x\in\cM^1(\J_x)$ is called
$T$--\emph{invariant}\index{T-invariance!of a family of measures} if
$\mu_x\circ T_x^{-1}=\mu_{\shift(x)}$ for a.e. $x\in X$.

The main results proved in this section are listed below.
\begin{thm}
  \label{thm:Gib50A}
  Let $\varphi\in\cH^\alpha(\J)$ and let $\cL=\cL_\vp$ be the
  associated transfer operator. Then the following holds.
  \begin{enumerate}
  \item\label{item:1} There exists a unique family of probability measures $\nu_x\in\cM(\J_x)$ such that
    $m$-a.e.
    \begin{equation}
      \label{eq:1}
      \cL^*_x\nu_{\shift(x)}=\lambda_x\nu_x \quad \text{where} \quad \lam_x = \nu_{\shift (x)} (\cL_x \1 )\, .
    \end{equation}
  \item\label{item:2} There exists a unique function $q\in \cC^0(\J)$
    such that $m$-a.e.
    \begin{equation}
      \label{eq:3}
      \cL_xq_x=\lambda_xq_{\shift(x)}\quad \text{and}\quad \,\nu_x(q_x)=1 .
    \end{equation}
    Moreover, $q_x\in\cH^\alpha(\J_x)$ for a.e. $x\in X$.
  \item\label{item:3} The family of measures
    $\{\mu_x:=q_x\nu_x\}_{x\in X}$ is $T$-invariant.
  \end{enumerate}
\end{thm}
\begin{thm}\
  \label{thm:Gib50B}
  \begin{enumerate}
  \item\label{item:4} Let
    \begin{displaymath}
      \hvarphi_x=\varphi_x+\log q_x - \log
    q_{\shift (x)}\circ T -\log\lambda_x.
    \end{displaymath}
   Denote
    $\hcL:=\cL_{\hvarphi }$. Then, for a.e. $x\in X$ and all $g_{x}\in
    C(\J_x)$,
    $$\hcL_{x}^n g_x \xrightarrow[n\to\infty]{}\int g_x q_x d\nu_x.$$
  \item\label{item:5} Let
    $\tilde{\varphi}_x=\varphi_x-\log\lambda_x$. Denote
    $\tcL:=\cL_{\tilde{\varphi} }$.  There exist a constant $B<1$ and
    a measurable function $A:X\to (0,\infty )$ such that for every
    function $g\in \cC^0(\J)$ with $g_x\in \cH^\alpha(\J_x)$ there
    exists a measurable function $A_g:X\to (0,\infty)$ for which
    $$\|( \tcL ^n g)_x -\Big(\int g_{\shift^{-n}(x)}d\nu_{\shift^{-n}(x)}
    \Big)q_x \|_\infty \leq A_g(\shift^{-n}(x))A(x) B^n $$
    for a.e. $x\in X$ and every $n\geq 1$.
  \item\label{item:6} There exists $B<1$ and a measurable function
    $A':X\to (0,\infty )$ such that for every $f_{\shift^n(x)}\in
    L^1(\mu_{\shift^n(x)})$ and every $g_x\in\cH^\alpha(\cJ_x)$,
    \begin{multline*}
      \big|\mu_x\big((f_{\shift ^n(x)}\circ T_x^n)g_x\big)-
      \mu_{\shift^n(x)}(f_{\shift ^n(x)})\mu_x(g_x)\big|\\
      \leq \mu_{\shift^n(x)}(|f_{\shift^n(x)}|)A'(\shift ^n(x))
      \Big( \int |g_x|d\mu_x+4\frac{v_\alpha(g_xq_x)}{Q_x}\Big)B^n.
    \end{multline*}
  \end{enumerate}
\end{thm}

A collection of measures $\{\mu_x\}_{x\in X}$ such that
$\mu_x\in\cM^1(\J_x)$ is called a {\em \index{Gibbs family}Gibbs family}
for $\varphi\in \cH^\alpha(\J)$ provided that there exists a
measurable function $D_\varphi:X\to[1,+\infty)$ and a function
$x\mapsto P_x$, called \emph{the \index{pseudo-pressure
    function}pseudo-pressure function}, such that
\begin{equation}\label{4.5}
  (D_\varphi(x)D_\varphi(\shift^n(x)))^{-1}
  \leq\frac{\mu_x(T^{-n}_y(B(T^n(y),\xi)))}
  {\exp(S_n\varphi(y)-S_n P_x)}\leq
  D_\varphi(x)D_\varphi(\shift^n(x))
\end{equation}
for every $n\geq 0$, a.e. $x\in X$ and every $z\in \cJ_x$ and with
$y=(x,z)$.

Towards proving uniqueness type result for Gibbs families we introduce
the following concept. Notice that in the case of random compact subsets of a Polish space (see
Section~\ref{sec:CRS}) this condition always holds (see Lemma~\ref{lem:Pre13}).

\

\fr\emph{\index{measurability!of cardinality of covers}Measurability of Cardinality of Covers}
There exists a measurable function $X\ni x\mapsto a_x\in\N$ such that
for every $x\in X$ there exists a finite sequence
$w_x^1,\ldots,w_x^{a_x}\in \J_x$ such that
$\bigcup_{j=1}^{a_x} B(w_x^j,\xi)=\J_x.$

\

\begin{thm}
  \label{thm:1}
  The collections $\{\nu_x\}_{x\in X}$ and $\{\mu_x\}_{x\in X}$ are
  Gibbs families. Moreover, if $\J$ satisfies the condition of
  measurability of cardinality of covers and if $\{\nu_x'\}_{x\in X}$
  is a Gibbs family, then $\nu_x'$ and $\nu_x$ are equivalent for
  almost every $x\in X$.
\end{thm}

\section{Frequently used auxiliary measurable functions}
\label{sec:FUC}
Some technical measurable functions appear throughout the paper so
frequently that, for convenience of the reader, we decided to collect
them in this section together. However, the reader may skip this part
now without any harm and come back to it when it is appropriately needed.

First, define
\begin{equation}
  \label{eq:D_xi}
  D_\xi(x):=\big(\deg T_{x}^{n}\big)^{-1}
  \exp({-2\|S_{n}\vp _{x}\|_\infty})
\end{equation}
with $n=n_\xi(x)$ being the index given by the topological exactness
condition (cf. (\ref{it:b''})).  Then, let $j=j(x)$ be the number
given by Lemma~\ref{lem:1} and define
\begin{equation}
  \label{eq:defCv}
  C_\vp (x)
  :=e^{Q_{x_{-j}}}\deg(T_{x_{-j}}^j)\max
  \big\{\exp(2\|S_k\vp _{x_{-k}}\|_\infty):0\leq k\leq j\big\}\geq 1.
\end{equation}

Now let $s>1$. Put
\begin{equation}
  \label{eq:defCmin}
  C_{\min}(x):=e^{-sQ_x}e^{-\|S_{j}\vp _{x_{-j}}\|_\infty}\leq 1
\end{equation}
and
\begin{equation}
  \label{eq:defCmax}
  C_{\max} (x):=e^{sQ_{x}}\deg\big(T_{x}^{n}\big)
  \exp({2\|S_{n}\vp _{x}\|_\infty})
\end{equation}
where $n:=n_\xi(x)$.  Then we define
\begin{equation}
  \label{eq:betax}
  \beta_{x}(s):=\frac{C_{\min}(x)}{C_\varphi(x)}\cdot
  \inf_{r\in(0,\xi]}\frac{1
    -\exp\big(-(s-1)H_{x_{-1}}\gamma_{x_{-1}}^{-\alpha}r^\alpha\big)}
  {1-\exp(-2sQ_{x}r^\alpha)}.
\end{equation}
Since by (\ref{eq:1.3})
\begin{equation}
  \label{eq:2Q}
  sQ_x=sQ_{x_{-1}}\gamma_{x_{-1}}^{-\alpha}+sH_{x_{-1}}\gamma_{x_{-1}}^{-\alpha},
\end{equation}
\begin{equation}
  \label{eq:3Q}
  (sQ_{x_{-1}}+H_{x_{-1}})\gamma_{x_{-1}}^{-\alpha}=
  sQ_x-(s-1)H_{x_{-1}}\gamma_{x_{-1}}^{-\alpha}.
\end{equation}
This, together with (\ref{eq:defCv}) and (\ref{eq:defCmin}), gives us
that
\begin{displaymath}
  0<\beta_{x}(s)=\frac{C_{\min}(x)}{C_\varphi(x)}
  \frac{(s-1)H_{x_{-1}}\gamma_{x_{-1}}^{-\alpha}}
  {2sQ_{x}}
  <\frac{C_{\min}(x)}{C_\varphi(x)}\leq 1.
\end{displaymath}

\section{\index{transfer!dual operators}Transfer Dual Operators}

In order to prove Theorem \ref{thm:Gib50A} we fix a point $x_0\in X$ and, as the
first step, we reduce the base space $X$ to the orbit
$$\cO _{x_0}=\{\shift ^n(x_0), n\in \Z\}.$$
The motivation for this is that then we can deal with a sequentially
topological compact space on which the transfer (or related) operators
act continuously. Our conformal measure then can be produced, for
example, by the methods of the fixed point theory, similarly as in the
deterministic case.

Removing a set of measure zero, if necessary, we may assume that this
orbit is chosen so that all the involved measurable functions are
defined and finite on the points of $\orb$. For every $x\in \orb$, let
$\ph_x \in \cC(\cJ_x )$ be the continuous potential of the transfer
operator $\cL_x:C(\cJ_x )\to C(\cJ_{\shift (x)})$ which has been
defined in (\ref{3.1}).

\bprop\label{4.3}
There exists probability measures $\nu _x \in M(\cJ_x )$ such that
$$\cL_x^* \nu_{\shift (x)} =\lam_x \nu_x \quad for \; every \; x\in \orb,$$
where
\begin{equation}
  \label{eq:lambdaform}
  \lambda_x:=\cL_x^*(\nu_{\shift (x)})(\1 )=\nu_{\shift (x)} ( \cL_x \1 ).
\end{equation}
\eprop

\begin{proof} 
  Let $\cC^*(\cJ _x )$ be the dual space of $\cC(\cJ_x)$ equipped with
  the weak$^*$ topology.  Consider the product space
  \begin{displaymath}
    \cD (\orb ):=\prod _{x\in \orb } \cC^*(\cJ _x )
  \end{displaymath}
  with the product topology. This is a locally convex topological space and the set
  \begin{displaymath}
    \cP (\orb ) :=\prod _{x\in \orb } \cM^1  (\cJ _x )
  \end{displaymath}
  is a compact subset of $\cD (\orb )$. A simple observation is that the map 
  \begin{displaymath}
    \Psi_x :
    \cM^1 (\cJ_{\shift (x)}) \to \cM^1 (\cJ_{x})
  \end{displaymath}
  defined by
  $$\Psi_x (\nu_{\shift (x)}) =\frac{\cL_x^* \nu_{\shift (x)}}{\cL_x^*\nu_{\shift (x)} (\1 )}$$
  is weakly continuous.  Consider then the global map $\Psi:\cP (\orb
  ) \to \cP (\orb )$ given by
  \begin{displaymath}
    \nu =(\nu_x)_{x\in \orb} \longmapsto \Psi ( \nu  ) =
    \left(\Psi _x \nu_{\shift (x)} \right)_{x\in \orb}.
  \end{displaymath}
  Weak continuity of the $\Psi_x$ implies continuity of $\Psi$ with
  respect to the coordinate convergence.  Since the space $\cP (\orb
  )$ is a compact subset of a locally convex topological space, we can
  apply the Schauder-Tychonoff fixed point theorem to get $\nu\in \cP
  (\orb )$ fixed point of $\Psi$, i.e.
  $$\cL_x^* \nu_{\shift (x)} = \lam _x \nu _x \;\;\; where \; \lam_x =\cL_x^*\nu_{\shift (x)} (\1 ) =
  \nu_{\shift (x)} (\cL_x (\1 ))$$ for every $x\in \orb $.
 \end{proof}

\brem \label{4.4}
The relation \eqref{eq:lambdaform} implies
\begin{equation}
  \label{eq:9}
  \inf_{y\in \J_x}e^{\varphi_x(y)}\leq \lambda_x\leq \|\cL_x\1\|_\infty.
\end{equation}
\erem

A straightforward adaptation of the proof of Proposition~2.2 in
\cite{DenUrb91a} leads to the following, to
Proposition~\ref{4.3} equivalent, characterization of Gibbs states:
if ${T^n_x}_{|A}$ is injective, then
  \begin{equation}\label{conformality}
    \nu_{\shift^n(x)}(T^n_x(A))
    =\lambda_x^{n} \int_{A}e^{-S_n\varphi}d\nu_x.
  \end{equation}

%%%
Here is one more useful estimate.

\begin{lem}
  \label{lem:l1rds18}
  For every $x\in \orb$ and $n\geq 1$,
  \begin{equation}\label{eq 3.1}
    \inf_{z\in \cJ_x}\exp\big(S_n\vp _x (z)\big)
    \leq\frac{\lambda_x^n}{\deg (T_x^n)}
    \leq \sup_{z\in \cJ_x}\exp\big(S_n\vp  _x ( z)\big).
  \end{equation}
  Moreover, for every $z\in \cJ_x$ and every $r>0$,
  \begin{equation}
    \label{eq:Con120}
    \nu_x(B(z,r))
    \geq  D(x,r),
  \end{equation}
  where
  \begin{equation}
    \label{eq:D147}
    D(x,r):=
    \left( \deg(T_x^N)\right)^{-1}\inf_{z\in \J_x}
    \exp\Big(\;{\inf_{a\in B(z,r)}S_{N}\varphi _x (a)}-
    {\sup_{b\in B(z,r)}S_{N}\varphi _x (b)}\Big)
  \end{equation}
  with $N=n_r(x)$ being the index given by Lemma~\ref{lem:1}.  It follows that the set $\cJ_x$
  is a topological support of $\nu_x$. In particular, with $D_\xi(x)$ defined in \eqref{eq:D_xi},
  \begin{equation}
    \label{eq:D_xi2}
    \nu_x(B(z,\xi))\geq D_\xi(x).
  \end{equation}
\end{lem}

\begin{proof} 
  The inequalities \eqref{eq 3.1} immediately follow from
  $$
    \nu_{\shift^n(x)}(\cL_x^n\1)=
    ((\cL^n_x)^*\nu_{\shift^n(x)})(\1)=\lambda_x^n\nu_x(\1)=
    \lambda_x^n .
  $$
  Now fix an arbitrary $z\in \cJ_x$ and $r>0$. Put $n=n_r(x)$ (see
  Lemma~\ref{lem:1}). Then, by (\ref{conformality}),
  \begin{equation*}
    \nu_x(B(z,r))\lambda_x^{n}\sup_{a\in B(z,r)}e^{-S_n\varphi _x (a)}\geq
    \lambda_x^{n}\int_{B(z,r)}e^{-S_n\varphi_x}d\nu_x\geq 1
  \end{equation*}
  which implies \eqref{eq:Con120}.
  \end{proof}

%%%%%%%%%%%%%%%%%%%%%%

\section{Invariant density} Consider now the normalized operator $\tcL$ given by
\beq\label{4.6}\tcL_x = \lam _ x ^{-1} \cL_x, \;\; x\in X.\eeq

\bprop\label{4.7} For every $x\in \orb$, there exists a function $q_x\in \holderx$
such that
$$\tcL_x q_x =q_{\shift (x)}\quad\textrm{ and }\quad\int_{\J_x}q_x d\nu_x=1.$$
In addition, 
\begin{displaymath}
  q_x (z_1)\leq \exp\{Q_x \vr ^\al (z_1,z_2)\} q_x(z_2)
\end{displaymath}
for all
$z_1,z_2\in \J_x$ with $\vr (z_1,z_2)\leq \xi$, and
\beq\label{4.18}
 1/C_\vp (x)\leq q_x\leq C_\vp (x),
 \eeq
where $C_\vp$ was defined in \eqref{eq:defCv}.
\eprop

In order to prove this statement we first need a good uniform distortion estimate.

\begin{lem}
  \label{lem:Cvarphi}
  For all $w_1,w_2\in \J_x$ and $n\geq 1$
  \begin{equation}
    \label{eq:preq0}
    \frac{\tcL_{x_{-n}}^n\1(w_1)}{\tcL_{x_{-n}}^n\1(w_2)}=
     \frac{\cL_{x_{-n}}^n\1(w_1)}{\cL_{x_{-n}}^n\1(w_2)}\leq
    C_\vp (x),
  \end{equation}
  where $C_\vp$ is given by (\ref{eq:defCv}). If in addition $\vr
  (w_1,w_2)\leq \xi$, then
  \beq\label{4.8}\frac{\tcL_{x_{-n}}^n\1(w_1)}{\tcL_{x_{-n}}^n\1(w_2)}\leq
  \exp\{Q_x \vr ^\al (w_1,w_2)\}.\eeq Moreover,
  \begin{equation}
    \label{eq:preq}
    1/C_\vp (x)\leq \tcL^n_{x_{-n}}\1 (w)\leq C_\vp (x) \quad\textrm{
      for every } \; w\in \J_{x} \; and \; n\geq 1.
  \end{equation}
\end{lem}
\begin{proof} 
  First, (\ref{4.8}) immediately follows from Lemma
  \ref{lem:l1rds13b}. Notice also that \beq\label{4.9} \exp\big(Q_x
  \vr ^\al (w_1,w_2)\big)\leq \exp Q_x \eeq since $\diam (\J_x) \leq 1$.
  The global version of (\ref{eq:preq0}) can be proved as follows. If
  $n=0,\ldots, j(x)$, then for every $w_1,w_2\in \J_x$,
  \begin{equation*}
    \cL_{x_{-n}}^n\1 (w_1)\leq
    \frac{\deg(T_{x_{-n}}^n)\exp(\|S_n\vp _{x_{-n}}\|_\infty)}
   {\exp(-\|S_n\vp _{x_{-n}}\|_\infty)}\cL_{x_{-n}}^n\1 (w_2)
    \leq C_\vp (x)\cL_{x_{-n}}^n\1 (w_2).
  \end{equation*}
  Next, let $n> j:=j(x)$. Take $w'_1\in T_{x_{-j}}^{-j}(w_1)$ such that
  \begin{equation*}
    e^{S_j\varphi(w'_1)}\cL^{n-j}_{x_{-n}}\1 (w'_1)
    =\sup_{y\in T_{x_{-j}}^{-j}(w_1)}\Big(e^{S_j\varphi(y)}
    \cL_{x_{-n}}^{n-j}\1 (y)\Big)
  \end{equation*}
  and $w'_2\in T_{x_{-j}}^{-j}(w_2)$ such that
  $\varrho_{x_{-j}}(w'_1,w'_2)\leq \xi$.  Then, by (\ref{4.8}) and (\ref{4.9}),
  \begin{align*}
    \cL_{x_{-n}}^n\1(w_1)
    & = \cL^{j}_{x_{-j}}(\cL^{n-j}_{x_{-n}}\1)(w_1)
    \leq
    \deg(T_{x_{-j}}^j)e^{S_j\varphi(w'_1)}\cL^{n-j}_{x_{-n}}\1(w'_1)
    \\
    & \leq  \deg(T_{x_{-j}}^j)e^{S_j\varphi(w'_1)}e^{Q_{x_{-j}}}\cL^{n-j}_{x_{-n}}\1(w'_2)
    \leq  C_\vp (x)\cL_{x_{-n}}^n\1(w_2).
  \end{align*}
   This shows (\ref{eq:preq0}).
By Proposition \ref{4.3}
  \begin{equation}
    \label{eq:preb}
    \int_{\J_x}\tcL^n_{x_{-n}}(\1 )d\nu_x=
    \int_{\J_{x_{-n}}}\1 d\nu_{x_{-n}}=1,
  \end{equation}
  which implies the existence of $w,w'\in \J_{x}$ such that $\tcL^n_{x_{-n}}\1 (w)\leq 1$
  and $\tcL^n_{x_{-n}}\1 (w')\geq 1$.  Therefore, by the already
  proved part of this lemma, we get \eqref{eq:preq}.
 \end{proof}

\begin{proof} [Proof of Proposition \ref{4.7}]
  Let $x\in \orb$. Then by Lemma \ref{lem:Cvarphi}, for every $k\geq
  0$ and all $w_1,w_2\in \J_x$ with $\vr (w_1,w_2)\leq \xi$, we have
  that
  $$|\tcL^k_{x_{-k}}\1 (w_1) -\tcL^k_{x_{-k}}\1 (w_2)|
  \leq C_\vp (x) 2 Q_x \vr ^\al (w_1,w_2)$$ and $1/C_\vp (x) \leq
  \tcL^k_{x_{-k}}\1 \leq C_\vp (x)$. It follows that the sequence
  $$q_{x,n} := \frac{1}{n} \sum_{k=0}^{n-1} \tcL^k_{x_{-k}}\1  \quad , \,\; n\geq 1,$$
  is equicontinuous for every $x\in \orb$.  Therefore, there exists a
  sequence $n_j\to\infty$ such that $q_{x,n_j}\to q_x$
  uniformly for every $x$ of the countable set $\orb$. The functions
  $q_x$ have all the required properties.
 \end{proof}

 Let
\begin{equation}\index{invariant density}
  \mu_x:=q_x\nu_x,
\end{equation}
and let $\hcL_x:=\cL_{\hvarphi,x}$ be the transfer operator with
potential 
\begin{displaymath}
  \hvarphi_x=\varphi_x+\log q_x - \log q_{\shift (x)}\circ T_x
  -\log\lambda_x.
\end{displaymath}
Then
 \begin{equation}
  \label{eq:def_hcL} \hcL_x g_x = \frac{1}{q_{\shift(x)}} \tcL_x (g_xq_x) \quad \textrm{ for
    every }g_x\in L^1(\mu_{x}).\end{equation}
  Consequently
  \begin{equation}
    \label{eq:hcL11}
    \hcL_x \1_x = \1_{\shift(x)}.
  \end{equation}

\begin{lem}
  \label{lem:T3.1.3}
  For all $g_{\shift (x)}\in L^1(\mu_{\shift
    (x)})=L^1(\nu_{\shift (x)})$, we have
  \begin{equation}
    \label{eq:hcL*3}
    \mu_{x}(g_{\shift (x)}\circ T_x)=\mu_{\shift (x)}(g_{\shift (x)}) \, .
  \end{equation}
\end{lem}

\begin{proof} 
  From conformality of $\nu_x$ (see Proposition \ref{4.3}) it follows
  that
  \begin{equation} \label{eq:hcL*}
    \begin{array}{rl}
      \hcL^*_x(\mu_{\shift (x)})(g_x)&=\int_{\cJ_{\shift
          (x)}}\hcL_x(g_x)
      d\mu_{\shift (x)}=
      \lambda_{x}^{-1}\int_{\cJ_{\shift (x)}}(\cL_x g_x q_{x})
      d\nu_{\shift (x)}\\
      &=\lambda_{x}^{-1}\hcL^*_x(\nu_{\shift (x)})(g_x q_x)
      = \nu_x(g_x q_x)=\mu_x(g_x).
    \end{array}
  \end{equation}
  So, if $f_x\cdot (g_{\shift (x)}\circ T_x)\in L^1(\mu_x)$, then
  \begin{equation}
    \label{eq:hcL*2}
    \begin{aligned}
      \mu_x\big((g_{\shift (x)}\circ T_x)f_x\big)&=
      \hcL^*_x(\mu_{\shift (x)})\big((g_{\shift (x)}\circ T_x)f_x\big)\\
      &=
      \mu_{\shift (x)}\Big(\hcL_x\big((g_{\shift (x)}\circ T_x)f_x\big)\Big)
      =  \mu_{\shift (x)}\big(g_{\shift (x)}\hcL_x(f_x)\big),
    \end{aligned}
  \end{equation}
  since 
  \begin{displaymath}
    \hcL_x\big((g_{\shift (x)}\circ T_x)f_x\big)=g_{\shift (x)}\hcL_x(f_x).  
  \end{displaymath}
  Substituting in \eqref{eq:hcL*2} $\1_x$ for $f_x$ and using
  \eqref{eq:hcL11}, we get the lemma.
 \end{proof}

\begin{rem}
  In Chapter~\ref{ch:meas-press-gibbs} we provide sufficient
  measurability conditions for these fiber measures $\nu_x$ and
  $\mu_x$ to be integrable to produce global measures projecting on
  $X$ to $m$. The measure $\mu$ defined by \eqref{eq:18} is then
  $T$-invariant.
\end{rem}

\section{Levels of Positive Cones of H\"older Functions}
For $s\geq 1$, set
\begin{multline} \label{4.10} \conex =\Big \{ g\in \cC (\J_x): g\geq
  0,\; \nu_x(g)=1 \textrm{ and }
  g(w_1)\leq e^{sQ_x \vr ^\al (w_1,w_2 )} g(w_2)\\
  \textrm{for all } w_1,w_2\in \J_x \textrm{ with }  \vr
  (w_1,w_2)\leq \xi \Big\}.
\end{multline}
In fact all elements of $\conex$ belong to $\cH^\alpha(\cJ_x)$. This
is proved in the following lemma.

\begin{lem}
  \label{4.11TXE}
  If $g\geq 0$ and for all $w_1,w_2\in \J_x$ with $\vr (w_1,w_2)\leq
  \xi$, we have
  \begin{displaymath}
    g(w_1)\leq e^{sQ_x \vr ^\al (w_1,w_2 )} g(w_2),
  \end{displaymath}
  then
  \begin{displaymath}
    v_\alpha(g)\leq sQ_x(\exp(sQ_x \xi^\alpha)) \xi^\alpha
    ||g||_\infty.
  \end{displaymath}
\end{lem}
\begin{proof} 
  Let $w_1,w_2\in \J_x$ be such that $\varrho(w_1,w_2)\leq
  \xi$. Without loss of generality we may assume that $g(w_1)>
  g(w_2)$. Then $g(w_1)>0$ and therefore, because of our hypothesis,
  $g(w_2)>0$. Hence, we get
  \begin{displaymath}
    \frac{|g(w_1)-g(w_2)|}{|g(z_2)|}=\frac{g(w_1)}{g(w_2)}-1\leq
    \exp\big( sQ_x \varrho^\alpha(w_1,w_2)\big)-1.
  \end{displaymath}
  Then
  \begin{displaymath}
    |g(w_1)-g(w_2)|\leq sQ_x(\exp(sQ_x \xi^\alpha)) \varrho^\alpha(w_1,w_2)
    ||g||_\infty.
  \end{displaymath}
 \end{proof}

Hence the set $ \conex$ is a level set of the cone defined in
\eqref{eq:21}, that is
\begin{displaymath}
  \conex=\cC^s_x\cap\{g:\nu_x(g)=1\}.
\end{displaymath}
In addition, in the following lemma we show that this set is bounded
in $\cH^\alpha(\cJ_x)$.  \blem\label{4.12} For a.e. $x\in X$ and every
$g\in \conex$, we have $\| g\| _\infty \leq C_{\max} (x)$, where
$C_{\max}$ is defined by \eqref{eq:defCmax}.  \elem 
\begin{proof} 
  Let $g\in
\conex$ and let $z\in \J_x$. Since $g\geq 0$ we get 
\begin{displaymath}
  \int_{B(z,\xi )}
g\, d\nu_x \leq \int _{\J_x} g\, d\nu_x =1.
\end{displaymath}
Therefore there exists $b\in \overline{B}(z,\xi )$ such that 
\begin{displaymath}
  g(b)\leq 1/\nu _x (B(z,\xi )) \leq 1/D_\xi(x),
\end{displaymath}
where the latter
inequality is due to Lemma \ref{lem:l1rds18}.  Hence
\begin{displaymath}
  g(z)\leq e^{sQ_x \vr ^\al (b,z)}g(b) \leq \frac{e^{sQ_x}}{D_\xi(x)}
  \leq C_{\max} (x).
\end{displaymath}
 
\end{proof}

A kind of converse to Lemma~\ref{4.11TXE} is given by the following.
\begin{lem}
  \label{4.11NEW} If $g\in \holderx$ and $g\geq 0$, then
  $$\frac{g+v_\alpha(g)/Q_x}{\nu_x (g)+v_\alpha(g)/Q_x} \in \Lambda_x^1.$$
\end{lem}
\begin{proof} 
  Consider the function $h=g+v_\alpha(g)/Q_x$. In order to get the
  inequality from the definition of $\conex$, we take $z_1,z_2\in
  \J_x$. If $h(z_1)\leq h(z_2)$ then this inequality is
  trivial. Otherwise $h(z_1)>h(z_2)$, and therefore
  \begin{displaymath}
    \frac{h(z_1)}{h(z_2)}-1=\frac{|h(z_1)-h(z_2)|}{|h(z_2)|}\leq
    \frac{v_\alpha(g)\varrho^\alpha(z_1,z_2)}
    {v_\alpha(g)/Q_x}=Q_x \varrho^\alpha(z_1,z_2).
  \end{displaymath}
 \end{proof}

An important property of the sets $\conex$ is their invariance with
respect to the normalized operator $\tcL_x = \lam_x^{-1} \cL_x$.

\blem \label{4.13}
Let $g\in \conex$. Then, for every $n\geq 1$,
\begin{displaymath}
  \frac{\tcL_x^n g(w_1)}{\tcL_x^n g(w_2)} \leq
  \exp \big( sQ_{x_n} \vr ^\al (w_1,w_2 )\big),
  \;\; w_1,w_2\in \J_{\shift ^n (x)} \; \textrm{with } \; \vr
  (w_1,w_2)\leq\xi .
\end{displaymath}
Consequently $\tcL_x^n (\conex ) \subset \cone_{\shift ^n (x)}$
for a.e. $x\in X$ and all $n\geq 1$.  \elem

Notice that the constant function $\1\in \conex$ for every $s\geq
1$. For this particular function our distortion estimation was
already proved in Lemma \ref{lem:Cvarphi}.

\begin{proof} [Proof of Lemma~\ref{4.13}]
  Let $g\in \conex$, let $\varrho_{\shift^n(x)}(w_1,w_2)\leq\xi$, and
  let $z_1\in T_{x}^{-n}(w_1)$.  For $y=(x,z_1)$, we put $z_2=
  T_y^{-n}(w_2)$.  With this notation, we obtain from
  Lemma~\ref{lem:l1rds13} and from the definition of $\conex$ that
  \begin{equation}
    \begin{aligned}[t]
      \label{eq:KeyToExpConv}
      \frac{\tcL_{x}^n g(w_1)}{\tcL_{x}^n g(w_2)}
      &\leq \sup_{z_1\in T_{x}^{-n}(w_1)}
      \frac{\exp\big({S_n\varphi _x (z_1)}\big)g(z_1)}
      {\exp\big({S_n\varphi _x (z_2)}g(z_2)\big)}\\
    &\leq \exp\Big(\varrho^\alpha(w_1,w_2)\Big(
      \sum_{j=0}^{n-1}H_{\shift^j(x)}
      (\gamma_{\shift^j(x)}^{n-j})^{-\alpha}
      +sQ_{x}(\gamma_{x}^{n})^{-\alpha}\Big)\Big).
    \end{aligned}
  \end{equation}
  Since
  \begin{equation}
    \label{eq:eQ1-1}
    Q_{x}(\gamma_{x}^{n})^{-\alpha}+
    \sum_{j=0}^{n-1}H_{\shift ^j(x)}(\gamma_{\shift^j(x)}^{n-j})^{-\alpha}
    =Q_{\shift^n(x)},
  \end{equation}
  the lemma follows.
 \end{proof}

\blem\label{4.14} With $C_{\min}$ the function given by \eqref{eq:defCmin} we have that
$$\tcL ^i _{x_{-i}} g \geq C_{\min} (x)\quad \text{for every}\;\; i\geq j(x) \;\; \text{and}\;\; g\in \cone_{x_{-i}}.$$
\elem

\bpf
First, let $i=j(x)$.  Since
$
  \int_{\cJ_{x_{-i}}}g d\nu_{x_{-i}}=1
$
there exists $a\in \J_{x_{-i}}$ such that $g(a)\geq 1$.  By definition
of $j(x)$, for any point $w\in \J_x$, there exists $z\in
T_{x_{-i}}^{-i}(x)\cap B(a,\xi)$. Therefore
\begin{equation*}
  \tcL_{x_{-i}}^ig (w)\geq e^{S_i\varphi_{x_{-i}}(z)}g(z)\geq
  e^{S_i\varphi_{x_{-i}}(z)}e^{-sQ_x}g(a)\geq C_{\min}(x).
\end{equation*}
The case $i > j(x)$ follows from the previous one, since
$\tcL_{x_{-i}}^{i-j(x)}g_{x_{-i}}\in \Lambda_{x_{-j(x)}}$.
\epf

\section{Exponential Convergence of Transfer Operators}
\label{sec:expon-conv-transf}

\blem \label{4.15} Let $\beta_x=\beta_{x}(s)$
(cf. \eqref{eq:betax}). Then for $x\in X$, $i\geq j(x)$ and
$g_{x_{-i}}\in \cone_{x_{-i}}$, there exists $h_x\in \conex$ such that
$$( \tcL ^i g)_x=\tcL^i_{x_{-i}} g_{x_{-i}} = \beta_x q_x +(1-\beta_x)h_x  \; .$$
\elem

\bpf
By Lemma~\ref{4.14}, we have
$
  \tcL^i_{x_{-i}} g_{x_{-i}} \geq C_{\min} (x).
$
Then by (\ref{4.18}) for all $w,z\in \J_x$ with
$\varrho_x(z,w)<\xi$,
\begin{displaymath}
  \begin{aligned}
    &\beta_x\Big(\exp\big(sQ_{x}\varrho_x^\alpha(z,w)\big)q_x(z)-q_x(w)\Big)
    \leq\\
    &\quad\leq \beta_x\Big(\exp\big(sQ_{x}\varrho_x^\alpha(z,w)\big)
    -\exp\big(-sQ_{x}\varrho_x^\alpha(z,w)\big)\Big)q_{x}(z)\\
    &\quad\leq \beta_x \Big(\exp\big(sQ_{x}\varrho_x^\alpha(z,w)\big)
    -\exp\big(-sQ_{x}\varrho_x^\alpha(z,w)\big)\Big)
    C_\varphi({x})\\
    &\quad\leq \beta_x {C_\varphi({x})
      \big(1-\exp({-2sQ_{x}\varrho_x^\alpha(z,w)})\big)}
    {\exp\big(sQ_{x}\varrho_x^\alpha(z,w)\big)}\\
    &\quad \leq\Big(\exp\big(sQ_x\varrho_x^\alpha(z,w)\big)
    -\exp\big((sQ_x-H_{x_{-1}}\gamma_{x_{-1}}^{-\alpha})
    \varrho_{x}^\alpha(z,w)\big)\Big)
    \tcL_{x_{-i}}^ig_{x_{-i}}(z)\\
    &\quad\leq\Big(\exp\big(sQ_x\varrho_x^\alpha(z,w)\big)
    -\exp\big((sQ_{x_{-1}}
    +H_{x_{-1}})\gamma_{x_{-1}}^{-\alpha}\varrho_{x}^\alpha(z,w)\big)\Big)
    \tcL_{x_{-i}}^ig_{x_{-i}}(z).
\end{aligned}
\end{displaymath}

Since by (\ref{eq:KeyToExpConv}), for $h\in \Lambda_{{x}_{-1}}^s$,
\begin{displaymath}
  \tcL_{x_{-1}}h(z)\leq
  \exp\big((sQ_{{x}_{-1}}+H_{{x}_{-1}})
  \gamma_{{x}_{-1}}^{-\alpha}\varrho_{x}^\alpha(z,w)\big)\tcL_{x_{-1}} h(w),
\end{displaymath}
\begin{displaymath}
  \tcL_{x_{-i}}^i g_{x_{-i}}(z)\leq
  \exp\big((sQ_{{x}_{-1}}+H_{{x}_{-1}})
  \gamma_{{x}_{-1}}^{-\alpha}\varrho_{x}^\alpha(z,w)\big)
  \tcL_{x_{-i}}^i g_{x_{-i}}(w).
\end{displaymath}
Then we have that
\begin{multline*}
  \beta_x
  \Big(\exp\big(sQ_{x}\varrho_{x}^\alpha(z,w)\big)q_{x}(z)-q_{x}(w)\Big)\\
  \leq
  \exp\big(sQ_{x}\varrho_{x}^\alpha(z,w)\big)\tcL_{x_{-i}}^i g_{x_{-i}}(z)-
  \tcL_{x_{-i}}^i g_{x_{-i}}(w)
\end{multline*}
and then
\begin{displaymath}
  \tcL_{x_{-i}}^i g_{x_{-i}}(w) -\beta_x q_{x}(w)
  \leq
  \exp\big(sQ_{x}\varrho_{x}^\alpha(z,w)\big)\big(\tcL_{x_{-i}}^i g_{x_{-i}}(z)-
  \beta_x q_{x}(z)\big).
\end{displaymath}
Moreover, $\beta_x q_{x}\leq C_{\min}(x)\leq \tcL_{x_{-i}}^i
g_{x_{-i}}$.  Hence the function
\begin{displaymath}
  h_x:=\frac{\tcL_{x_{-i}}^i g_{x_{-i}}-\beta_x q_{x}}{1-\beta_x}
  \in \Lambda_{x}^s.
\end{displaymath}
\epf

We are now ready to establish the first result about exponential convergence.

\bprop\label{4.16} Let $s>1$. There exist $B<1$
and a measurable function $A:X\to (0,\infty )$ such that for
a.e. $x\in X$ for every $N\geq 1$ and $g_{x_{-N}}\in
\Lambda_{x_{-N}}^s$ we have
$$\|( \tcL ^N g)_x -q_x \|_\infty=\| \tcL ^N_{x_{-N}} g_{x_{-N}}
-q_x \|_\infty \leq A(x) B^N .$$
\eprop

\begin{proof} 

 Fix $x\in X$. Put $g_{n}:=g_{x_n}$, $\beta_n:=\beta_{x_n}$,
$\cone_n:=\cone_{x_n}$ and $( \tcL^{n} g )_{k}:=( \tcL^{n} g
)_{x_k}$. Let $(i(n))_{n=1}^{\infty}$ be a sequence of integers such
that $i({n+1}) \geq j(x_{-S(n)} )$, where $S(n)=\sum_{k=1}^n i(k)$,
$n\geq 1$, and where $S(0)=0$. If $g_{-S(n)}\in \Lambda_{-S(n)}^s$,
then Lemma~\ref{4.15} yields the existence of a function $h_{n-1}
\in \cone_{-S(n-1)}$ such that
\begin{displaymath}
  \begin{aligned}
    \Big( \tcL^{i(n)} g \Big)_{-S(n-1)} &=\beta_{{-S(n-1})}
    q_{{-S({n-1})}} +(1-\beta_{{-S({n-1})}})
    h_{n-1}\\
    &= \Big( 1-(1-\beta_{{-S({n-1})}})\Big)q_{{-S({n-1})}}
    +(1-\beta_{{-S({n-1})}})h_{n-1}.
  \end{aligned}
\end{displaymath}
Since
\begin{equation*}
\begin{aligned}
 \Big( \tcL^{i(n)+i(n-1)} g& \Big)_{-S(n-2)}= \left(\tcL^{i(n-1)}
 \Big( \tcL^{i(n)} g \Big)\right)_{{-S({n-2}})}\\
  &\quad=\left(\tcL^{i(n-1)} \Big( \beta_{{-S(n-1})}
    q_{{-S({n-1})}} +(1-\beta_{{-S({n-1})}})
    h_{n-1} \Big)\right)_{{-S({n-2}})}\\
  &\quad=\beta_{-S(n-1)} q_{-S(n-2)} +(1-\beta _{-S(n-1)})
  \left(\tcL^{i(n-1)}(h_{n-1}) \right)_{{-S({n-2}})}\\
  \end{aligned}
\end{equation*}
it follows again from Lemma \ref{4.15} that there is $h_{n-2} \in
\cone_{-S(n-2)}$ such that
\begin{multline*}
  \Big( \tcL^{i(n)+i(n-1)} g \Big)_{-S(n-2)}=\\
  \shoveright{= \beta_{-S(n-1)} q_{-S(n-2)}+(1-\beta _{-S(n-1)})
  \Big( \beta _{-S(n-2)} q_{S(n-2)} +(1-\beta_{-S(n-2)}) h_{n-2} \Big)}\\
  =\Big( 1-(1-\beta _{-S(n-2)})(1-\beta _{-S(n-1)}) \Big)
  q_{S(n-2)}+ (1-\beta _{-S(n-2)})
  (1-\beta _{-S(n-1)}) h_{n-1}.
\end{multline*}
It follows now by induction that
there exists $h\in \cone_x$ such that
$$\Big( \tcL ^{S(n)} g\Big)_x =\Big( \tcL ^{i(n)+...+i(1)} g\Big)_x =
(1 -\Pi_x^{(n)} )q_x +\Pi_x^{(n)} h$$
where we set $\Pi_x^{(n)} =\prod_{k=0}^{n-1} (1-\beta _{x _{-S(k)}}).$
Since $h \in \conex$, we have $|h| \leq C_{\max}(x)$. Therefore,
\beq\label{4.17}
\left| \Big( \tcL ^{S(n)} g\Big)_x -\Big( 1- \Pi_x^{(n)} \Big)q_x \right|
\leq C_{\max} (x) \Pi_x^{(n)} \qquad \text{if}\;\; g_{-S(n)}\in \Lambda_{-S(n)}^s\,.\eeq

%****************************
 By measurability of $\beta$ and $j$ one can find $M>0$ and $J\geq
   1$ such that the set
   \begin{equation}
     \label{eq:defG}
     G:=\{x:\beta_x\geq M \textrm{ and } j(x)\leq J\}
   \end{equation}
   has a positive measure larger than or equal to $3/4$. Now, we will
   show that for a.e. $x\in X$ there exists a sequence
   $(n_k)_{k=0}^\infty$ of non-negative integers such that $n_0=0$,
   for $k>0$, we have that $x_{-Jn_k}\in G$, and
   \begin{equation}
     \label{eq:retG}
     \#\{n: 0\leq n<n_k \textrm{ and } x_{-Jn}\in G\}=k-1.
   \end{equation}
   Indeed, applying Birkhoff's Ergodic Theorem to the mapping $\shift
   ^{-J}$ we have that for almost every $x\in X$,
   \begin{displaymath}
     \lim_{n\to\infty}
     \frac{\#\{0\leq m\leq n-1:\shift ^{-Jm}(x)\in G\}}{n}=\cE(\1_G|\cI_J)(x),
   \end{displaymath}
   where $\cE(\1_G|\cI_J)$ is the conditional expectation of $\1_G$
   with the respect to the $\sigma$-algebra $\cI_J$ of $\shift
   ^{-J}$-invariant sets. Note that if a measurable set $A$ is $\shift
   ^{-J}$-invariant, then set $\cup_{j=0}^{J-1}\shift ^{j}(A)$ is
   $\shift ^{-1}$-invariant. If $m(A)>0$, then from ergodicity of
   $\shift ^{-1}$ we get that $m(\cup_{j=0}^{J-1}\shift ^{j}(A))=1$,
   and then by invariantness of the measure $m$, we conclude that
   $m(A)\geq 1/J$. Hence we get that for almost every $x$ the sequence
   $n_k$ is infinite and
   \begin{equation}
     \label{eq:10a}
     \lim_{k\to\infty}\frac{k}{n_k}\geq \frac{3}{4J}.
   \end{equation}

   Fix $N\geq 0$ and take $l\geq 0$ so that $Jn_{l}\leq N \leq
   Jn_{l+1}$.  Define a finite sequence $\big(S(k)\big)_{k=1}^l$ by
   $S(k):=Jn_k$ for $k<l$ and $S(l):=N$, and observe that by
   (\ref{eq:10a}), we have $N\leq J n_{l+1} \leq 4 J^2l$.  Then~(\ref{4.18})
   and (\ref{4.17}) give
   \begin{displaymath}
     \begin{aligned}
       ||\tcL_{x_{-N}}^{N}g_{x_{-N}}-q_x||_\infty&\leq
       \Big|\Big|\tcL_{x_{-N}}^{N}g_{x_{-N}}-
       \Big(1-\Pi_x^{(l)}\Big)q_x\Big|\Big|_\infty+
       \Pi_x^{(l)}||q_x||_\infty\\
       &\leq (1-M)^{l}\big(C_\varphi(x)+ C_{\max}(x)\big)\\
       &\leq
       (\sqrt[4J^2]{1-M})^{N}\big(C_\varphi(x)+ C_{\max}(x)\big).
     \end{aligned}
   \end{displaymath}
   This establishes our proposition with $B=\sqrt[4J^2]{1-M}$ and
   \begin{displaymath}
     A(x):=\max\{2C_{\max}(x)B^{-Jk^*_x},(C_\varphi(x)+ C_{\max}(x))\},
   \end{displaymath}
   where $k^*_x$ is a measurable function such that we have
$
      \frac{k}{n_k}\geq \frac{1}{2J}
  $ for all $k\geq k^*_x$.
 \end{proof}

From now onwards throughout this section, rather than the operator
$\tcL$, we consider the operator $\hcL_x$ defined previously in \eqref{eq:def_hcL}.

\blem \label{lem:Inv100NEW} Let $s>1$ and let $g:\cJ \to \R$ be any
function such that $g_x\in\cH^\alpha(\cJ_x)$. Then, with the notation
of Proposition~\ref{4.16}, we have
  \begin{displaymath}
    \Big|\Big|\hcL_x^n g_x -
    \Big(\int g_xd\mu_x\Big)\1\Big|\Big|_\infty\leq
    C_{\vp } (\shift ^n (x))\Big(\int |g_x|d\mu_x
    +4\frac{v_\alpha(g_xq_x)}{Q_x}\Big)A(\shift ^n(x))B^n.
  \end{displaymath}
\elem

\begin{proof} 
  Fix $s> 1$. First suppose that $g_x\geq 0$. Consider the function
  \begin{displaymath}
    h_x=\frac{g_x+v_\alpha(g_x)/Q_x}{\Delta_x}\quad
    \textrm{where} \quad \Delta_x:=\nu_x (g_x)+v_\alpha(g_x)/Q_x.
  \end{displaymath}
  It follows from Lemma~\ref{4.11NEW} that $h_x$ belongs to the set
  $\Lambda^s_x$ and from Proposition~\ref{4.16} we have
 \begin{eqnarray*}
     \Big|\Big|\tcL ^n _x g_x -\Big( \int g_x \, d\nu_x \Big)
      q_{\shift^n(x)}\Big|\Big|_\infty &\leq&
      \Big|\Big|\Delta_x\tcL ^n _x h_x
      -\frac{v_\alpha(g_x)}{Q_x}\tcL ^n _x\1_x- \Big( \int g_x \,
      d\nu_x \Big)
      q_{\shift^n(x)}\Big|\Big|_\infty\\
      & =&\Big|\Big|\Delta_x\tcL ^n _x h_x
      -\Delta_x q_{\shift^n(x)}
      +\frac{v_\alpha(g_x)}{Q_x}\Big(q_{\shift^n(x)} -
      \tcL^n_x\1_x\Big)
      \Big|\Big|_\infty\\
     & \leq&
      \Big(\Delta_x+\frac{v_\alpha(g_x)}{Q_x}\Big)A(\shift ^n (x))B^n.
  \end{eqnarray*}

  Then applying this inequality for $g_xq_x$ and using (\ref{4.18})
  we get
  \begin{equation*}
    \begin{aligned}
      \Big|\Big|\hcL ^n _x g_x -\Big( \int g_x \, d\mu_x \Big)
      \1_{\shift ^n (x)} \Big|\Big|_\infty&\leq
     \Big|\Big|\frac{1}{q_{\shift^n (x)}}\Big|\Big|\cdot
      \Big|\Big|\tcL ^n _x (g_xq_x)
      -\Big( \int g_xq_x \, d\nu_x \Big) q_{\shift ^n (x)}  \Big|\Big|_\infty\\
      &\leq C_{\vp } (\shift ^n (x))
      \Big( \int g_x \, d\mu_x
      +2\frac{v_\alpha(g_xq_x)}{Q_x}\Big)A(\shift ^n (x))B^n.
    \end{aligned}
  \end{equation*}
  So, we have the desired estimate for non-negative $g_x$. In the
  general case we can use the standard trick and write
  $g_x=g_x^+-g_x^-$, where $g_x^+,g_x^-\geq 0$. Then the lemma
  follows.
 \end{proof}

The estimate obtained in Lemma \ref{lem:Inv100NEW} is a bit
inconvenient for it depends on the values of a measurable function,
namely $C_\vp A$, along the positive $\shift$--orbit of $x\in X$.  In
particular, it is not clear at all from this statement that the item
\eqref{item:4} in Theorem~\ref{thm:Gib50B} holds. In order to remedy
this flaw, we prove the following proposition.

\bprop \label{4.19NEW} For $m$--a.e. $x\in X$ and every $g_x \in
\cC(\cJ_x)$, we have
\begin{displaymath}
  \Big|\Big|\hcL_x^n g_x -
  \Big(\int g_xd\mu_x\Big)\1_{\shift^n(x)}\Big|\Big|_\infty
  \xrightarrow[n\to \infty]{ } 0.
  \end{displaymath}
\eprop

\bpf First of all, we may assume without loss of generality that the
function $g_x\in \holderx$ since every continuous function is a limit
of a uniformly convergent sequence of H\"older functions.  Now, let
$\cA >0$ be sufficiently big such that the set
\begin{equation}
  \label{eq:defXA}
  X_{\cA} =\{x\in X ; \,   A(x) \leq \cA \}
\end{equation}
has positive measure. Notice that, by ergodicity of $m$, some
iterate of a.e. $x\in X$ is in the set $X_{\cA}$.
Then by Poincar\'e recurrence theorem and ergodicity of $m$, for
a.e. $x\in X$, there exists a sequence $n_j\to \infty$ such that
$\shift ^{n_j}(x) \in X_{\cA}$, $j\geq 1$. Therefore we get, for such
an $x\in X_{\cA}$, from Lemma \ref{lem:Inv100NEW} that
\begin{equation}
  \label{eq:anotherconv100}
  \Big\|\hcL_x^{n_j} g_x -
  \Big(\int g_xd\mu_x\Big)\1_{\shift^{n_j}(x)}\Big\|_\infty
  \Big( \int |g_x| \, d\mu_x+
  4\frac{v_\alpha(g_xq_x)}{Q_x}\Big)^{-1} \leq \cA B^{n_j}
\end{equation}
for every $j\geq 1$. Finally, to pass from the subsequence $(n_j)$ to
the sequence of all natural numbers we employ the monotonicity
argument that already appeared in Walters paper \cite{Wal78}. Since
$\hcL_x \1_x =\1_{\shift(x)}$, we have for every $w\in \cJ_{\shift
  (x)}$ that
\begin{equation*}
  \inf_{z\in \cJ_x}g_x(z)
  \leq\sum_{z\in T_x^{-1}(w)}g_x(z)e^{\hat{\varphi}(z)}
  \leq \sup_{z\in \cJ_x}g_x(z).
\end{equation*}
Consequently the sequence
\begin{displaymath}
(M_{n,x})_{n=0}^\infty=(\sup_{w\in
  \cJ_{\shift ^n(x)}}\hcL_x^n g_x(w))_{n=0}^\infty
\end{displaymath}
is weakly decreasing. Similarly we have a weakly increasing sequence
\begin{displaymath}
   (m_{n,x})_{n=0}^\infty=(\inf_{w\in \cJ_{\shift ^n(x)}}\hcL_x^n
g_x(w))_{n=0}^\infty. 
\end{displaymath}
The proposition follows since, by (\ref{eq:anotherconv100}), both
sequences converge on the subsequence $(n_j)$.  \epf

\section{Exponential Decay of Correlations}
The following proposition proves item \eqref{item:6}
 of Theorem~\ref{thm:Gib50B}.
For a function $f_x\in L^1(\mu_x)$ we denote its $L^1$--norm with
respect to $\mu_x$ by
$$\|f_x\|_1:=\int|f_x|d\mu_x.$$

\begin{prop}
  \label{prop:EDC} There exists a $\shift$--invariant set $X'\subset
  X$ of full $m$--measure such that, for every $x\in X'$, every
  $f_{\shift^n(x)}\in L^1(\mu_{\shift^n(x)})$ and every
  $g_x\in\cH^\alpha(\cJ_x)$,
  \begin{displaymath}
    \big|\mu_x\big((f_{\shift ^n(x)}\circ T_x^n)g_x\big)-
    \mu_{\shift ^n(x)}(f_{\shift ^n(x)})\mu_x(g_x)\big|\leq
    A_*(g_x,\shift ^n(x))B^n ||f_{\shift ^n(x)}||_1
  \end{displaymath}
  where
  \begin{displaymath}
    A_*(g_x,\shift ^n(x)):=
    C_{\vp } (\shift ^n (x))\Big(\int |g_x| d\mu_x
    +4\frac{v_\alpha(g_xq_x)}{Q_x}\Big)A(\shift ^n(x)).
  \end{displaymath}
\end{prop}
\begin{proof} 
  Set $h_x=g_x-\int g_x d\mu_x$ and note that by (\ref{eq:hcL*2}) and
  (\ref{eq:hcL11}) we have that
  \begin{equation}
    \label{eq:EDC10}
    \begin{aligned}
      \mu_x\big((f_{\shift ^n(x)}\circ T_x^n)g_x\big)&-
      \mu_{\shift ^n(x)}(f_{\shift ^n(x)})\mu_x(g_x)=\\
      &=\mu_{\shift^n(x)}\big(f_{\shift ^n(x)}\hcL_x^n(g_x)\big)-
      \mu_{\shift ^n(x)}(f_{\shift ^n(x)})\mu_x(g_x)\\
      &=\mu_{\shift^n(x)}\big(f_{\shift ^n(x)}\hcL_x^n(h_x)\big).
  \end{aligned}
  \end{equation}
  Since Lemma~\ref{lem:Inv100NEW} yields
 $
    ||\hcL_{x}^n h_x||_\infty\leq
    A_*(g_x,\shift ^n(x)) B^n
 $
  it follows from (\ref{eq:EDC10}) that
  \begin{displaymath}
    \begin{aligned}
      \big|\mu_x\big((f_{\shift ^n(x)}\circ T_x^n)g_x\big)-
      \mu_{\shift ^n(x)}(f_{\shift ^n(x)})\mu_x(g_x)\big|&
      \leq
      \int\big|f_{\shift ^n(x)}\hcL_x^n(h_x)\big|d\mu_{\shift
        ^n(x)}\\
      &\leq A_*(g_x,\shift ^n(x))B^n\int\big|f_{\shift
        ^n(x)}\big|d\mu_{\shift ^n(x)}.
    \end{aligned}
  \end{displaymath}
 \end{proof}
Using similar arguments like in Proposition \ref{4.19NEW} we obtain
the following.
\begin{cor}
  \label{cor:E100}
  Let $f_{\shift ^n(x)}\in L^1(\mu_{\shift^n(x)})$ and $g_x\in
  L^1(\cJ_x)$, where $x\in X'$ and $X'$ is the set given by Lemma
  \ref{prop:EDC}. If $||f_{\shift^n(x)}||_1\neq 0$ for all $n$, then
  \begin{displaymath}
    \frac{\big|\mu_x\big((f_{\shift ^n(x)}\circ T_x^n)g_x\big)-
      \mu_{\shift^n(x)}(f_{\shift^n(x)})\mu_x(g_x)\big|}{||f_{\shift^n(x)}||_1}
     \longrightarrow 0 \quad as \;\; n\to\infty.
  \end{displaymath}
\end{cor}

\brem\label{rem:expgrow} Note that if $||f_{\shift ^n(x)}||_1$ grows
subexponentially, then
\begin{equation}
  \label{eq:2}
  \big|\mu_x\big((f_{\shift^n(x)}\circ T_x^n)g_x\big)-
  \mu_{\shift^n(x)}(f_{\shift^n(x)})\mu_x(g_x)\big|
  \longrightarrow 0 \quad \textrm{ as } \;\; n\to \infty.
\end{equation}
This is for example the case if $x\mapsto\log||f_x||_1$ is
$m$-integrable since Birkhoff's Ergodic Theorem implies
that $(1/n)\log||f_{\shift ^n(x)}||_1 \to 0$ for a.e. $x\in X$.
\erem

\section{Uniqueness }
\label{sec:uniqueness}

\begin{lem}
  \label{lem:2}
  The family of measures $x\mapsto \nu_x$ is uniquely determined
  by condition \eqref{eq:1}.
\end{lem}

\begin{proof} 
Let $\tilde{\nu}_x$ be a family of probability measures satisfying \eqref{eq:1}.
%Then $\lam_x=\lam_x \nu_x (\1 )=\cL_x^* \nu_\theta (x) (\1 ) = \nu_{\theta (x)} (\cL_x \1 )$.
% Let again $\hat{\cL}_x$ be given by \eqref{eq:def_hcL}. Then
For $x\in X$ choose arbitrarily a sequence of points $w_n\in \cJ_{\theta^n (x)}$ and
define $$\nu_{x,n}:= {(\cL_x^n)^* \delta_{w_n} \over \cL_x^n \1 (w_n)}\, .$$
 Then, by Proposition~\ref{4.19NEW}, for a.e. $x\in X$ and all $g_x\in \cC (\cJ_x)$ we have
 \begin{equation}
    \label{eq:nuform}
     \lim_{n\to\infty}\nu_{x,n} (g_x)=
    \lim_{n\to\infty}\frac{\cL_x^n g_x(w_n)}{\cL_x^n \1 (w_n)}
    =\lim_{n\to\infty}\frac{\hcL_x^n(g_x/q_x)(w_n)}{\hcL_x^n(1/q_x)(w_n)}
    =\frac{\nu_x(g_x)}{\nu_x(\1 )}
    =\nu_x(g_x).
  \end{equation}
In other words,
 \begin{equation}
  \label{eq:dDirc}
  \nu_{x,n}\xrightarrow[n\to\infty ]{ }\nu_x.
\end{equation}
in the weak* topology. Uniqueness of the measures $\nu_x$ follows.
 \end{proof}

\begin{lem}
  \label{lem:4}
  There exists a unique function $q\in \cC^0(\J)$ that satisfies
  \eqref{eq:3}.
\end{lem}

\begin{proof} 
  Follows from Proposition~\ref{4.16}.
 \end{proof}

\section{Pressure function}

\emph{The \index{pressure function}pressure function} is defined by
the formula
\begin{displaymath}
  x\mapsto P_x(\varphi):=\log \lambda_x.
\end{displaymath}
If it does not lead to misunderstanding, we will also denote the
pressure function by $P_x$. It is important to note that this function
is generally non-constant, even for a.e. $x\in X$. Actually, if the
pressure function is a.e. constant, then the random map shares many
properties with a deterministic system. This will be explained in
detail in section~\ref{ch:section 5}.
Note that (\ref{eq:nuform}) and (\ref{eq:lambdaform}) imply an alternative
definition of $P_x(\varphi)$, namely
\begin{equation}
  \label{eq:formofPx}
  P_x(\varphi)= \log(\nu_{\shift (x)}(\cL_x \1 )) =
  \lim_{n\to\infty}
  \log\frac{\cL_x^{n+1} \1 (w_{n+1})}{\cL_{\shift (x)}^{n} \1 (w_{n+1})}
\end{equation}
where, for every $n\in \N$, $w_n$ is an arbitrary point from $\cJ_{\shift^n(x)}$.

\begin{lem}
  \label{lem:10}
  For $m$-a.e. $x\in X$ and every sequence $(w_n)_n\subset \J_x$
  \begin{displaymath}
    \lim_{n\to\infty}\frac{1}{n}S_n P_{x_{-n}}
    -\frac{1}{n}\log\cL_{x_{-n}}^{n}
    \1_{x_{-n}} (w_n)=0.
  \end{displaymath}
\end{lem}

\begin{proof} 
  By \eqref{4.18} and Proposition~\ref{4.16}, we have that
  \begin{displaymath}
    \frac{1}{C_\vp(x)} -A(x) B^n\leq \frac{\cL ^n_{x_{-n}}
      \1_{x_{-n}}(w)}{\lambda_{x_{-n}}^n}
    \leq C_\vp(x) +A(x) B^n
  \end{displaymath}
  for every
  $w\in \J_x$ and every $n\in \N$.
  Therefore
  \begin{displaymath}
    \log\Big(\frac{1}{C_\vp(x)} -A(x) \Big) \leq
    \log\cL ^n_{x_{-n}} \1_{x_{-n}}(w)
    -\log\lambda_{x_{-n}}^n
    \leq \log\Big(C_\vp(x) + A(x)\Big).
  \end{displaymath}
 \end{proof}

\begin{lem}
  \label{lem:9}
  For $m$-a.e. $x\in X$ and for every sequence $y_n\in \J_{x_n}$, $n\geq 0$,
  \begin{displaymath}
    \slim_{n\to\infty}\Big(\frac{1}{n}S_{n} P_x
    -\frac{1}{n}\log\cL_{x}^{n}
    \1_x (y_{n})\Big)=0.
  \end{displaymath}
\end{lem}
\begin{proof} 
  Using Egorov's Theorem and Lemma~\ref{lem:10} we have that for
  each $ \delta>0$ there exists a set $ F_\delta$ such that $
  m(X \sms X_\delta)<\delta$
  and
  \begin{displaymath}
    \frac{1}{n}S_n P_{x_{-n}}
    -\frac{1}{n}\max_{y\in \J_{x_n}}\log\cL_{x_{-n}}^{n}
    \1_{x_{-n}} (y)\xrightarrow[n\to \infty]{ } 0
  \end{displaymath}
  uniformly on $F_\delta$. The lemma follows now from Birkhoff's
  Ergodic Theorem.
 \end{proof}

\begin{lem}
  \label{lem:11}
  If there exist $g\in L^1(m)$ such that $\log \|\cL_x\1\|_\infty\leq
  g(x)$, then
  \begin{displaymath}
    \lim_{n\to\infty}\left\| \frac{1}{n}S_{n} P_x
    -\frac{1}{n}\log\cL_{x}^{n}
    \1_x \right\|_\infty=0.
  \end{displaymath}
\end{lem}
\begin{proof} 
  Let $F:=F_\delta$ be the set from the proof of Lemma~\ref{lem:9},
  let $x\in X'_{+F}$ and let $(n_j)$ be the visiting sequence. Let $j$
  be such that $n_j< n\leq n_{j+1}$.  Then
  \begin{equation}
    \label{eq:10}
    \log \cL_x^n\1(y)\leq
    \log \|\cL_x^{n_j}\1\|+S_{n-n_j} g(\shift^{n_j}(x))\quad \text{for every} \;\; y\in
  \J_{\shift^n(x)}.
  \end{equation}
  Now, let $h(x):=\|\varphi_x\|_\infty$.  Since by \eqref{eq:9}
  $-\log\lambda_x\leq \|\varphi_x\|_\infty$,
  \begin{displaymath}
    -\log \lambda_x^n=-\log\lambda_x^{n_j}-\log\lambda_{x_{n_j}}^{n-n_j}
    \leq S_{n_j} P_x + S_{n-n_j} h(\shift^{n_j}(x)).
  \end{displaymath}
  Then by \eqref{eq:10}
  \begin{displaymath}
    \frac{1}{n}S_{n} P_x-\frac{1}{n}\log\cL_{x}^{n} \1_x(y_{x_n})\leq
    \frac{1}{n_j}S_{n_j} P_x
    -\frac{1}{n_j}\log\cL_{x}^{n_j} \1_x(y_{x_{n_j}})+ \frac{1}{n}S_{n-n_j}
    (g+h)(\shift^{n_j}(x)).
  \end{displaymath}
  On the other hand, for $y\in
  \J_{\shift^n(x)}$,
  \begin{displaymath}
    \log \cL_x^n\1(y)\geq
    \log \cL_x^{n_{j+1}}\1(T_{\shift^n(x)}^{n_{j+1}-n}(y))
    -S_{n_{j+1}-n} g(\shift^{n}(x))
  \end{displaymath}
  and by
  \eqref{eq:9},
  \begin{displaymath}
    \log\lambda_x^n=\log\lambda_x^{n_{j+1}}
    -\log\lambda_{x_n}^{n_{j+1}-n}\leq \log \|\cL_x^{n_{j+1}}\1\|
    +S_{n_{j+1}-n} h(\shift^{n}(x)).
  \end{displaymath}
  The lemma follows now by Birkhoff's Ergodic Theorem.
 \end{proof}

\section{\index{Gibbs property}Gibbs property}
\label{sec:gibbs-property}

\begin{lem}
  \label{lem:Gibbs}
  Let $w\in \cJ_x$, set $y=(x,w)$ and let $n\geq 0$. Then
  \begin{displaymath}
    e^{-Q_{\shift ^n(x)}}(D_\xi(\shift ^n(x)))
    \leq
    \frac{\nu_x(T^{-n}_y(B(T^n(y),\xi)))}
    {\exp(S_n\varphi(y)-S_n P_x(\varphi))}
    \leq e^{Q_{\shift^n(x)}}.
  \end{displaymath}
\end{lem}
\begin{proof} 
  Fix an arbitrary $z\in \cJ_x$ and set $y=(x,z)$. Then by
  Lemma~\ref{lem:l1rds13b} and (\ref{conformality}) we have that
    \begin{multline*}
      \frac{\nu_x(T^{-n}_y(B(T^n(y),\xi)))}
      {\exp(S_n\varphi(y)-S_nP_x(\varphi))}\\\leq
      \frac{(\lambda_x^{n})^{-1} \nu_{\shift^n(x)}(B(T^n(y),\xi))\sup_{z'\in
          T^{-n}_y(B(T^n(y),\xi))}e^{S_n\varphi(z')}}
      {(\lambda_x^{n})^{-1}e^{S_n\varphi(y)}}\\
      \leq e^{Q_{S^n(x)}}.
    \end{multline*}
  On the other hand
  \begin{displaymath}
    \begin{split}
      \frac{\nu_x(T^{-n}_y(B(T^n(y),\xi)))}
      {\exp(S_n\varphi(y)-S_nP_x(\varphi))} &\geq
      \frac{(\lambda_x^{n})^{-1} \nu_{\shift ^n(x)}(B(T^n(y),\xi))\inf_{z'\in
          T^{-n}_y(B(T^n(y),\xi))}e^{S_n\varphi(z')}}
      {(\lambda_x^{n})^{-1}e^{S_n\varphi(y)}}\\
      &\geq \nu_{\shift^n(x)}(B(T^n(y),\xi)) e^{-Q_{S^n(x)}}.
    \end{split}
  \end{displaymath}
  The lemma follows by (\ref{eq:D_xi2}).
 \end{proof}

\begin{lem}
  \label{lem:Gibbs211}
  Let $T:\J\to\J$ satisfy the condition of measurability of
  cardinality of covers and let $\{\nu_{i,x}\}$, where $i=1,2$, be two
  Gibbs families with pseudo-pressure functions $x\mapsto P_{i,x}$.
  Then, for a.e. $x$, the measures $\nu_{1,x}$ and $\nu_{2,x}$ are
  equivalent and
  \begin{displaymath}
    \lim_{k\to\infty}\frac{1}{n_k}S_{n_k}P_{1,x}=
    \lim_{k\to\infty}\frac{1}{n_k}S_{n_k}P_{2,x}=
    \lim_{k\to\infty}\frac{1}{n_k}S_{n_k}P_{x}
  \end{displaymath}
  where $(n_k)=(n_k(x))$ is the visiting sequence of an essential set.
\end{lem}
\begin{proof} 
  Let $A$ be compact subset of $\J_x$ and let $\delta>0$. By
  regularity of $\nu_{2,x}$ we can find $\varepsilon>0$ such that
  \begin{equation}
    \label{eq:UN10}
    \nu_{2,x}(B_x(A,\varepsilon))\leq \nu_{2,x}(A)+\delta.
  \end{equation}
  Now, let $N_x$ be a measurable function such that
  $\xi(\gamma_x^{N_x})^{-1}\leq \varepsilon/2$.
  Set
  \begin{displaymath}
    A^j_n:=\{y\in T_x^{-n}(y_{x_n}^j):
    A\cap T_y^{-n}(B(y_{x_n}^j,\xi))\neq \emptyset\}.
  \end{displaymath}

  Let $Z$ be a $L, N, D, D $-essential set of $a_x,N_x,D_1, D_2$ and
  let $(n_k)=(n_k(x))$ be the visiting sequence of $Z$. Fix $k\in\N$
  and put $n=n_k(x)$. Then we have
  \begin{displaymath}
    A\subset \bigcup_{j=1}^{a_{x_n}}\bigcup_{y\in A^j_n} T_y^{-n}B(y_{x_n}^j,\xi)\subset
    B_x(A,\varepsilon).
  \end{displaymath}
  By (\ref{4.5}) it follows that
  \begin{equation}
    \label{eq:22}
    \begin{aligned}
      \nu_{1,x}(A) \leq \sum_{j=1}^{a_{x_n}}\sum_{y\in A^j_n}
      \nu_{1,x}(T_y^{-n}B(y_{x_n}^j,\xi))
      {\leq D_1(x)D \sum_{j=1}^L\sum_{y\in
          A^j_n} \exp(S_n\varphi(y)-S_n P_{1,x}(\varphi))}.
    \end{aligned}
  \end{equation}
  Then by (\ref{eq:UN10}) and again by (\ref{4.5})
  \begin{multline}
    \label{eq:gibbs110}
    \nu_{1,x}(A) \leq D_1(x)D \exp(S_n P_{2,x}-S_n P_{1,x})
    \sum_{j=1}^{a_{x_n}}\sum_{y\in A^j_n}
    \exp(S_n\varphi(y)-S_n P_{2,x}(\varphi))\\
    \shoveleft{\quad\quad\leq D_1(x)D_2(x)D^2 \exp(S_n P_{2,x}-S_n
      P_{1,x}) \sum_{j=1}^{a_{x_n}}\sum_{y\in A^j_n}
      \nu_{2,x}(T_y^{-n}B(y_{x_n}^j,\xi))}\\
    \shoveleft{\quad\quad\leq D_1(x)D_2(x)D^2
      L\exp(S_n P_{2,x}-S_n P_{1,x})\nu_{2,x}(B(A,\varepsilon))}\\
    \leq D_1(x)D_2(x)D^2 L\exp(S_n P_{2,x}-S_n
    P_{1,x})(\nu_{2,x}(A)+\delta),
   \end{multline}
  since for $y\neq y'$ such that $y,y'\in T_x^{-n}(y_{x_n}^j)$, we
  have that
  \begin{displaymath}
    T_y^{-n}B(y_{x_n}^j,\xi)\cap
    T_{y'}^{-n}B(y_{x_n}^j,\xi)=\emptyset.
  \end{displaymath}
  Hence the difference $S_{n_k} P_{2,x}-S_{n_k} P_{1,x}$ is bounded
  from below by some constant, since otherwise taking $A=\J_x$ we
  would obtain that $\nu_{1,x}(\J_x)=0$ on a subsequence of $(n_k)$ in
  (\ref{eq:gibbs110}). Similarly, exchanging $\nu_{1,x}$ with
  $\nu_{2,x}$ we obtain that $S_{n_k} P_{1,x}-S_{n_k} P_{2,x}$ is
  bounded from above. Then, letting $\delta$ go to zero, we have that
  $\nu_{1,x}$ and $\nu_{2,x}$ are equivalent.

  Note that
  \begin{multline*}
    \exp(-S_nP_{1,x})\cL_x^n\1_x(y_n)
    =\sum_{y\in T_x^{-n}(y_n)}e^{S_n\varphi_x(y)-S_nP_{1,x}}\\
    \leq
    D_1(x)D\sum_{y\in T_x^{-n}(y_n)}
    \nu_{1,x}(T_y^{-n}B(y_{n},\xi))\leq D_1(x)D\nu_{1,x}(\J_x)=D_1(x)D.
  \end{multline*}
  Then
  \begin{displaymath}
    \frac{1}{n}\log \cL_x\1_x(y_n)-\frac{1}{n}\log (D_1(x)D)
    \leq \frac{1}{n}S_nP_{1,x}.
  \end{displaymath}
  On the other hand, by (\ref{eq:22}), on the same subsequence
  \begin{displaymath}
    1= \nu_x^1(\J_x)\leq D_1(x)D
    L\sum_{y\in T_x^{-n}(y_n)}e^{S_n\varphi_x(y)-S_nP_{1,x}}
  \end{displaymath}
  for some $y_n\in\{y_{x_n}^1,\ldots, y_{x_n}^{a_{x_n}}\}$.  Therefore,
  using Lemma~\ref{lem:9} and the Sandwich Theorem, we have that, for
  $x\in X'_Z\cap X'_P$,
  \begin{displaymath}
    \lim_{k\to\infty}\frac{1}{n_k}S_{n_k}P_{1,x}=
    \lim_{k\to\infty}\frac{1}{n_k}S_{n_k}P_{x}.
  \end{displaymath}
 \end{proof}

\begin{rem}
  We cannot expect that $P_{1,x}=P_x(\varphi)$ $m$-almost
  surely since, for any measurable function $x\mapsto g_x$,
   $ P_{1,x}:=P_x(\varphi)+g_x-g_{\shift(x)}$, is also a pseudo-pressure function
     (see Lemma~\ref{lem:Gibbs}).
\end{rem}

%****************************************************************************************
\section{Some comments on Uniformly Expanding Random Maps}
By $\cC^\infty_*(\J)$ we denote the space of $\cB$-measurable
mappings $g:\J\to \R$ with ${g_x}:\cJ_x\to \R$ continuous such that
$\sup_{x\in X}\|g_x\|_\infty<\infty$. For $H_0\geq 0$, by
$\cH^\alpha_*(\J,H_0)$ we denote the space of all functions $\varphi$
in $\cH_m^\alpha(\J)\cap \cC^\infty_m(\J)$ such that all of $H_x$ are
bounded above by $H_0$. Let
\begin{displaymath}
  \cH^\alpha_*(\J)=\bigcup_{H_0\geq
    0}\cH^\alpha_*(\J,H_0).
\end{displaymath}
For $\varphi\in \cH^\alpha(\cJ ,H_0)$ we put
\begin{displaymath}
  Q:=H_0\sum_{j=1}^\infty \gamma^{-\alpha j}=\frac{H_0\gamma^{-\alpha}}
  {1-\gamma^{-\alpha}}.
\end{displaymath}
Then Lemma~\ref{lem:l1rds13b} takes on the following form.
\begin{lem}
  \label{lem:l1rds13UERM} For every $\varphi\in \cH^\alpha_*(\cJ
  ,H_0)$,
  \begin{equation*}
    |S_n\varphi_x(T^{-n}_y(w_1))-S_n\varphi_x(T_y^{-n}(w_2))|
    \leq Q\varrho^{\alpha}(w_1,w_2)
  \end{equation*}
  for all $n\geq 1$, all $x\in X$, every $z\in \J_x$ and every
  $w_1,w_2\in B(T^n(z),\xi)$ and where $y=(x,z)$.
\end{lem}

In this paper, whenever we deal with uniformly expanding random maps,
we always assume that potentials belong to $\cH^\alpha_*(\J)$. Hence all
the functions $C_\varphi(x)$, $C_{\max}(x)$, $C_{\min}(x)$ and
$\beta_x$ defined respectively by (\ref{eq:defCv}),
(\ref{eq:defCmax}), (\ref{eq:defCmin}) and (\ref{eq:betax}) are
uniformly bounded on $X$. Therefore, there exists $A\in \R$ such that
$A(x)\leq A$ for all $x\in X$, where $A(x)$ is the function from
Proposition~\ref{4.16}. In particular, we can prove the following.

\begin{lem}
  \label{lem:L590}
  There exists a constant $A_\lambda$ such that, for $x\in X$ and all
  $y_1,y_2\in \J_{x_n}$
  \begin{displaymath}
    \Big|\frac{\cL_x^{n}\1(y_1)}{\cL_{x_1}^{n-1}\1(y_1)}-\lambda_x\Big|
    \leq A_\lambda B^{n}.
  \end{displaymath}
\end{lem}
\begin{proof} 
  It follows from Proposition~\ref{4.16} that
  \begin{displaymath}
    |\tcL_{x_1}(\tcL \1)(y_1)-\tcL_{x_1}\1(y_2)|\leq 2AB^{n-1}.
  \end{displaymath}
  Then by Lemma~\ref{lem:l1rds18} and (\ref{eq:preq}) we have, for some $x$-independent constant $A_\lambda$, that
  \begin{displaymath}
    \Big|\frac{\cL_x^n\1(y_1)}{\cL_{x_1}^{n-1}\1(y_2)}-\lambda_x\Big|
    \leq \frac{2AB^{-1}B^{n}\lambda_x}{\tcL_{x_1}^n(\1)(y_2)}\leq
    A_\lambda B^{n}.
  \end{displaymath}
 \end{proof}

%%% Local Variables:
%%% mode: latex
%%% TeX-master: "RDSmain"
%%% End:

%% file: RDSsection4.tex
\chapter{Measurability, Pressure and Gibbs Condition}
\label{ch:meas-press-gibbs}
We now study measurability of the objects produced
in the previous section. Up to now we do not know, for example, whether the family of
measures $\nu_x$ represents the disintegration of a global Gibbs state
$\nu$ with marginal $m$ on the fibered space $\cJ$. Therefore, we
define abstract measurable expanding random maps for which the above
measurabilities of $\lambda_x$, $q_x$, $\nu_x$ and $\mu_x$ can be shown.
 Then, we can construct a Borel probability invariant
ergodic measure on $\J$ for the skew-product transformation $T$ with
Gibbs property and study the corresponding expected pressure.

Our settings are related to those of smooth expanding random mappings
of one fixed Riemannian manifold from \cite{Kif92} and those of random
subshifts of finite type whose fibers are subsets of
$\N^\N$ from \cite{BogGun95}. One possible extension of these works is
to consider expanding random transformations on subsets of a fixed
Polish space. A general framework for this was, in fact, prepared by
Crauel in \cite{Cra02}. In
Chapter~\ref{sec:CRS} we show how Crauel's random compact subsets of
Polish spaces fit into our general framework and, therefore, our
settings comprise all these options and go beyond.

The issue of measurability of $\lambda_x$, $q_x$, $\nu_x$ and $\mu_x$
does not seem to have been treated with care in the literature. As a
matter of fact, it was not quite clear to us even for symbol dynamics
or random expanding systems of smooth manifolds until, very recently,
when Kifer's paper \cite{Kif08} has appeared to take care of these
issues.

\section{Measurable Expanding Random Maps}
Let $T:\J\to\J$ be a general expanding random map.  Define $\pi_X : \J\to X$ by
$\pi_X(x,y)= x$.
Let $\cB:=\cB_{\J}$ be a $\sigma$-algebra on $\J$ such that
\begin{enumerate}
\item $\pi_X$ and $T$ are measurable,
\item\label{item:7} for every $A\in \cB$, $\pi_X(A)\in\cF$,
\item $\cB|_{\J_x}$ is the Borel $\sigma$-algebra on $\J_x$.
\end{enumerate}
By $L^0_m(\J)$ we denote the set of all $\cB_{\J}$-measurable
functions and by $\cC^0_m(\J)$ the set of all $\cB_{\J}$-measurable
functions $g$ such that $g_x\in \cC(\J_x)$.
\begin{lem}
  \label{lem:3}
  If $g\in \cC^0_m(\J)$, then $x\mapsto \|g_x\|_\infty$ is measurable.
\end{lem}
\begin{proof}
  The proof is a consequence of \eqref{item:7}. Indeed, let $(G_n)$ be
  an increasing approximation of $|g|$ by step functions. So let
$
    G_n=\sum_{k=1}^{m}a_k\1_{A_k},
 $
  where $(a_k)$ is an increasing sequence of non-negative real
  numbers, and $A_k$ are $\cB_{\J}$-measurable. Then, define
  \begin{displaymath}
    X_m:=\pi_X(A_m) \textrm{ and } X_k:=\pi_X(A_k)\sms \cup_{j=k+1}^m \pi_X(A_j)
  \end{displaymath}
  where $k=1,\ldots, m-1$. Let
 $
    H_n(x):=\sum_{k=0}^m a_k\1_{X_k}(x)=\sup_{y\in \J_x}G_n(x,y)
 $.
  Then the sequence $(H_n)$ is increasing and converges pointwise to
  the function $x\mapsto \|g_x\|_\infty$.
\end{proof}

The space $L^1_m(\J)$ is, by definition, the set of all $g\in L^0_m(\J)$,
such that
$
  \int \|g_x\|_\infty dm(x)<\infty.
$
We also define
\begin{displaymath}
  \cC^1_m(\J):=\cC^0_m(\J)\cap L^1_m(\J)
\end{displaymath}
and
\begin{displaymath}
  \cH^\alpha_m(\J):=\cC^1_m(\J)\cap \cH^\alpha(\J).
\end{displaymath}

By $\cM^1(\cJ )$ we denote the set of probability measures and by
$\cM^1_m(\J )$ its subset consisting of measures $\nu'$ such that
there exists a system of fiber measures $\{\nu_x'\}_{x\in X}$ with the
property that for every $g\in L^1_m(\J)$, the map
$
  x\mapsto \int_{\cJ_x} g_x \, d\nu_x'
$
is measurable and
\begin{displaymath}
  \int_\J g d\nu'=\int_X\int_{\J_x}g_x d\nu_x'dm(x).
\end{displaymath}
Then
\begin{equation}
  \label{eq:5}
  m=\nu'\circ \pi_X^{-1}
\end{equation}
and the family $(\nu_x')_{x\in X}$ is the canonical system of
conditional measures of $\nu'$ with respect to the measurable partition
$\{\J_x\}_{x\in X}$ of $\cJ$.  It is also instructive to notice that
in the case when $\J$ is a Lebesgue space then \eqref{eq:5} implies
that $\nu'\in\cM_m^1(\J)$.

The measure $\mu'\in\mathcal{M}^1(\J )$ is called
$T$--\emph{invariant}\index{T-invariance!of a measure} if $\mu'\circ
T^{-1}=\mu'$. If $\mu'\in\mathcal{M}^1_m(\J ) $, then, in terms of the
fiber measures, clearly $T$--invariance equivalently means that the
family $\{\mu_x'\}_{x\in X}$ is $T$-invariant; see
Chapter~\ref{sec:formulation-theorems} for the definition of
$T$-invariance of a family of measures.\index{T-invariance!of a family
  of measures}

Fix $\varphi\in \cH_m^1(\J)$. Then the general expanding random map
$T:\J\to\J$ is called \emph{a\index{measurable expanding random map}
  measurable expanding random map} if the following conditions are
satisfied.
\subsubsection*{\index{measurability!of the transfer
    operator}Measurability of the Transfer Operator}
The transfer operator is \emph{measurable} i.e. $\cL g\in \cC^0_m(\J)$ for every $g\in
\cC^0_m(\J)$.
\subsubsection*{\index{integrability of the logarithm of the transfer
    operator}Integrability of the Logarithm of the Transfer Operator}
The function $X\ni x\mapsto \log\|\cL_x\1_x \|_\infty$ belongs to
$L^1(m)$.

\

We shall now provide a simple, easy to verify, sufficient condition
for integrability of the logarithm of the transfer operator.
\begin{lem}
  \label{lem:6}
  If $\log(\deg(T_x)) \in L^1(m)$, then $x\mapsto  \log\|\cL_x\1_x
  \|_\infty$ belongs to $L^1(m)$.
\end{lem}
\begin{proof}
  Recall that
  \begin{displaymath}
   e^{-\|\vp_x \|_\infty}\leq \sum_{T_x(z)=w}e^{\vp_x(z)} \leq
  \deg(T_x) e^{\|\vp_x \|_\infty}.
  \end{displaymath}
    Hence
  $-\|\vp_x \|_\infty\leq \log\|\cL_x\1_x
  \|_\infty  \leq \log (\deg(T_x)) + \|\vp_x \|_\infty .$
\end{proof}

\section{Measurability}
\label{sec:measurability}
Now, we assume that $T:\J\to\J$ is a measurable expanding random
map. In particular, the operator $\cL$ is measurable. Armed with these
assumptions, we come back to the families of Gibbs states
$\{\nu_x\}_{x\in X}$ and $\{\mu_x\}_{x\in X}$ whose pointwise
construction was given in Theorem~\ref{thm:Gib50A}. Since we have
already established good convergence properties, especially the
exponential decay of correlations, it will follow rather easily that
these families form in fact conditional measures of some measures
$\nu$ and $\mu$ from $\cM^1_m(\J)$. As an immediate consequence of
item (\ref{item:3}) of Theorem~\ref{thm:Gib50A}, we get that the
probability measure $\mu$ is invariant under the action of the map
$T:\J\to\J$.  All of this is shown in the following lemmas.

\begin{lem}
  \label{lem:5}
  For every $g\in L^1_m(\J)$, the map $x\mapsto \nu_x(g_x)$ is
  measurable.
\end{lem}

\begin{proof}
  It follows from \eqref{eq:nuform} that
  \begin{displaymath}
    \lim_{n\to\infty}\frac{\|\cL_x^n g_x\|_\infty}{\|\cL_x^n \1 \|_\infty}
    =\nu_x(g_x).
  \end{displaymath}
  Then measurability of $x\mapsto \nu_x(g_x)$ is a direct consequence
  of measurability of the transfer operator.
\end{proof}
This lemma enables us to introduce the probability measure $\nu$ on
$\J$ given by the
formula
\begin{displaymath}
  \nu(g)=\int_X\int_{\J_x} g_x d\nu_x dm(x).
\end{displaymath}
This measure, therefore, belongs to $\cM_m^1(\J)$.

\begin{lem}
  \label{lem:measurability}
  The map $X\ni x\mapsto \lambda_x\in \R$ is measurable and the
  function $q:\J\ni (x,y)\mapsto q_x(y)$ belongs to $L^0_m(\J)$.
\end{lem}
\begin{proof}
  Since $\nu\in\cM_m^1(\J)$, measurability of $\lambda$'s follows from
  the formula \eqref{eq:lambdaform} and measurability of the transfer
  operator. Then measurability of $\lambda$'s
  and of the transfer operator together with $\lim_{n\to\infty} \tcL _{x_{-n}} ^n \1 =q_x$
  (see Proposition~\ref{4.16}) imply measurability of $q$.
\end{proof}

From this lemma and Lemma~\ref{lem:5} it follows that we can define
a measure $\mu$ by the formula
\begin{equation}
  \label{eq:18}
  \mu(g)=\int_X\int_{\J_x} q_x g_x d\nu_x dm(x).
\end{equation}

\section{The expected pressure}

The pressure function of a measurable expanding random map has the
following important property.
\begin{lem}
  \label{lem:8}
  The pressure function $X\ni x\mapsto P_x(\varphi)$ is integrable.
\end{lem}
\begin{proof}
  It follows from the definition of the transfer operator, that
  \begin{equation}
    \label{eq:11}
    -\|\varphi_x\|_\infty\leq \log \nu_{\shift(x)}(\cL_x\1)
    \leq \log \|\cL_x\1\|_\infty.
  \end{equation}
  Then, by \eqref{eq:lambdaform} and integrability of the logarithm of
  the transfer operator, the function $P_x(\varphi)$ is bounded above
  and below by integrable functions, hence integrable.
\end{proof}

Therefore, \emph{the \index{expected pressure}expected pressure} of
$\vp$ given by
\begin{displaymath}
  \cE P(\varphi) =\int_X P_x(\varphi) dm(x)
\end{displaymath}
is well-defined.

The equality (\ref{eq:nuform}) yields alternative formulas for the
expected pressure. In order to establish them, observe that by
Birkhoff's Ergodic Theorem
\begin{equation}
  \label{eq:defEPbylambda}
  \cE P(\varphi)=\lim_{n\to\infty} \frac{1}{n}\log \lambda_x^n \quad for \; a.e. \;\; x\in X.
\end{equation}
In addition, by (\ref{eq:lambdaform}),
$
  \lambda_x^n=\lambda_x^n\nu_x(\1 )=\nu_{\shift ^n(x)}(\cL_x^n(\1 )).
$
Thus, it follows that
\begin{displaymath}
  \frac{1}{n}\log \lambda_x^n=\lim_{k\to \infty}\frac{1}{n}
  \log \frac{\cL_x^{k+n} \1_{x}(w_{k+n})}
  {\cL_{\shift^n(x)}^{k} \1_{{\shift^n(x)}}(w_{k+n})}.
\end{displaymath}
However, by Lemma~\ref{lem:11} we can get even more interesting formula.
\begin{lem}
  \label{lem:FormOfPressPart}
  For every $\varphi\in \holder$ and for almost every $x\in X$
\begin{displaymath}
  \cE P(\varphi)=\lim_{n\to\infty}\frac{1}{n}\log\cL_x^{n} \1 (w_{n})
\end{displaymath}
where the points $w_{n}\in  \cJ_{\shift^{n}(x)}$ are arbitrarily chosen.
\end{lem}

\section{Ergodicity of $\mu$}

\begin{prop}
  The measure $\mu$ is ergodic.
\end{prop}
\begin{proof}
  Let $B$ be a measurable set such that $T^{-1}(B)=B$ and, for  $x\in X$, denote by
  $B_x$ the set $\{y\in \cJ_x:(x,y)\in B\}$. Then we have that
  $T_x^{-1}(B_{\shift (x)})=B_x$. Now let
  \begin{displaymath}
    X_0:=\{x\in X:\mu_x(B_x)>0\}.
  \end{displaymath}
  This is clearly a $\shift$-invariant subset of $X$.
  We will show that, if $m(X_0)>0$, then $\mu_x (B_x)=1$ for a.e. $x\in X_0$.
   Since $\shift$ is ergodic with respect
  to $m$, this implies ergodicity of $T$ with respect to $\mu$.

  Define a function $f$ by $f_{x}:=\1_{B_x}$. Clearly $f_x\in
  L^1(\mu_x)$ and $f_{\shift^n(x)}\circ T_x^n=f_x$ $m$--a.e.
  Let $x\in X'\cap X_0$, where $X'$ is given by Proposition~\ref{prop:EDC}.
  Let $g_x$ be a function from $L^1(\cJ_x)$ with $\int g_x d\mu_x=0$.
  Then using (\ref{eq:2}) we obtain that
  \begin{displaymath}
     \lim_{n\to\infty}\mu_x\big((f_{\shift^n(x)}\circ T_x^n)g_x\big)\to 0.
  \end{displaymath}
    Consequently
    \begin{displaymath}
          \int_{B_x}g_x\, d\mu_x=0.
    \end{displaymath}
  Since this holds for every mean zero function $g_x\in L^1(\cJ_x)$ ,
  we have that $\mu_x(B_x)=1$ for every $x\in X'\cap X_0$.  This
  finishes the proof of ergodicity of $T$ with respect to the measure
  $\mu$.
\end{proof}

A direct consequence of Lemma~\ref{lem:Gibbs211} and ergodicity of $T$
is the following.
\begin{prop}
  The measure $\mu\in\cM_m^1(\J)$ is a unique $T$-invariant measure satisfying
  (\ref{4.5}).
\end{prop}

\section{Random Compact Subsets of Polish Spaces}\index{random compact
  subsets of Polish spaces}
\label{sec:CRS}
Suppose that $(X,\cF,m)$ is a complete measure space. Suppose
also that $(Y,\varrho )$ is a Polish space which is normalized so that
$\diam (Y)=1$. Let $\cB_Y$ be the $\sg$--algebra of Borel subsets of
$Y$ and let $\cK_Y$ be the space of all compact subsets of $Y$
topologized by the Hausdorff metric. Assume that a measurable mapping
$X\ni x\mapsto \cJ_x\in \cK_Y$ is given.

Following Crauel \cite[Capter 2]{Cra02}, we say that a map $X\ni x\mapsto
Y_x\subset Y$ is \emph{measurable} if for every $y\in
Y,$ the map $x\mapsto d(y,Y_x)$ is measurable, where
\begin{displaymath}
  d(y,Y_x):=\inf\{d(y,y_x):y_x\in Y_x\}.
\end{displaymath}
This map is also called \emph{a random set}. If every $Y_x$ is closed
(res. compact), it is called \emph{a closed (res. compact) random
  set}.  With this terminology $X\ni x\mapsto \J_x\subset Y$ is a
compact random set (see \cite[Remark 2.16, p. 16]{Cra02}).

Closed random sets have the following important properties
(cf. \cite[Proposition 2.4 and Theorem 2.6]{Cra02}).
\begin{thm}
  \label{thm:Pre11}
  Suppose that $X\ni x\mapsto Y_x$ is a closed random set such that
  $Y_x\neq\emptyset$.
  \begin{enumerate}
  \item[(a)] For all open sets $V\subset Y$, the set $\{x\in X:Y_x\cap
    V\neq\emptyset\}$ is measurable.
  \item[(b)] The set
    $
      \cJ:=\graph(x\mapsto Y_x):=\{(x,y_x):x\in X\textrm{ and
      }y_x\in Y_x\}
   $
    is a measurable subset of $X\times Y$ i.e. $\J$ is a subset of
    $\mathcal{F}\otimes\mathcal{B}_Y$, the product $\sigma$-algebra of
    $\mathcal{F}$ and $\mathcal{B}_Y$.
  \item[(c)] For every $n$, there exists a measurable function $X\ni
    x\mapsto y_{x,n}\in Y_x$ such that
  \begin{displaymath}
    Y_x=\cl\{y_{x,n}:n\in \N\}.
  \end{displaymath}
  In particular, there exists a measurable map $X\ni x\mapsto
  y_x\in Y_x$.
  \end{enumerate}
\end{thm}
Note that item (b) implies that $\cJ$ is a measurable subset of
$X\times Y$. Let
$\mathcal{B}_\cJ:=\mathcal{F}\otimes\mathcal{B}_Y|_\cJ$.  Then by
Theorem 2.12 from \cite{Cra02} we get that for all $A\in\cB_\cJ$,
$\pi_X(A)\in \cF$.

Now, let $X\ni x\mapsto Y_x$ be a compact random set and let $r>0$ be
a real number. Then every set $Y_x$ can be covered by some finite
number $a_x=a_x(r)\in\N$ of open balls with radii equal to
$r$. Moreover, by Lebesgue's Covering Lemma, there exits
$R_x=R_x(r)>0$ such that every ball $B(y_x,R_x)$ with $y_x\in Y_x$ is
contained in a ball from this cover. As we prove below, we can
actually choose $a_x$ and $R_x$ in a measurable way. Hence for the compact random set
$x\mapsto \J_x$ the  measurability
of cardinality of covers (see Chapter~\ref{sec:formulation-theorems},
just before Theorem~\ref{thm:1}) holds automatically.

In the proof of Lemma~\ref{lem:Pre13} we will use the following Proposition~2.1
from \cite[p. 15]{Cra02}.
\begin{prop}
  \label{prop:Pre12}
  For compact random set $x\mapsto Y_x$ and for every $\varepsilon$,
  there exists a (non-random) compact set $Y_\varepsilon\subset Y$
  such that
  \begin{displaymath}
    m(\{x\in X:Y_x\subset Y_\varepsilon\})\geq 1-\varepsilon.
  \end{displaymath}
\end{prop}

\begin{lem}
  \label{lem:Pre13}
  There exists a measurable set $X'_a\subset X$ of full measure $m$
  such that, for every $r>0$ and every positive integer $k$, there
  exists a measurable function $X'_a\ni x\mapsto y_{x,k}\in Y_x$ and
  there exist measurable functions $X'_a\ni x\mapsto a_x\in\N$ and
  $X'_a\ni x\mapsto R_x\in \R_+$ such that for every $x\in X'_a$,
  \begin{displaymath}
    \bigcup_{k=1}^{a_x}B_x(y_{x,k},r)\supset Y_x,
  \end{displaymath}
  and for every $y_x\in Y_x$, there exists $k=1,\ldots,a_x$ for which
$
    B_x(y_x,R_x)\subset B_x(y_{x,k},r).
 $
\end{lem}

\begin{proof}
  For $n\in \N$ let $Y_{1/n}\subset Y$ be a compact set given by
  Proposition~\ref{prop:Pre12}. Then the set $X_n:=\{x\in X:Y_x\subset
  Y_{1/n}\}$ is measurable and has the measure $m(X_n)$ greater or
  equal to $1-1/n$. Define
  \begin{displaymath}
    X_a':=\bigcup_{n\in\N}X_n.
  \end{displaymath}
  Then $m(X'_a)=1$.

  Let $\{y_n:n\in\N_+\}$ be a dense subset of
  $Y$. Since $Y_{1/n}$ is compact, there exists a positive integer
  $a(n)$ such that
  \begin{equation}
    \label{eq:Pre17}
    \bigcup_{k=1}^{a(n)}B(y_k,r/2)\supset Y_{1/n}.
  \end{equation}
  Define a function $X'_a\ni x\mapsto a_x$, by $a_x=a(n)$ where
  $
    n:=\min\{k:x\in X_k\}.
 $
  The measurability of $X_n$ gives us the required measurability of
  $x\mapsto a_x$.

  Let $\{y_k:k\in \N\}$ be a countable dense set of $Y$ and
  $m\in\N$. For every $k\in \N$ define a function $x\mapsto G_{x,k}$
  by
  \begin{displaymath}
    G_{x,k}=\left\{
      \begin{array}{ll}
        \overline{B}(y_k,r/2) & \textrm{ if } Y_x\cap B(y_k,r/2)\neq \emptyset\\
        Y_x & \textrm{ otherwise.}
      \end{array}
    \right.
  \end{displaymath}
  Since, by Theorem~\ref{thm:Pre11} (a), the set
  $
    \{x\in X:Y_x\cap
    B(y_k,r/2)\neq \emptyset\}
 $
  is measurable, it follows that $X\ni x\mapsto G_{x,k}$ is a closed
  random set. Hence, by Theorem~\ref{thm:Pre11} (c), there exists a
  measurable selection $X\ni x\mapsto y_{x.k}\in G_{x,k}$. Note that,
  if $y_{x.k}\in \overline{B}(y_k,r/2)$, then $B(y_k,r/2)\subset
  B(y_{x.k},r)$. Therefore, by (\ref{eq:Pre17}),
    \begin{displaymath}
      \bigcup_{k=1}^{U_x}B(y_{x,k},r)\supset Y_{1/n}\supset Y_x  \quad \text{for all}\;\; x\in X_n\, .
  \end{displaymath}
   Finally, for $x\in X_n$, let $R_x>0$ be a real number such that, for
  $y\in Y_{1/n}$, there exists $k=1,\ldots,U(n)$ for which
 $
    B(y, R_x)\subset B(y_k,r/2)\subset B(y_{x,k},r).
  $
  Then $X'_U\ni x\mapsto R_x\in\R_+$ is also measurable.
\end{proof}

%%% Local Variables:
%%% mode: latex
%%% TeX-master: "RDSmain"
%%% End:

%% file: RDSsection5.tex
\chapter{Fractal  Structure of Conformal  Expanding Random
  Repellers} \label{ch:section 5}

We now deal with \emph{conformal expanding random maps}. We prove an
appropriate version of Bowen's Formula, which asserts that the
Hausdorff dimension of almost every fiber $\J_x$, denoted throughout
the paper by $\HD$, is equal to a unique zero of the function
$t\mapsto \cE P(t)$. We also show that typically Hausdorff and packing
measures on fibers respectively vanish and are infinite. A simple
example of such a phenomenon is a Random Cantor Set described.

Later in this paper the
reader will find more refined and general examples of Random Conformal
Systems notably Classical Random Expanding Systems, Br\"uck and
B\"urger Polynomial Systems and DG-Systems.

In the following we suppose that all the fibers $\cJ_x$ are in an ambient space $Y$ 
which is a smooth Riemannian manifold. We will deal with $C^{1+\alpha}$--conformal
mappings $f_x$ and denote then $|f_x'(z)|$ the norm of the derivative of $f_x$
which, by conformality, is nothing else than the similarity factor of $f_x'(z)$.
Finally, let
$||f_x'||_\infty$ be the supremum of $|f_x'(z)|$ over $z\in \cJ_x$.
Since we deal with expanding systems we have
\beq\label{5.1}
|f_x'| \geq \gamma_x \; \textrm{ for a.e. } \; x\in X .
\eeq

\bdfn\label{conformal RDS} Let
$f:(x,z)\mapsto (\shift (x) , f_x(z))$
be a measurable expanding random map having fibers $J_x\subset Y$ and such that the mappings
$f_x:\cJ_x\to \cJ_{\shift (x)}$ can be extended to a neighborhood
of $\cJ_x$ in $Y$ to conformal $C^{1+\alpha}$ mappings.  If in
addition $\log ||f_x'||_\infty \in L^1 (m)$  then we call $f$
\index{conformal!expanding random map}\emph{conformal expanding random
  map}.

A conformal random map $f:\J\to \J$ which is uniformly
expanding is called \emph{\index{conformal!uniformly
    expanding map}conformal uniformly expanding}.
    \edfn
%****************************************************************** End Conformal RDS

\section{Bowen's Formula}
For every $t\in\R$ we consider the potential
$
\varphi _t (x,z)= -t\log|f_x'(z)|
$.
The associated topological pressure $P(\varphi_t)$ will be denoted $P(t)$. Let
$$
\cE P(t)=\int_X P_x(t)dm(x)
$$
be its expected value with respect to the measure $m$. In view of (\ref{5.1}),
it follows from Lemma \ref{4.25}
that the function $ t\mapsto \cE P(t)$ has a unique zero. Denote it by $h$. The
result of this subsection is the following version of Bowen's formula, identifying
the Hausdorff dimension of almost all fibers with the parameter $h$.

\begin{thm}[Bowen's Formula]\index{Bowen's Formula}
  \label{thm:Hau100} Let $f$ be a conformal expanding random
  map. The parameter $h$, i.e. the zero of the function $ t\mapsto\cE P(t)$,
  is $m$-a.e. equal to the Hausdorff dimension $\HD(\J_x)$ of the
  fiber $\cJ_x$.
\end{thm}

Bowen's formula has been obtained previously in various settings
first by
 Kifer \cite{Kif96} and then by Crauel and Flandoni \cite{CF98},
Bogensch\"utz and Ochs \cite{BogOch99} and Rugh \cite{Rug07}.

\begin{proof}
  Let $(\nu_{x,h})_{x\in X}$ be the measures produced in
  Theorem~\ref{thm:Gib50A} for the potential $\varphi_h$. Fix $x\in X$
  and $z\in \J_x$ and set again $y=(x,z)$. For every $r\in(0,\xi]$ let
  $k=k(z,r)$ be the largest number $n\ge 0$ such that
    \begin{equation}
    \label{eq:Bow5}
  B(z,r)\subset f_y^{-n}(B(f_x^n(z),\xi)).
\end{equation}
By the expanding property this inclusion holds for all $0\le n\le k$
and $\lim_{r\to 0}k(z,r)=+\infty$. Fix such an $n$.
By Lemma~\ref{lem:Gibbs},
\begin{equation}
  \label{eq:Bow10}
  \nu_{x,h}(B(z,r))\leq\nu_{x,h}(f_y^{-n}(B(f_x^n(z),\xi)))
  \leq \exp\(hQ_{\shift ^n(x)}\)|(f_x^n)'(z)|^{-h}\exp(-P_x^n(h)).
\end{equation}
On the other hand,
$
  B(z,r)\not\subset f_y^{-(s+1)}(B(f_x^{s+1}(z),\xi))
$
for every $s\ge k$. But, since by Lemma~\ref{lem:l1rds13b},
\begin{displaymath}
  B(z,\exp(-Q_{\shift^{s+1}(x)}\xi^\alpha)|(f_x^{s+1})'(z)|^{-1}\xi)
  \subset f_y^{-(s+1)}(B(f^{s+1}_x(z),\xi)),
\end{displaymath}
we get
\begin{equation}\label{5.3a}
  \exp\(-Q_{\shift^{s+1}(x)}\xi^\alpha\)|(f_x^{s+1})'(z)|^{-1}\xi\le r
\end{equation}
and
$  |(f_x^{s})'(z)|^{-1}\le\xi^{-1}\exp\(Q_{\shift^{s+1}(x)}\xi^\alpha\)r.
$
Inserting this to (\ref{eq:Bow10}) we obtain,
\begin{multline}
  \label{eq:Con90}
  \nu_{x,h}(B(z,r))
  \leq \xi^{-h}\exp\(hQ_{\shift ^n(x)}\)\exp\(hQ_{\shift^{s+1}(x)}\xi^\alpha\)
      r^h\\
      \cdot\exp(-P_x^n(h))|\(f_{\shift ^n(x)}^{s+1-n}\)'(f_x^n(z))|^h.
\end{multline}
or, equivalently,
\begin{multline}
  \label{eq:Bow20}
  \frac{\log \nu_{x,h}(B(z,r))}{\log r}
  \geq h+\frac{hQ_{\shift^n(x)}}{\log r}
  +\frac{hQ_{\shift^{s+1}(x)}\xi^\alpha}{\log r}
  +\frac{-h\log\lt(\lt|\(f_{\shift ^n(x)}^{s+1-n}\)'(f_x^n(z))\rt|\rt)}{\log r}\\
  +\frac{-h\log \xi}{\log r}
  +\frac{-P_x^n(h)}{\log r}.
\end{multline}
Our goal is to show that
\begin{displaymath}
  \liminf_{r\to 0}\frac{\log \nu_{x,h}(B(z,r))}{\log r}\geq h \; \textrm{ for a.e. }\;
  x\in X \; \textrm{ and all } \; z\in \cJ_x.
\end{displaymath}
Since the function $x \mapsto Q_x$ is measurable and almost everywhere
finite, there exists $M>0$ such that $m(A)>0$, where $A=\{x\in X:
Q_x\le M\}.$ Fix $n=n_k\ge 0$ to be the largest integer less than or
equal to $k$ such that $\shift^n(x)\in A$ and $s=s_k$ to be the least
integer greater than or equal to $ k$ such that $\shift^{s+1}(x)\in
A$. It follows from Birkhoff's Ergodic Theorem that
$
\lim_{k\to\infty}{s_k/ n_k}=1.
$
Of course if for $k\geq 1$ we take any $r_k>0$ such that $k(z,r_k)=k$,
then $\lim_{k\to\infty}r_k=0$.

Now, note that by (\ref{eq:Bow5}), the formula
\begin{displaymath}
  f_y^{-n}(B(f_x^n(z),\xi))\subset
  B(z,\exp(Q_{\shift^{n}(x)}\xi^\alpha)|(f_x^n)'(z)|^{-1}\xi)
\end{displaymath}
yields
$
  r\leq \exp(Q_{\shift^{n}(x)}\xi^\alpha)|(f_x^n)'(z)|^{-1}\xi.
$
Equivalently,
\begin{displaymath}\label{5.2}
  -\log r\geq \log |(f_x^n)'(z)|-\xi^\alpha Q_{\shift^{n}(x)}-\log \xi.
\end{displaymath}
Since $\log|(f_x^n)'(z)|\ge \log\g_x^n$ and since the function $x\mapsto\log\g_x$
is integrable and  $$\chi=\min \{1, \int\log\g \, dm\}>0$$ we get from Birkhoff's
Ergodic Theorem that for a.e. $x\in X$ and all $r>0$ small enough
(so $k$ and $n_k$ and $s_k$ large enough too)
\begin{equation}\label{3rds115}
-\log r\ge {\frac{\chi}{2}}n\ge {\frac{\chi}{3}}s.
\end{equation}
Remember that $\shift^n(x)\in A$ and $\shift^{s+1}(x)\in A$. We thus obtain from
(\ref{eq:Bow20}) that
\begin{equation}\label{1rds115}
  \liminf_{r\to 0}\frac{\log \nu_{x,h}(B(z,r))}{\log r}
  \geq h
  - 3h\limsup_{k\to\infty}{\frac{1}{s}}
  \log\lt(\lt|\(f_{\shift^n(x)}^{s+1-n}\)'(f_x^n(z))\rt|\rt)
  - 2\frac{1}{n}P_x^n(h).
\end{equation}
for a.e. $x\in X$ and all $z\in \cJ_x$. But as $\int P_x(h)dm(x)=0$, we have by
Birkhoff's Ergodic Theorem that
\begin{equation}\label{2rds115}
\lim_{n\to\infty}\frac{1}{n}P_x^n(h)=0.
\end{equation}
Also, since the measure $\mu_h$ is $f$-invariant, it follows from
Birkhoff's Ergodic Theorem that there exists a measurable set $X_0\sbt
X$ such that for every $x\in X_0$ there exists at least one (in fact
of full measure $\mu_{x,h}$) $z_x\in \cJ_x$ such that
$$
\lim_{j\to\infty}\frac{1}{j}\log\lt|\(f_x^j\)'(z_x)\rt|
=\hat\chi
:=\int_\cJ \log|f_x'(z)|d\mu_h(x,z)\in (0,+\infty) .
$$
Hence, remembering that $\shift^n(x)$ and $\shift^{s+1}(x)$ belong to
$A$, we get
\begin{multline*}
  \limsup_{k\to\infty}{\frac{1}{s}}\log\lt(\lt|\(f_{\shift^n(x)}^{s+1-n}\)'
  (f_x^n(z))\rt|\rt)\\\shoveleft{\qquad =
  \limsup_{k\to\infty}{\frac{1}{s}}\lt(\log\lt|\(f_ x^{s+1}\)'(z)\rt|-
  \log\lt|\(f_ x^n\)'(z)\rt|\rt)} \\
  \shoveleft{\qquad= \limsup_{k\to\infty}{\frac{1}{s}}\lt(\log\lt|\(f_
    x^{s+1}\)'(z_x)\rt|-
    \log\lt|\(f_ x^n\)'(z_x)\rt|\rt)}\\
  \shoveleft{\quad\le \limsup_{k\to\infty}{\frac{1}{s}}\log\lt|\(f_
  x^{s+1}\)'(z_x)\rt|-
  \liminf_{k\to\infty}{\frac{1}{s}} \log\lt|\(f_ x^n\)'(z_x)\rt|
  = \hat\chi - \hat\chi = 0 \;.\quad }
\end{multline*}
Inserting this and (\ref{2rds115}) to (\ref{1rds115}) we get that
\begin{equation}\label{1rds117}
\liminf_{r\to 0}\frac{\log \nu_{x,h}(B(z,r))}{\log r}
\geq h.
\end{equation}

Keep $x\in X$, $z\in \cJ_x$ and $r\in(0,\xi]$.
Now, let $l=l(z,r)$ be the least integer $\geq 0$ such that
\begin{equation}
  \label{eq:Bow55}
  f_y^{-l}(B(f_x^l(z),\xi)) \subset B(z,r).
\end{equation}
Then, by Lemma~\ref{lem:Gibbs},
\beq \label{eq:Bow60}
\aligned
  \nu_{x,h}(B(z,r))
&\geq \nu_{x,h}(f_y^{-l}(B(f_x^l(z),\xi)))\\
&\geq D_1(\shift^l(x))\exp\(-Q_{\shift^l(x)}\)|(f_x^l)'(z)|^{-l}\exp(-P_x^l(h)).
\endaligned
\eeq
On the other hand
$
  f_y^{-(l-1)}(B(f_x^{l-1}(z),\xi)) \not\subset B(z,r).
$
But, since
\begin{displaymath}
  f_y^{-(l-1)}(B(f^{l-1}_x(z),\xi))\subset B(y,\exp(Q_{\shift^{l-1}}(x)\xi^\alpha)
  |(f_x^{l-1})'(z)|^{-1}\xi),
\end{displaymath}
we get
\begin{equation}
\label{1rsd28}
  r\le\xi\exp(Q_{\shift^{l-1}(x)}\xi^\alpha)|(f_x^{l-1})'(y)|^{-1}.
\end{equation}
Thus
$
  |(f_x^{l-1})'(z)|^{-1}\ge\xi^{-1} \exp\(-Q_{\shift^{l-1}(x)}\xi^\alpha\)r.
$
Inserting this to (\ref{eq:Bow60}) we obtain,
\begin{multline}
  \label{eq:Con100}
  \nu_{x,h}(B(z,r))
  \geq \xi^{-h}D_1(\shift^l(x))e^{-Q_{\shift^l(x)}}|(f_{\shift^{l-1}(x)})'(f_x^{l-1}(z))|^{-h}\cdot \\
   \exp\(-hQ_{\shift^{l-1}(x)}\xi^\alpha\)r^h\exp(-P_x^l(h)).
\end{multline}
Now, given any integer $j\ge 1$ large enough, take $R_j>0$ to be the least
radius $r>0$ such that $$f_y^{-j}(B(f_x^j(z),\xi))\sbt B(z,r)\, .$$
 Then $l(y,R_j)=j$.
Since the function $Q$ is measurable and almost everywhere finite, and
$\shift$ is a measure-preserving transformation, there exist a set $\Ga\sbt X$ with
positive measure $m$ and a constant $E>0$ such that $Q_x\le E$, $D_1(x)\le E$
and $Q_{\shift^{-1}(x)}\le E$ for all $x\in \Ga$. It follows from Birkhoff's
Ergodic Theorem and ergodicity of the map $\th:X\to X$ that there exists
a measurable set $X_1\sbt X$ with $m(X_1)=1$ such that for every $x\in X_1$
there exists an unbounded increasing sequence $(j_i)_{i=1}^\infty$ such that
$\shift^{j_i}(x)\in\Ga$ for all $i\ge 1$. Formula (\ref{1rsd28}) then yields
$$
-\log R_{{j_i}}
\ge -E\xi^\alpha+\log\xi+\log|(f_x^{j_i-1}(z)|
\ge -E\xi^\alpha+\log\xi+\log\g_x^{j_i-1}
\ge {\frac{\chi}{2}}j_i,
$$
where the last inequality was written because of the same argument as (\ref{3rds115})
was, intersecting also $X_1$ with an apropriate measurable set of measure $1$.
Now we get from (\ref{eq:Con100}) that
\begin{multline*}
\frac{\log \nu_{x,h}\(B(z,R_{j_i})\)}{\log R_{j_i}}
\le h +{\frac{2\log E}{\chi j_i}}-{\frac{2E}{\chi j_i}}-{\frac{2h}{\chi}}
{\frac{1}{j_i}}
    \log||(f_{\shift^{j_i-1}(x)})'||_\infty-{\frac{2h\xi^\a E}{\chi j_i}}\\
     -{\frac{2h\log\xi}{\chi j_i}}-{\frac{2}{\chi}}{\frac{1}{j_i}}P_x^{j_i}(h).
\end{multline*}
Noting that $\int_X P_x(t)dm(x)=0$ and applying Birkhoff's Ergodic
Theorem, we see that the last term in the above estimate converges to
zero. Also ${\frac{1}{j_i}} \log||(f_{\shift^{j_i-1}(x)})'||_\infty$
converges to zero because of Birkhoff's Ergodic Theorem and
integrability of the function $x\mapsto\log||f_x'||_\infty$.  Since all
the other terms obviously converge to zero, we thus get for a.e.
$x\in X$ and all $z\in \cJ_x$, that
$$
\liminf_{r\to 0}\frac{\log \nu_{x,h}(B(z,r))}{\log r}
\le \liminf_{i\to \infty}\frac{\log \nu_{x,h}\(B(z,R_{j_i})\)}{\log R_{j_i}}
\le h.
$$
Combining this with (\ref{1rds117}), we obtain that
$$
\liminf_{r\to 0}\frac{\log \nu_{x,h}(B(z,r))}{\log r}=h
$$
for a.e. $x\in X$ and all $z\in \cJ_x$. This gives that $\HD(\cJ_x)=h$ for
a.e. $x\in X$. We are done. \end{proof}

\section{Quasi-deterministic and essential systems}

We now investigate the fractal structure of the Julia sets and we will see
that the random systems naturally split into two classes depending on the
asymptotic behavior of Birkhoff's sums of the topological pressure $P_x^n(h)$.

\bdfn\label{d1rds193} Let $f$ be a conformal uniformly expanding random map. It is called
\emph{essentially random}\index{essentially random}
if for $m$-a.e. $x\in X$,
\beq\label{cond essential}
\limsup_{n\to\infty}P_x^n(h)=+\infty \
\text{ \rm and } \
\liminf_{n\to\infty}P_x^n(h)=-\infty ,
\eeq
where $h$ is the Bowen's parameter coming from Theorem~\ref{thm:Hau100}.
The map $f$ is called \emph{quasi-deterministic} if for $m$--a.e. $x\in X$ there exists $L_x>0$ such that
\beq\label{cond quasi-det}
-L_x\le P_x^n(h)\le L_x\quad \text{
  for} \;\; m\text{-almost all} \;\; x\in X
  \;\; \text{and all}\;\; n\ge 0\, .
  \eeq
\edfn

\brem\label{r2rds193} Because of ergodicity of the transformation
$\th:X\to X$, for a uniformly conformal random map to be essential it suffices
to know that the condition (\ref{cond essential}) is satisfied for a set of points
$x\in X$ with a positive measure $m$.  \erem

\brem\label{r3rds193}
If the number
$$
\sg^2(P(h))=\lim_{n\to\infty}{\frac{1}{n}}\int\Big(S_n(P(h))\Big)^2dm>0
$$
and if the Law of Iterated Logarithm holds, i.e. if
$$
\aligned -\sqrt{2 \sg^2(P(h))}=
\liminf_{n\to\infty}\frac{P_x^n(h)}{\sqrt{n\log\log n}}
\le \limsup_{n\to\infty}\frac{P_x^n(h)}{\sqrt{n\log\log n}} =
\sqrt{2\sg^2(P(h))} \qquad m-a.e,
\endaligned
$$
then our conformal random map is essential. It is essential even if
only the Central Limit Theorem holds, i.e. if
$$
m\lt(\left\{x\in X: {\frac{P_x^n(h)}{\sqrt
      n}} < r \right\}\rt)\to \frac 1{\sg\sqrt{2\pi}} \int_{-\infty}^r
e^{-s^2/2\sg^2(P(h))}\,ds.
$$
\erem

\brem\label{r4rds193} If there exists a bounded everywhere defined
measurable function $u:X\to\R$ such that $P_x(h)=u(x)-u\circ\th(x)$
(i.e. if $P(h)$ is a coboundary) for all $x\in X$, then our system is
quasi-deterministic.  \erem

 For every $\a>0$ let $\Ha^\a$ refer
to the $\a$-dimensional Hausdorff measure and let $\Pa^\a$ refer to the
$\a$-dimensional packing measure. Recall that a Borel probability measure $\mu$
defined on a metric space $M$ is geometric with an exponent $\a$ if and only
if there exist $A\ge 1$ and $R>0$ such that
$$
A^{-1}r^\a\le\mu(B(z,r))\le Ar^\a
$$
for all $z\in M$ and all $0\le r\le R$. The most significant basic properties of
geometric measures are the following:
\begin{itemize}
\item[(\textsc{gm}1)] The measures $\mu$, $\Ha^\a$, and $\Pa^\a$ are all mutually equivalent with
Radon-Nikodym derivatives separated away from zero and infinity.
\item[(\textsc{gm}2)] $0<\Ha^\a(M),\Pa^\a(M)<+\infty$.
\item[(\textsc{gm}3)] $\HD(M)=h$.
\end{itemize}

 The main result of this section is the following.

\begin{thm}\label{t2rds109}
Suppose $f:\J\to \J$ is a conformal uniformly expanding random map.
\begin{itemize}
\item[(a)] If the system $f:\J\to \J$ is essential, then $\Ha^h(\J_x)=0$
  and $\Pa^h(\J_x)=+\infty$
for $m$-a.e. $x\in X$.
\item[(b)] If, on the other hand, the system $f:\J\to \J$ is
  quasi-deterministic, then for every $x\in X$ $\nu_x^h$ is a geometric measure with exponent $h$
  and therefore (\textsc{gm}1)-(\textsc{gm}3) hold.
\end{itemize}
\end{thm}
\begin{proof}
  Part (a). Remember that by its very definition $\cE P(h)=\int
  P_x(h)dm(x)=0$. By Definition~\ref{d1rds193} there exists a
  measurable set $X_1$ with $m(X_1)=1$ such that for every $x\in X_1$
  there exists an increasing unbounded sequence $(n_j)_{j=1}^\infty$
  (depending on $x$) of positive integers such that
\begin{equation}\label{1rds109}
\lim_{j\to\infty}P_x^{n_j}(h)=-\infty.
\end{equation}
Since we are in the uniformly expanding case, the formula
(\ref{eq:Bow60}) from the proof of Theorem~\ref{thm:Hau100} (Bowen's
Formula) takes on the following simplified form
\begin{equation}\label{1rds111}
\nu_x(B(z,r))\ge D^{-1}r^h\exp\(-P_x^{l(z,r)}(h)\)
\end{equation}
with some $D\ge 1$ and all $z\in \J_x$. Since the map is  uniformly expanding,
for all $j\ge 1$ large enough, there exists $r_j>0$ such that
$l(z,r_j)=n_j$. So disregarding finitely many terms, we may assume
without loss of generality, that this is true for all $j\ge 1$.
Clearly
$
\lim_{j\to\infty}r_j=0.
$
It thus follows from (\ref{1rds111}) that
$$
\nu_{x,h}(B(z,r_j))
\ge D^{-1}r_j^h\exp\(-P_x^{n_j}(h)\)
$$
for all $x\in X_1$, all $z\in \J_x$ and all $j\ge 1$. Therefore,
by (\ref{1rds109}),
\begin{displaymath}
  \begin{aligned}
    \limsup_{r\to 0}{\frac{\nu_{x,h}(B(z,r))}{r^h}}  \ge
    \limsup_{j\to\infty}{\frac{\nu_{x,h}(B(z,r_j))}{r_j^h}}
    \geq D^{-1}
    \limsup_{j\to\infty}\exp\(-P_x^{n_j}(h)\) =+\infty
  \end{aligned}
\end{displaymath}
which implies that $\Ha^h(\J_x)=0$.

\sp The proof for packing measures is similar. By Definition~\ref{d1rds193}
there exists a measurable set $X_2$ with $m(X_2)=1$ such that for every $x\in X_2$
there exists an increasing unbounded sequence $(s_j)_{j=1}^\infty$ (depending on $x$)
of positive integers such that
\begin{equation}\label{3rds111}
\lim_{j\to\infty}P_x^{s_j}(h)=+\infty.
\end{equation}
Since we are in the expanding case, formula (\ref{eq:Con90}) from the proof of
Theorem~\ref{thm:Hau100} (Bowen's Formula), applied with $s=k(z,r)$, takes on the
following simplified form.
\begin{equation}\label{2rds111}
\nu_x(B(z,r))\le Dr^h\exp\(-P_x^{k(z,r)}(h)\)
\end{equation}
with $D\ge 1$ sufficiently large, all $x\in X_2$ and all $z\in \J_x$. By our uniform
assumptions, for all $j\ge 1$ large enough, there exists $R_j>0$ such that
$k(z,R_j)=s_j$. Clearly
$
\lim_{j\to\infty}R_j=0.
$
It thus follows from (\ref{2rds111}) that
$$
\nu_{x,h}(B(z,r_j))
\le DR_j^h\exp\(-P_x^{s_j}(h)\)
$$
for all $x\in X_2$, all $z\in \J_x$ and all $j\ge 1$. Therefore, using (\ref{3rds111}),
we get
$$
        \liminf_{r\to 0}{\frac{\nu_{x,h}(B(z,r))}{r^h}}
\le     \liminf_{j\to\infty}{\frac{\nu_{x,h}(B(z,R_j))}{R_j^h}}
\le    D\liminf_{j\to\infty}\exp\(-P_x^{s_j}(h)\)
        =0.
$$
Thus $\Pa^h(\J_x)=+\infty$. We are done with part (a).

\sp\fr Suppose now that the map $f:\J\to \J$ is quasi-deterministic.
It then follows from Definition~\ref{d1rds193} and (\ref{1rds111}) along
with (\ref{2rds111}), that for every $x\in X$ and for every $r>0$ small enough
independently of $x\in X$, we have.
$$
(L_x D)^{-1}r^h\le\nu_{x,h}(B(y,r))\le L_x Dr^h, \  \  x\in X, \  z\in \J_x.
$$
This means that each $\nu_{x,h}$, $x\in X$, is a geometric measure with
exponent $h$ and the theorem follows. \end{proof}

As a straightforward consequence of this theorem we get a
corollary transparently stating
that essential conformal random systems are entirely new objects,
drastically different from deterministic self-conformal sets.

\bcor\label{cbilipscitz} Suppose that conformal random map $f:\J\to \J$
is essential. Then for $m$-a.e. $x\in X$ the following hold.
\begin{itemize}
\item[(1)] The fiber $\J_x$ is not bi-Lipschitz equivalent to any
deterministic nor quasi-deterministic self-conformal set.
\item[(2)] $\J_x$ is not a geometric circle nor even a piecewise smooth curve.
\item[(3)] If $\J_x$ has a non-degenerate connected component (for
  example if $\J_x$ is connected), then $h=\HD(\J_x)>1$.
\item[(4)] Let $d$ be the dimension of the ambient Riemannian space
  $Y$. Then $\HD(\J_x)<d$.
\end{itemize}
\ecor
\begin{proof} Item (1) follows immediately from Theorem~\ref{t2rds109}(a) and (b3).
Item (3) from Theorem~\ref{t2rds109}(a) and the observation that
$\Ha^1(W)>0$ whenever
$W$ is connected. The proof of (4) is similar.
 Since (3) obviously implies (2), we are done. \end{proof}

%**************************************************** Cantor Example *********************

\section{Random Cantor Set}
\label{sec:random-cantor-set}
\index{random cantor set} Here is a first example of an essentially
random system.  Define
\begin{displaymath}
  f_0(x)=3x (\Mod 1) \textrm{ for }x\in [0,1/3]\cup [2/3,1]
\end{displaymath}
and
\begin{displaymath}
  f_1(x)=4x (\Mod 1) \textrm{ for }x\in [0,1/4]\cup [3/4,1].
\end{displaymath}
Let $X=\{0,1\}^\Z$, $\theta$ be the shift transformation and $m$ be
the standard Bernoulli measure. For
$x=(\ldots,x_{-1},x_0,x_1,\ldots)\in X$ define $f_x=f_{x_0}$,
$f_x^n=f_{\theta^{n-1}(x)}\circ f_{\theta^{n-2}(x)}\circ \ldots \circ
f_{x}$ and
\begin{displaymath}
  \J_x=\bigcap_{n=0}^\infty (f_x^n)^{-1}([0,1]).
\end{displaymath}
The skew product map defined on $\bu_{x\in X}J_x$ by the formula
\begin{displaymath}
  f(x,y)=(\th(x),f_x(y))
\end{displaymath}
generates a conformal random expanding system. We shall show that this
system is essential.  To simplify the next calculation, we define
recurrently:
\begin{eqnarray*}
  \xi_x(1)=\left\{
    \begin{array}{ll}
      3 & \textrm{ if } x_0=0 \\
      4 & \textrm{ if } x_0=1
    \end{array}
\right., & \xi_x(n)=\xi_{\theta^{n-1}(x)}(1)\xi_x(n-1).
\end{eqnarray*}
Consider the potential $\varphi^t$ defined by the formula
$ \varphi^t_x=-t \log \xi_x(1)$.
Then
\begin{displaymath}
  S_n \varphi^t_x= - t \log \xi_x(n).
\end{displaymath}
Let $C_{n}$ be a cylinder of the order $n$ that is $C_n$ is a subset
of $\J_x$ of diameter $(\xi_x(n))^{-1}$ such that $f_x^n|_{C_n}$ is
one-to-one and onto $\J_{\shift^n(x)}$. We can project the measure $m$
on $\J_x$ and we call this measure $\mu_x$. In other words, $\mu_x$ is
such a measure that all cylinders of level $n$ have the measure
$1/2^n$.  Then by Law of Large Numbers for $m$-almost every $x$
\begin{displaymath}
  \lim_{n\to\infty}\frac{\log \mu_x(C_n)}{\log \diam(C_n)}=
  \frac{\log 2}{(1/n)\log \xi_x(n)}=\frac{\log 4}{\log 12}=:h.
\end{displaymath}
Therefore the Hausdorff dimension of $\J_x$ is for $m$-almost every $x$
constant and equal to $h$. Next note that
\begin{equation}
  \label{eq:ex123}
  \frac{\mu_x(C_n)}{\diam(C_n)^h}=\exp(-S_n P_x)
\end{equation}
where
\begin{displaymath}
  P_x:=\log 2 - h \log \xi_x(1).
\end{displaymath}
This will give us the value of the Hausdorff and packing measure. So
let $Z_0,Z_1,\ldots$ be independent random variables, each having the
same distribution such that the probability of $Z_n=\log 2 - h\log 3$
is equal to the probability of $Z_n=\log 2 - h\log 4$ and is equal to
$1/2$. The expected value of $Z_n$, $\cE P$, is zero and its standard
deviation $\sigma>0$. Then the Law of the Iterated Logarithm tells us
that the following equalities
\begin{displaymath}
  \liminf_{n\to \infty }\frac{Z_1+\ldots + Z_n}
  {\sqrt{n \log \log n}}=-\sqrt{2} \sigma \quad \text{and}\quad
  \limsup_{n\to \infty }\frac{Z_1+\ldots + Z_n}
  {\sqrt{n \log \log n}}=\sqrt{2} \sigma
\end{displaymath}
hold with probability one. Then, by (\ref{eq:ex123}),
\begin{displaymath}
  \limsup_{n\to\infty} \frac{\mu_x(C_n)}{\diam(C_n)^h}=\infty\quad
  \textrm{and}\quad
    \liminf_{n\to\infty} \frac{\mu_x(C_n)}{\diam(C_n)^h}=0
\end{displaymath}
for $m$-almost every $x$. In particular, the Hausdorff measure of
almost every fiber $\J_x$ vanishes and the packing measure is
infinite. Note also that the Hausdorff dimension of fibers is not
constant as clearly $\HD(\J_{0^\infty})=\log 2/\log3$, whereas
$\HD(\J_{1^\infty})=\log 2/\log4=1/2$.

%******************************************* End Cantor Example

%%% Local Variables:
%%% mode: latex
%%% TeX-master: "RDSmain"
%%% End:

%% file: RDSsection6.tex
\chapter{Multifractal analysis}
\label{cha:mult-analys}

The second direction of our study of fractal properties of conformal
random expanding maps is to investigate the multifractal spectrum
of Gibbs measures on fibers. We show that the multifractal formalism
is valid.  It seems that it is impossible to do it with a method
inspired by the proof of Bowen's formula since one gets full measure
sets for each real $\alpha$ and not one full measure set $X_{ma}$ such
that for all $x\in X_{ma}$, the multifractal spectrum of the Gibbs
measure on the fiber over $x$ is given by the Legendre transform of a
temperature function which is independent of $x\in X_{ma}$.  In order
to overcome this problem we work out a different proof in which we
minimize the use Birkhoff's Ergodic Theorem and instead we base the
proof on the definition of Gibbs measures and the behavior of the
Perron-Frobenius operator. In this point we were partially motivated
by the approach presented in Falconer's book \cite{Fal97}

Another issue we would like to bring up here is real analyticity of
the multifractal spectrum. We establish it assuming that the system is
uniformly expanding and we apply the real-analitycity results proven
for the expected pressure in the Appendix,
Chapter~\ref{sec:pres-real-analyt}.

\section{Concave Legendre Transform}
\label{sec:conc-legend-transf}

Let $\varphi\in H_m(\J)$ be such that $\cE P(\varphi)=0$. Fix
$q\in\R$. We will not use the function $q_x$ and therefore this will
not cause any confusion. Define auxiliary potentials
\begin{displaymath}
  \varphi_{q,x,t}(y):=q (\varphi_x(y)-P_x(\varphi)) - t\log|f'_x(y)|.
\end{displaymath}
By Lemma~\ref{lem:pre4.25}, the function
$
 (q,t)\mapsto\cE P(q,t):=\cE P(\varphi_{q,t})
$
is convex. Moreover, since $\log|f_x'(y)|\ge \log\g_x>0$, it follows from
Lemma~\ref{4.25} that for every $q\in\R$ there exists a unique
$T(q)\in\R$ such that
\begin{displaymath}
\cE P(\varphi_{q,T(q)})=0.
\end{displaymath}
The function $q\mapsto T(q)$ defined implicitly by this formula is
referred to as \emph{the \index{temperature function}temperature
  function}. Put
$$
\varphi_{q}:=\varphi_{q,T(q)}
$$

By $D_T$ we denote the set of differentiability points of the
temperature function $T$. By convexity of $\cE P$, for
$\lambda\in(0,1)$,
\begin{multline*}
  \cE P(\lambda q_1 + (1-\lambda)q_2,\lambda T(q_1) + (1-\lambda)T(q_2))
  \leq
  \lambda\cE P(q_1,T(q_1))+(1-\lambda)\cE P(q_2,T(q_2))=0.
\end{multline*}
Since $t\mapsto \cE P(\lambda q_1 + (1-\lambda)q_2,t)$ is
decreasing,
\begin{displaymath}
  T(\lambda q_1 + (1-\lambda)q_2)\leq \lambda T(q_1) + (1-\lambda)T(q_2).
\end{displaymath}
Hence the function $q\mapsto T(q)$ is convex and
continuous. Furthermore, it follows from its convexity that the
function $T$ is differentiable everywhere but a countable set, where
it is left and right differentiable.
Define
\begin{displaymath}
  \Leg(T)(\alpha):=\inf_{-\infty<q<\infty}\big(\alpha q + T(q)\big),
\end{displaymath}
where
\begin{displaymath}
  \alpha\in \Dom(L)=\big[\lim_{q\to -\infty}-T'(q^-),
  \lim_{q\to \infty}-T'(q^+)\big].
\end{displaymath}
We call $\Leg$ \emph{\index{concave Legendre transform}the concave
  Legendre transform}. This transform is related to the (classical)
Legendre transform $\Legc$ by the formula
$\Leg(T)(\alpha)=-\Legc(T)(-\alpha)$.  The transform $\Leg$ sends convex
functions to concave ones and, if $q\in D_{T}$, then
\begin{displaymath}
  \Leg(T)(-T'(q))=-T'(q) q+T(q).
\end{displaymath}

\begin{lem}
  \label{lem:MA11}
  Let $q\in D_T$. Then for every $\varepsilon>0$ there exists
  $\delta_\varepsilon>0$, such that, for all $\delta\in(0,
  \delta_\varepsilon)$, we have
  \begin{displaymath}
    \cE P((1+\delta)q,T(q)+(qT'(q)+\varepsilon)\delta)<0
  \end{displaymath}
  and
  \begin{displaymath}
    \cE P((1-\delta)q,T(q)+(-qT'(q)+\e)\d)<0.
  \end{displaymath}
\end{lem}
\begin{proof}
  Since the temperature function $T$ is differentiable at the point
  $q$, we may write
  \begin{displaymath}
    T(q+\delta q)=T(q)+T'(q)\d q+o(\delta).
  \end{displaymath}
  for all $\delta>0$ sufficiently small, say $\delta\in(0,
  \delta_\varepsilon^{(1)})$. So,
$$
T(q)+(qT'(q)+\varepsilon)\delta-T((1+\delta) q)=\e\d+o(\delta)>0.
$$
Then, in virtue of Lemma~\ref{4.25}, we get that
$$
\cE P((1+\delta)q,T(q)+(qT'(q)+\varepsilon)\delta)
<\cE P ((1+\delta) q),T((1+\delta) q))=0,
$$
meaning that the first assertion of our lemma is proved. The second
one is proved similarly producing a positive number
$\delta_\varepsilon^{(2)}$. Setting then
$\d_\e=\min\{\delta_\varepsilon^{(1)}, \delta_\varepsilon^{(2)}\}$
completes the proof.
\end{proof}

\section{Multifractal Spectrum}
\label{sec:level-sets}

Let $\mu$ be the invariant Gibbs measure for $\varphi$ and let $\nu$ be
the $\varphi$-conformal measure. For every $\a\in\R$ define
\begin{displaymath}
  K_x(\alpha):=\Big\{y\in \J_x:d_{\mu_x}(y):=\lim_{r\to 0}\frac{\log
\mu_x(B(y,r))}{\log r}
  =\alpha\Big\}.
\end{displaymath}
and
\begin{displaymath}
  K_{x}'
  :=\Big\{y\in \J_x:\textrm{the limit }
  \lim_{r\to 0}\frac{\log \mu_x(B(y,r))}{\log r}
  \textrm{ does not exist}\Big\}.
\end{displaymath}
This gives us the multifractal decomposition
\begin{displaymath}
  \J_x:=\biguplus_{\a\ge 0} K_x(\alpha) \uplus K_x'.
\end{displaymath}
The multifractal spectrum is the family of functions
$\{g_{\mu_x}\}_{x\in X}$ given by the formulas
\begin{displaymath}
g_{\mu_x}(\alpha):=\HD(K_x(\alpha)).
\end{displaymath}
The function $d_{\mu_x}(y)$ is called the local dimension of the
measure $\mu_x$ at the point $y$. Since for $m$ almost every $x\in X$
the measures $\mu_x$ and $\nu_x$ are equivalent with Radon-Nikodym
derivatives uniformly separated from $0$ and infinity (though the
bounds may and usually do depend on $x$), we conclude that we get the
same set $K_x(\a)$ if in its definition the measure $\mu_x$ is
replaced by $\nu_x$.  Our goal now is to get a "smooth" formula for
$g_{\mu_x}$.

Let $\mu_q$ and $\nu_q$ be the measures for the potential $\varphi_q$
given by Theorem~\ref{thm:Gib50A}. The main technical result of this
section is this.

\bprop\label{p1rds129} For every
$q\in D_T$ there exists a measurable set $X_{ma}\sbt X$ with $m(X_{ma})=1$ and
such that, for every $x\in X_{ma}$, and all $q\in D_T$, we have
$$
g_{\mu_x}(-T'(q))=-qT'(q) + T(q)
$$
\eprop
\begin{proof}
Firstly, by Lemma~\ref{lem:PGibbs}, for every $0<R\leq \xi$ there
exists a measurable function $D_R:X\to (0,+\infty)$ such that for all
$q\in\R$, all $x\in X$, all $y\in \cJ_x$, and all integers $n\geq 0$,
we have
\beq\label{1rds123}
    D_R^{-q^*}(\shift ^n(x))
    \leq
    \frac{\nu_{q,x}(f^{-n}_y(B(f^n(y),R)))}
    {\exp\(q(S_n\varphi(y)-P_x^n(\varphi))\)|(f_x^n)'(y)|^{-T(q)}}
    \leq
    D_R^{q^*}(\shift ^n(x)),
\eeq
where $q^*:=(q,T(q))^*$ as defined in \eqref{eq:20}.
In what follows we keep the notation from the proof of
Theorem~\ref{thm:Hau100}. The formulas (\ref{eq:Bow5}) and
(\ref{eq:Bow55}) then give for every $j\ge l$ and every
$0\le i\le k$, that
\beq\label{2rds123}
\aligned
D_\xi^{-q^*} (\shift ^j(x)))^{-1} &\exp
  \(q(S_j\varphi(y)-P_x^j(\varphi))\)|(f_x^j)'(y)|^{-T(q)}\le \\
  &\le\nu_{q,x}(B(y,r))\le \\
&\le D_\xi^{q^*} (\shift ^i(x)))\exp
  \(q(S_i\varphi(y)-P_x^i(\varphi))\)|(f_x^i)'(y)|^{-T(q)}.
\endaligned
\eeq
By $Q_x$ we denote the measurable function given by
Lemma~\ref{lem:l1rds13b} for the function $-\log|f'|$. Let $X_*$ be an
essential set for the functions $X\ni x\mapsto R_x$, $X\ni x\mapsto
a(x)$, $x\mapsto Q_x$, and $X\ni x\mapsto D_\xi(x)$ with constants
$\hat R$, $\hat a$, $\hat Q$ and $\hat D_\xi$.  Let $(n_j)_1^\infty$
be the positively visiting sequence for $X_*$ at $x$.  Let $X_\cE'$ be
the set given by Lemma~\ref{lem:pre4.25} for potentials $\phi_{q,t}$,
$q,t\in\R^2$. Let
\begin{displaymath}
  X_+':=X_\cE'\cap X'_{+X_*}.
\end{displaymath}
Let us first prove the upper bound on $g_{\mu_x}(-T'(q))$. Fix $x\in
X_+'$. Fix $\e_1>0$.  For every $j\ge 1$ let $\{w_k(x_{n_j}):1\le k\le
a(x_{n_j})\}$ be a $\xi$ spanning set of $\J_{x_{n_j}}$. As $\cE
P(\phi_q)=0$, it follows from Lemma~\ref{4.25} that
$\g:={1\over 2}\cE P(\phi_{q,T(q)+\e_1})<0$. So, in virtue of
Lemma~\ref{lem:pre4.25}, there exists $C\ge 1$ such that
\beq\label{2rds230}
\cL_{\phi_{q,T(q)+\e_1},x}\1(w_k(x_{n_j}))\le Ce^{-\g n_j}
\eeq
for all $j\le 1$ and all $k=1,2,\ld,a(\th^{n_j}(x))\le \hat a$. Now, fix
an arbitrary $\e_2\in\R$ such that $q\e_2\ge 0$. For every integer
$l\ge 1$ let
$$
\aligned
K_x(\e_2,l)=\bigg\{y\in K_x(-T'(q)): &-T'(q)-{1\over 2}|\e_2|
\le {\log\nu_x(B(y,r))\over \log r}
\le -T'(q)+{1\over 2}|\e_2|  \\
&\text{ for all } \  0<r\le 1/l\bigg\}.
\endaligned
$$
Note that
\beq\label{3rds230}
K_x(-T'(q))=\bu_{l=1}^\infty K_x(\e_2,l).
\eeq
Let
$$
\Ga_{n_j}(x)=\lt\{z\in\bu_{k=1}^{a(x_{n_j})}f_x^{-n_j}(w_k(x_{n_j})):K_x(\e_2,l)
\cap f_z^{-n_j}(B(f^{n_j}(z),\xi/2))\ne\es\rt\}.
$$
Then
\beq\label{1rds230}
 K_x(\e_2,l)\sbt \bu_{z\in\Ga_{n_j(x)}}f_z^{-n_j}(B(f^{n_j}(z),\xi/2)).
\eeq
For every $z\in \Ga_{n_j}(x)$, say $z\in f_x^{-n_j}(w_k(x_{n_j}))$,
choose
\begin{displaymath}
  \hat z\in K_x(\e_2,l)\cap
f_z^{-n_j}(B(w_k(x_{n_j}),\xi/2)).
\end{displaymath}
Then $B\(w_k(x_{n_j}),\xi/2\)\sbt B(f^{n_j}(z),\xi)$, and therefore
$$
f_z^{-n_j}(B(w_k(x_{n_j}),\xi/2))
\sbt f_{\hat z}^{-n_j}\(B(f^{n_j}(\hat z),\xi)\).
$$
It follows from this and \eqref{1rds230} that
\beq\label{1rds233}
K_x(\e_2,l)
\sbt \bu_{z\in\Ga_{n_j}(x)}f_{\hat z}^{-n_j}(B(f^{n_j}(\hat z),\xi)).
\eeq
Put
\begin{displaymath}
  r_j^{(1)}(\hat z)=\hat Q^{-1}|(f_x^{n_j})'(\hat z)|^{-1} \  \
  \text{ and } \  \
  r_j^{(2})(\hat z)=\hat Q|(f_x^{n_j})'(\hat z)|^{-1}
\end{displaymath}
We then have
$$
B\(\hat z,r_j^{(1)}(\hat z)\)
\sbt f_{\hat z}^{-n_j}(B(f^{n_j}(\hat z),\xi))
\sbt B\(\hat z,r_j^{(2})(\hat z)\).
$$
Therefore, assuming $j\ge 1$ to be sufficiently large so that the
radii $r_j^{(1)}(\hat z)$ and $r_j^{(1)}(\hat z)$ are sufficiently
small, particularly $\le 1/l$, we get
\begin{displaymath}
  \aligned
  {\log\nu_x\(f_{\hat z}^{-n_j}(B(f^{n_j}(\hat z),\xi))\)\over
    -\log|(f_x^{n_j})'(\hat z)|}
  &\le {\log\nu_x\(B(\hat z),\hat Q^{-1}|(f_x^{n_j})'(\hat z)|^{-1}\)\over
    -\log|(f_x^{n_j})'(\hat z)|} \\
  &\le {\log\nu_x\(B(\hat z),r_j^{(1)}(\hat z))\) \over
    \log(r_j^{(1)}(\hat z))+\log\hat Q}
  \le -T'(q)+|\e_2|.
\endaligned
\end{displaymath}
and
\begin{displaymath}
  \aligned
  {\log\nu_x\(f_{\hat z}^{-n_j}(B(f^{n_j}(\hat z),\xi))\)\over
    -\log|(f_x^{n_j})'(\hat z)|}
  &\ge {\log\nu_x\(B(\hat z),\hat Q|(f_x^{n_j})'(\hat z)|^{-1}\)\over
    -\log|(f_x^{n_j})'(\hat z)|} \\
  &\ge {\log\nu_x\(B(\hat z),r_j^{(2})(\hat z))\) \over
    \log(r_j^{(2})(\hat z))-\log\hat Q}
  \ge -T'(q)-|\e_2|.
  \endaligned
\end{displaymath}
Hence,
$$
|q|\(\log\nu_x\(f_{\hat z}^{-n_j}(B(f^{n_j}(\hat
z),\xi))\)-(T'(q)+|\e_2|)\log|(f_x^{n_j})'\hat (z)|\)\le 0
$$
and
$$
|q|\(\log\nu_x\(f_{\hat z}^{-n_j}(B(f^{n_j}(\hat
z),\xi))\)-(T'(q)-|\e_2|)\log|(f_x^{n_j})'\hat (z)|\)\ge 0.
$$
So, in either case (as $\e_2q>0$),
$$
-q\(\log\nu_x\(f_{\hat z}^{-n_j}(B(f^{n_j}(\hat
z),\xi))\)-(T'(q)-|\e_2|)\log|(f_x^{n_j})'\hat (z)|\)\le 0
$$
or equivalently,
\beq\label{1rds232}
\nu_x^{-q}\(f_{\hat z}^{-n_j}(B(f^{n_j}(\hat z),\xi))\)
      |(f^{n_j})'(\hat z)|^{qT'(q)-\e_2 q} \le 1.
\eeq
 Put $t=-qT'(q)+T(q)+\e_1+\e_2 q$. Using \eqref{1rds232} and
\eqref{2rds230} we can then estimate as follows.

\begin{multline*}
  \!\!\!\!\sum_{z\in\Ga_{n_j}(x)} \diam^{-qT'(q)+T(q)+\e_1+\e_2 q}
  \(f_{\hat z}^{-n_j}(B(f^{n_j}(\hat z),\xi))\) =\\
  \shoveleft{=\!\!\!\!\sum_{z\in\Ga_{n_j}(x)}\!\!\!\!
    \diam^{T(q)+\e_1} \(f_{\hat z}^{-n_j}(B(f^{n_j}(\hat z),\xi))\)
    \diam^{-qT'(q)+\e_2 q}\(f_{\hat z}^{-n_j}(B(f^{n_j}(\hat
    z),\xi))\)} \\
  \shoveleft{\le \sum_{z\in\Ga_{n_j}(x)} (\hat
    Q\xi^{-1})^{t}|(f^{n_j})'(z)|^{-(T(q)+\e_1)}
    (\hat Q\xi)^{-t}|(f^{n_j})'(\hat z)|^{qT'(q)-\e_2 q}} \\
  \shoveleft{= (\hat Q\xi^{-1})^{2t}\sum_{z\in\Ga_{n_j}(x)}
    \exp\(q(S_{n_j}\varphi(z)-P_x^{n_j}(\varphi))-(T(q)+\e_1
    \log|(f_x^{n_j})'(z)|\)\cdot} \\
  \shoveright{\cdot\exp\(q(P_x^{n_j}(\varphi)-S_{n_j}\varphi(z)\)
    |(f^{n_j})'(\hat z)|^{qT'(q)-\e_2 q}} \\
  \shoveleft{\le (\hat Q\xi^{-1})^{2t}e^{q\hat Q_\phi}
    \sum_{z\in\Ga_{n_j}(x)} (\hat
    Q\xi^{-1})^{2t}\sum_{z\in\Ga_{n_j}(x)}
    \exp\(q(S_{n_j}\varphi(z)-P_x^{n_j}(\varphi)) -}\\
  \shoveright{-(T(q)+\e_1)\log|(f_x^{n_j})'(z)|\)
    \exp\(q(P_x^{n_j}(\varphi)-S_{n_j}\varphi(\hat
    z))\)|(f^{n_j})'(\hat z)|^{qT'(q)-\e_2 q}} \\
  \shoveleft{\le (\hat Q\xi^{-1})^{2t}e^{q\hat
      Q_\phi}\sum_{z\in\Ga_{n_j}(x)} (\hat
    Q\xi^{-1})^{2t}\sum_{z\in\Ga_{n_j}(x)}
    \exp\(q(S_{n_j}\varphi(z)-P_x^{n_j}(\varphi)) -}\\
  \shoveright{-(T(q)+\e_1)\log|(f_x^{n_j})'(z)|\) \nu_x^{-q}\(f_{\hat
      z}^{-n_j}(B(f^{n_j}(\hat z),\xi))\)
    |(f^{n_j})'(\hat z)|^{qT'(q)-\e_2 q}} \\
  \shoveleft{\le (\hat Q\xi^{-1})^{2t}e^{q\hat
      Q_\phi}\sum_{z\in\Ga_{n_j(x)}} (\hat
    Q\xi^{-1})^{2t}\sum_{z\in\Ga_{n_j}(x)}
    \exp\(q(S_{n_j}\varphi(z)-P_x^{n_j}(\varphi))-}\\
  \shoveright{-(T(q)+\e_1)\log|(f_x^{n_j})'(z)|\)} \\
  \shoveleft{\le (\hat Q\xi^{-1})^{2t}e^{q\hat
      Q_\phi}\sum_{k=1}^{a(x_{n_j})}
    \cL_{\phi_{q,T(q)+\e_1},x}\1(w_k(x_{n_j}))} \\
  \le  C(\hat Q\xi^{-1})^{2t}e^{q\hat Q_\phi}a(x_{n_j})e^{-\g n_j}
  \le C(\hat Q\xi^{-1})^{2t}e^{q\hat Q_\phi}ae^{-\g n_j}.
\end{multline*}
Letting $j\to\infty$ and looking also at \eqref{1rds233}, we thus
conclude that $\Ha^t(K_x(\e_2,l))=0$. In
virtue of \eqref{3rds230} this implies that
$\Ha^t(K_x(-T'(q)))=0$. Since $\e_1>0$ and $\e_2q>0$ were arbitrary, it
follows that
\beq\label{1RDS45}
g_{\mu_x}(-T'(q))=\HD(K_x(-T'(q)))\le -qT'(q) + T(q).
\eeq
Let us now prove the opposite inequality. For every $s\ge 1$ let $s_-$
be the largest integer in $[0,s-1]$ such that
$\th^{s_-}(x)\in X_*$ and let $s_+$  be the least integer in $[s+1,+\infty)$ such
that $\th^{s_+}(x)\in X_*$. It follows from (\ref{2rds123}) applied with $j=l_+$
and $i=k_-$, that
(\ref{5.3a}) is true with $s+1$ replaced by $k_+$, and (\ref{1rsd28})
is true with $l-1$ replaced by $l_-$, that
$$
{\log\nu_{q,x}(B(y,r))\over \log r}
\le {-q^*\log \hat D_\xi+q\(S_{l_+}\varphi(y)-P_x^{l_+}(\varphi)\)-T(q)\log|(f_x^{l_+})'(y)|
\over \log\xi+\xi^\a \hat Q - \log|(f_x^{l_-})'(y)|}
$$
and
$$
{\log\nu_{q,x}(B(y,r))\over \log r}
\ge {q^*\log \hat D_\xi+q\(S_{k_-}\varphi(y)-P_x^{k_-}(\varphi)\)-T(q)\log|(f_x^{k_-})'(y)|
\over \log\xi-\xi^\a \hat Q  - \log|(f_x^{k_+})'(y)|}.
$$
Hence,
\begin{multline}
  \label{1rds125}
  {\limsup_{r\to 0}} {\log\nu_{q,x}(B(y,r))\over \log r}\le\\
  \le{\limsup_{n\to\infty}}\lt(q{P_x^{n_+}(\varphi)-S_{n_+}\varphi(y)
    \over \log|(f_x^{n_-})'(y)|}\rt)
  +T(q){\limsup_{n\to\infty}}{\log|(f_x^{n_+})'(y)|\over
    \log|(f_x^{n_-})'(y)|}
\end{multline}
and
\begin{multline}
  \label{2rds125}
  {\liminf_{r\to 0}}{\log\nu_{q,x}(B(y,r))\over \log r}\ge\\
  \ge{\liminf_{n\to\infty}}\lt(q{P_x^{n_-}(\varphi)-S_{n_-}\varphi(y) \over
    \log|(f_x^{n_+})'(y)|}\rt)
  +T(q){\liminf_{n\to\infty}}{\log|(f_x^{n_-})'(y)|\over
    \log|(f_x^{n_+})'(y)|}.
\end{multline}
Now, given $\e>0$ and $\d_\e>0$ ascribed to $\e$ according to
Lemma~\ref{lem:MA11}, fix an arbitrary $\d\in(0,\d_\e]$. Set
$$
\phi^{(1)}=\phi_{\e,\d}^{(1)}=\phi_{(1+\d)q,T(q)+(qT'(q)+\e)\d}\exp\(-(1+\d)P(\phi_q)\)
$$
and
$$
\phi^{(2)}=\phi_{\e,\d}^{(2)}=\phi_{(1-\d)q,T(q)+(-qT'(q)+\e)\d}\exp\(-(1+\d)P(\phi_q)\).
$$
Since
$$
\aligned
\cE P(\phi^{(1)})
=\cE P(\phi_{(1+\d)q,T(q)+(qT'(q)+\e)\d}\)+(1+\d)\int P(\phi_q)dm
=\cE P(\phi_{(1+\d)q,T(q)+(qT'(q)+\e)\d}\)
\endaligned
$$
and
$$
\aligned
\cE P(\phi^{(2)})
&=\cE P(\phi_{(1-\d)q,T(q)+(-qT'(q)+\e)\d}\)+(1-\d)\int P(\phi_q)dm\\
&=\cE P(\phi_{(1-\d)q,T(q)+(-qT'(q)+\e)\d}\),
\endaligned
$$
it follows from Lemma~\ref{lem:MA11} and Lemma~\ref{lem:pre4.25},
there exists $=\ka(q,\e,\d)\in(0,1)$ such that for all $k=1,2$, and
all $n\ge 1$ sufficiently large, we have
$
{1\over n}\log \cL_{\phi_x^{(k)}}^n(\1)(w)\leq\log\ka
$
for all $x\in X_+'$ and all $w\in\J_{\th^n(x)}$. Equivalently,
\begin{equation}
  \label{eq:MA124}
  \cL_{\phi_x^{(k)}}^n(\1)(w)\leq \ka^n.
\end{equation}
Now, for all $x\in X_+'$, all $j\ge 1$, all $1\le k\le
a(\th^{n_j}(x)\le \hat a$, and all $z\in f^{-n_j}_x(w_k(x_{n_j}))$, define
\begin{displaymath}
 A(z):=\lt\{y\in f_z^{-n_j}(B(w_k(x_{n_j}),\xi)):
      B(f^{n_j}(y),R)\subset B(w_k(x_{n_j}),\xi)\rt\}.
\end{displaymath}
Note that
\beq\label{2rds219}
\bu_{k=1}^{a(x_{n_j})}\bu_{z\in f^{-n_j}_x(w_k(x_{n_j}))}A(z)=\J(x).
\eeq
Fix any $q\in D_T$ and set
\begin{multline*}
\De_\e=\sup_{0<\d\le\d_\e}\Big\{\max\{((1+\d)q,T(q)+(qT'(q)+\e)\d)^*,
                                 ((1-\d)q,T(q)+(-qT'(q)+\e)\d)^*\}\Big\}.
\end{multline*}
Let $x\in X_+'$. Set
\begin{displaymath}
  M:=\exp\(\hat Q\d(-qT'(q)+T(q)-\e)\).
\end{displaymath}
Then, using \eqref{2rds219}, Lemma~\ref{lem:l1rds13b} (for the
potential $(x,z)\mapsto \log|f_x'(z)|$, \eqref{2rds123}, and
\eqref{eq:MA124}, we obtain
\begin{multline}
\label{1rds219}
\nu_{q,x}\(\{y\in\J_x:\nu_{q,x}\(f^{-n_j}_y(B(f^{n_j}(y),R))\)
\geq |(f_x^{n_j})'(y)|^{-(-qT'(q)+T(q))+\e}\}\)=\\
\shoveleft{=
  \nu_{q,x}\(\{y\in\J_x:\nu_{q,x}\(f^{-n_j}_y(B(f^{n_j}(y),R))\)
  |(f_x^{n_j})'(y)|^{-qT'(q)+T(q)-\e}\ge 1\}\)}\\
\shoveleft{=
  \nu_{q,x}\(\{y\in\J_x:\nu_{q,x}^{\d}\(f^{-n_j}_y(B(f^{n_j}(y),R))\)
  |(f_x^{n_j})'(y)|^{\d(-qT'(q)+T(q)-\e)}\ge 1\}\)}\\
\shoveleft{\le
  \int_{\J_x}\nu_{q,x}^{\d}\(f^{-n_j}_y(B(f^{n_j}(y),R))\)
  |(f_x^{n_j})'(y)|^{\d(-qT'(q)+T(q)-\e)}d\nu_{q,x}(y)}\\
\shoveleft{\le \sum_{k=1}^{a(x_{n_j})}\sum_{z\in
    f^{-n_j}_x(w_k(x_{n_j}))}
  \int_{A(z)}\nu_{q,x}^{\d}\(f^{-n_j}_y(B(f^{n_j}(y),R)))\)}\\
\shoveright{|(f_x^{n_j})'(y)|^{\d(-qT'(q)+T(q)-\e})d\nu_{q,x}(y)}\\
\shoveleft{\le \sum_{k=1}^{a(x_{n_j})}\sum_{z\in
    f^{-n_j}_x(w_k(x_{n_j}))}
  \nu_{q,x}^{\d}\(f^{-n_j}_z(B(w_k(x_{n_j}),\xi)))\)}\\
\shoveright{|(f_x^{n_j})'(z)|^{\d(-qT'(q)+T(q)-\e)}
  M\nu_{q,x}(A(z))}\\
\shoveleft{\le M\sum_{k=1}^{a(x_{n_j})}\sum_{z\in
    f^{-n_j}_x(w_k(x_{n_j}))}
  \!\!\!\!\!\!\!\!\!\!\nu_{q,x}^{\d}\(f^{-n_j}_z(B(w_k(x_{n_j}),\xi)))\)
  |(f_x^{n_j})'(z)|^{\d(-qT'(q)+T(q)-\e)} \cdot} \\
\shoveright{\cdot\nu_{q,x}\(f^{-n_j}_z(B(w_k(x_{n_j}),\xi)))\)}\\
\shoveleft{= M\sum_{k=1}^{a(x_{n_j})}\sum_{z\in
    f^{-n_j}_x(w_k(x_{n_j}))}
  \!\!\!\!\!\!\!\!\!\!\nu_{q,x}^{1+\d}\(f^{-n_j}_z(B(w_k(x_{n_j}),\xi)))\)
  |(f_x^{n_j})'(z)|^{\d(-qT'(q)+T(q)-\e)}}\\
\shoveleft{\le
  MD_\xi^{\De_\e}\!\!\!\sum_{k=1}^{a(x_{n_j})}\!\!\!\sum_{z\in
    f^{-n_j}_x(w_k(x_{n_j}))}\!\!\!\!\!\!\!\!\!\!
  \exp\Big((1+\d)q\big(S_{n_j}\phi(z)-P_x^{n_j}(\phi(z)\big)
  -(1+\d)P_x(\phi_q^{n_j})\Big)}\\
\shoveright{ |(f_x^{n_j})'(z)|^{-(T(q)(1+\d)+\d(qT'(q)-T(q)+\e))}
  \exp(-(1+\delta)\P_x^{n_j}(\phi_q(z)))}\\
\shoveleft{= MD_\xi^{\De_\e}\sum_{k=1}^{a(x_{n_j})}\!\!\!\!\sum_{z\in
    f^{-n_j}_x(w_k(x_{n_j}))}\!\!\!\!\!\!\!\!\!\!
  \exp\Big((1+\d)q\big(S_{n_j}
  \phi(z)-P_x^{n_j}(\phi(z)\big) -(1+\d)P_x(\phi_q^{n_j})\Big)}\\
\shoveright{
  \cdot|(f_x^{n_j})'(z)|^{-(T(q)+(qT'(q)+\e)\d)}\exp(-(1+\delta)
  P_x^{n_j}(\phi_q(z)))}\\
= MD_\xi^{\De_\e}\sum_{k=1}^{a(x_{n_j})}
\cL_{\phi_x^{(1)}}^{n_j}(\1)(w_k(x_{n_j}))\le
MD_\xi^{\De_\e}a\ka^{n_j}.\quad\quad\quad\quad\quad\quad\quad\quad
\end{multline}
Therefore,
\begin{displaymath}
  \sum_{j=1}^\infty\nu_{q,x}\(\{y\in\J_x:\mu_{q,x}\(f^{-n_j}_y(B(f^{n_j}(y),R))\)
         \geq |(f_x^{n_j})'(y)|^{-(-qT'(q)+T(q))+\e}\}\)<+\infty.
\end{displaymath}
Hence, by the Borel-Cantelli Lemma, there exists a measurable set
$\J_{1,\e,x}^q\sbt \J_x$ such that $\nu_{q,x}(\J_{1,\e,x}^q)=1$ and
\begin{multline}
\label{1rds221}
\#\Big\{j\ge 1:\nu_{q,x}\(\{y\in\J_x:\mu_{q,x}\(f^{-n_j}_y(B(f^{n_j}(y),R))\)\\
         \geq |(f_x^{n_j})'(y)|^{-(-qT'(q)+T(q))-\e}\}\)\Big\}<\infty.
\end{multline}
Arguing similarly, with the function $\phi^{(1)}$ replaced by
$\phi^{(2)}$, we produce a measurable set
$\J_{2,\e,x}^q\sbt \J_x$ such that $\nu_{q,x}(\J_{2,\e,x}^q)=1$ and
\begin{multline}
\label{1rds221a}
\#\Big\{j\ge 1:\nu_{q,x}\(\{y\in\J_x:\mu_{q,x}\(f^{-n_j}_y(B(f^{n_j}(y),R))\)\\
         \le |(f_x^{n_j})'(y)|^{-(-qT'(q)+T(q))+\e}\}\)\Big\}<\infty.
\end{multline}

Set
$$
\J_x^q=\bi_{n=1}^\infty\J_{1,1/n,x}^q\cap \J_{2,1/n,x}^q.
$$
Then $\nu_{q,x}(\J_x^q)=1$ and, it follows from \eqref{1rds219} and
\eqref{1rds123}, that for all $y\in \J_x^q$, we have
$$
\lim_{j\to\infty} {q(P_x^{n_j}(\varphi)-S_{n_j}\varphi(y))
    \over \log|(f_x^{n_j})'(y)|} =-qT'(q)
$$
Since $\lim_{n\to\infty}{n_-\over n_+}=1$, it thus follows from \eqref{1rds125}
and \eqref{2rds125} that
\beq\label{mu10.6}
d_{\nu_{q,x}}(y)=-qT'(q)+T(q),
\eeq
and (recall that $\nu_{1,x}=\nu_x$ and $T(1)=0$)
$$
\lim_{r\to 0}{\log\nu_x(B(x,r))\over \log r}=-T'(q)
$$
for all $y\in \J_x^q$. As the latter formula implies that $\J_x^q\sbt
K(-T'(q))$, and as $\nu_{q,x}(\J_x^q)=1$, applying \eqref{mu10.6}, we
get that
$$
g_{\mu_x}(-T'(q))=\HD(K_x(-T'(q)))
\ge \HD(\J_x^q))
=-qT'(q) + T(q).
$$
Combining this formula with \eqref{1RDS45} completes the proof.
\end{proof}

\fr As an immediate consequence of this proposition we get the following
theorem.

\bthm\label{ranma} Suppose that $f(x,z)=(\shift (x),f_x(z))$ is a
conformal random expanding map. Then the Legendre conjugate,
$g:\Range(-T')\to [0,+\infty)$, to the temperature function $\R\ni
q\mapsto T(q)$ is differentiable everywhere except a countable set of
points, call it $D_T^*$, and there exists a measurable set $X_{ma}\sbt
X$ with $m(X_{ma})=1$ such that for every $\a\in D_T^*)$ and every
$x\in X_{ma}$, we have
$$
g_{\mu_x}(\a)=g(\a).
$$
\ethm

\section{Multifractal spectrum for uniformly expanding random maps }
\label{sec:mult-spectr-unif}

 Now, as in Chapter~\ref{sec:derivative}, we assume that we deal
with a conformal uniform random expanding map. In particular, the
essential infimum of $\gamma_x$ is larger than some $\gamma>1$ and
functions $H_x$, $n_\xi(x)$, $j(x)$ are finite.  In addition, we have
that there exist constants $L$ and $c>0$ such that
\begin{equation}
  \label{eq:Pre100MA}
  S_n\varphi_x (y) \leq - n c +L
\end{equation}
for every $y\in \J_x$ and $n$ and $\cE P(\varphi)=0$. With these
assumptions we can get the following property of the function $T$.

\begin{prop}
  \label{thm:MA99}
  Suppose that $f:\J\to \J$ is a conformal uniformly random expanding
  map. Then the temperature function $T$ is real-analytic and for
  every $q$, we have
    \begin{equation}
      \label{eq:MA112}
      T'(q)=\frac{\int_\J\varphi d\mu_q}{\int_\J\log|f'|d\mu_q}<0.
    \end{equation}
\end{prop}
\begin{proof}
The potentials
\begin{displaymath}
\varphi_{q,x,t}(y):=q (\varphi_x(y)-P_x(\varphi)) - t\log|f'_x(y)|.
\end{displaymath}
extend by the the same formula to holomorphic functions
$\C\times\C\ni(q,t) \mapsto \varphi_{q,x,t}(y)$. Since these functions
are in fact linear, we see that the assumptions of
Theorem~\ref{thm:RAN} are satisfied, and therefore the function
$\R\times\R\ni (q,t)\mapsto \cE P(q,t)$ is real-analytic.  Since
$|f'_x(y)|>0$, in virtue of Proposition~\ref{prop:derivative} we
obtain that
  \begin{equation}
    \label{eq:MA102}
    \frac{\partial \cE P(q,t)}{dt}=- \int_\J \log |f'_x| d\mu_{q,x,t}dm(x)<0.
  \end{equation}
Hence, we can apply the Implicit Function Theorem to conclude that the
temperature function $\R\ni q\mapsto T(q)\in\R$, satisfying the
equation,
  \begin{displaymath}
    \cE P({q,T(q)})=0,
  \end{displaymath}
is real-analytic. Hence,
  \begin{displaymath}
    0=\frac{d \cE P(\varphi_{q})}{d q}=
    \frac{\partial \cE P({q,t})}{\partial q}\Big|_{{t=T(q)}}+
    \frac{\partial \cE P({q,t})}{\partial t}\Big|_{t=T(q)}T'(q).
  \end{displaymath}
  Then
  \begin{multline*}
    T'(q)=-\frac{\frac{\partial \cE P({q,t})}{\partial q}\big|_{{t=T(q)}}}
    {\frac{\partial \cE P({q,t})}{\partial t}\big|_{t=T(q)}}
    =-\frac{\int_\J (\varphi_x - P_x) d\mu_{q,x}dm(x)}
    {\int_\J -\log|f'_x| d\mu_{q,x}dm(x)}\\
    =\frac{\int_\J \varphi_x d\mu_{q,x}dm(x)-\int_X P_x dm(x)}
    {\int_\J \log|f'_x| d\mu_{q,x}dm(x)}
    =\frac{\int_\J\varphi d\mu_q}{\int_\J\log|f'|d\mu_q}.
  \end{multline*}
  So, we obtain (\ref{eq:MA112}). It follows, in particular, that
  \begin{equation}
    \label{eq:MA122}
    T'(q)<0,
  \end{equation}
 since by (\ref{eq:Pre100MA}), the integral $\int_\J\varphi d\mu_q$ is negative.
\end{proof}

Combining this proposition with Proposition~\ref{p1rds129} we get
the following result which concludes this section.

\bthm\label{ranma2} Suppose that $f:\J\to \J$ is a conformal uniformly
random expanding map. Then the Legendre conjugate, $g:\Range(-T')\to
[0,+\infty)$, to the temperature function $\R\ni q\mapsto T(q)$ is
real-analytic, and there exists a measurable set $X_{ma}\sbt X$ with
$m(X_{ma})=1$ such that for every $\a\in\Range(-T')$ and every $x\in
X_{ma}$, we have
$$
g_{\mu_x}(\a)=g(\a).
$$
\ethm

%%% Local Variables:
%%% mode: latex
%%% TeX-master: "RDSmain"
%%% End:

%% file: RDSsection7.tex
\chapter{Expanding in the mean}
\label{sec:expanding-mean}

In this chapter we deal with a class of random maps satisfying an
allegedly weaker expanding condition
\begin{displaymath}
  \int \log \gamma_x dm(x)> 0.
\end{displaymath}
We start with a precise definition of this class, and then we explain
how this case can be reduced to random expanding maps by looking at an
appropriate induced map.

\section{Definition of maps expanding in the mean}

Let $T:\J \to\J$ be a skew-product map as defined in
Section~\ref{sec:preliminaries} satisfying the properties of
\emph{Measurability of the Degree} and \emph{Topological
  Exactness}. Such a random map is called \emph{expanding in the mean},
if for some $\xi>0$ and some measurable function $X\ni x\mapsto \gamma_x\in
\R_+$ with
\begin{displaymath}
  \int \log \gamma_x dm(x)>0
\end{displaymath}
we have that all inverse branches of every $T_x^n$ are well defined on
balls of radii $\xi$ and are $(\gamma_x^n)^{-1}$--Lipschitz continuous. More
precisely, for every $y=(x,z)\in \J$ and every $n\in\N$, there exists
\begin{displaymath}
  T^{-n}_y:B_{\shift^n(x)}(T^n(y),\xi)\to \J_x
\end{displaymath}
such that 
\begin{enumerate}
\item $T^n\circ T_y^{-n}=\Id |_{B_{\shift^n(x)}(T^n(y),\xi)}$,
\item $\varrho (T^{-n}_y(z_1),T^{-n}_y(z_2))\leq
  (\gamma^n_x)^{-1}\varrho (z_1 ,z_2 )$
  for all $z_1,z_2 \in B_{\shift^n(x)}\big(T^n(y),\xi\big)$.
\end{enumerate}

\section{Associated induced map}

In this section we show how the expanding in the mean maps can be
reduced to our setting from Section~\ref{exp rds def}. 

Let $T:\J \to\J$ be an expanding in the mean random map.
To this map and to a set $A\subset X$ of positive measure
we associate an induced map $\ov T$ in the following way.
Let $\tau_A$ be the first return
map to the set $A$, that is
\begin{displaymath}
  \tau_A(x)=\min\{n\geq 1:\shift^n(x)\in A\}.
\end{displaymath}
 Define also 
\begin{displaymath}
  \shift_A(x):=\shift^{\tau_A(x)}(x)\textrm{ and }
  \gamma_{A,x}:=\prod_{j=0}^{\tau_A(x)-1}\gamma_{\shift^j(x)}.
\end{displaymath}
Then the \emph{induced map} $\ov T$ is the random map over $(A, \cB , m_A)$
defined by 
\begin{displaymath}\label{induced map}
\ov T_x = T_x ^{\tau _A (x)} \quad \text{for a.e.} \;\; x\in A.
\end{displaymath}

The following lemma show that the set $A$ can be chosen 
such that $\ov T$ is an expanding random map.

\begin{lem}
  \label{lem:ExpMean1}
  There exists a measurable set $A\subset X$ with $m(A)>0$ such that
$$   \gamma_{A,x} >1 \quad \text{for all } x\in A\; .$$
\end{lem}

\begin{proof} 
  First, define inductively 
  \begin{displaymath}
    A_1:=\{x:\log \gamma_x>0\}
  \end{displaymath}
  and, for $k\geq 1$,
  \begin{displaymath}
    A_{k+1}:=\{x\in A_k:\log \gamma_{A_k,x}>0\}.
  \end{displaymath}
  Since 
  \begin{displaymath}
    0<\int_X \log \gamma_x dm(x)=\int_{A_1} \log \gamma_{A_1,x} dm(x)=
    \int_{A_k} \log \gamma_{A_k,x} dm(x),
  \end{displaymath}
we have that $m(A_k)>0$ for all $k\geq 1$.
Obviously, the sequence $(A_k)_{k=1}^\infty$ is decreasing. Let
$$A=\bigcap _{k=1}^\infty A_k \quad \text{and}
\quad E=X\setminus A\;.$$
Notice that the points $x\in E$ have the property that
 $\log \gamma_x^n \leq 0$ for some $n\geq 1$.
 
\

\noindent
Claim: $m(A)>0$. 

\

If on the contrary $m(A)=\lim_{k\to\infty }m(A_k) =0$, then $m(E)=1$.
Since the measure $m$ is $\shift$--invariant, we have that 
$m(E_\infty )=1$ where 
$$E_\infty =\bigcap_{n=0}^\infty \shift ^{-n}(E)\;.$$
For $x\in E_\infty$ we have that $\log \gamma_x^n \leq 0$
for infinitely many $n\geq 1$. This contradicts
Birkhoff's Ergodic Theorem since, by hypothesis,
$\int \log \gamma_x >0$. Therefore the set $A$
has positive measure.

\

Since $m(A)>0$, $\tau_A$ is almost surely finite.
  Now let $x\in A$. Then, for every 
  point $\shift^j(x)$, $j=1,\ldots, \tau_A(x)-1$, we can find $k(j)$ such that
  $\shift^j (x)\in X\sms A_{k(j)}$. Put
  \begin{displaymath}
    K(x)=\max\{k(j):j=1,\ldots, \tau_A(x)-1\}+1.
  \end{displaymath}
  Hence $x$ and $\shift_A(x)$ are in $A_{K(x)}$ and $\shift^j (x)\notin A_{K(x)}$
  for $j=1,\ldots, \tau_A(x)-1$. Hence $\tau_A(x)=\tau_{A_{K(x)}}(x)$, and
  therefore 
  \begin{displaymath}
    \gamma_{A,x}=\gamma_{A_{K(x)},x}>1.  
  \end{displaymath}
\end{proof}

Now we consider an appropriate class of H\"older potentials.  First,
to every $y=(x,z)$ we associate the following neighborhood
\begin{displaymath}
  U(z)=\bigcup_{n=0}^\infty T_y^{-n}\big(B_{\shift^n(x)}\big(T^n(y),\xi\big)\big)\subset \J_x.
\end{displaymath}
Fix $\alpha\in (0,1]$. As in Section~\ref{sec:spacesHol} a
function $\varphi\in \cC^1(\cJ )$ is called \emph{\index{H\"older
    continuous with an exponent $\alpha$}H\"older continuous with an
  exponent $\alpha$} provided that there exists a measurable function
$H: X\to [1,+\infty)$, $x\mapsto H_x$, such that 
\begin{equation}
  \label{eq:6}
  \int_X\log H_x dm(x)<\infty
\end{equation}
and such that 
\begin{displaymath}
  v_\alpha(\varphi_x)\leq H_x\textrm{ for a.e. }x\in X.
\end{displaymath}
The subtlety here is that the infimum in the definition (\ref{eq:4}) of $v_\alpha$
is now taken over all $z_1,z_2\in \cJ_x $ with $z_1,z_2\in U(z)$,
 $z\in \cJ_x$. For example, any function, which is $H_x$--H\"older
over entire $\J_x$ is fine.

Let $T$ be an expanding in the mean random map and $\varphi$
a H\"older potential according to the definition above.
Having associated in \eqref{induced map} to $T$ the induced map $\ov T$,
one naturallt has to replace the potential $\varphi$ by
the \emph{induced potential}
\begin{displaymath}
  \ov{\varphi}_x(z)=\sum_{j=0}^{\tau_A(x)-1}\varphi_{\shift^j(x)}(T_x^j(z)).
\end{displaymath}
Although, it is not clear if the potential $\ov{\varphi}$ satisfies
the condition (\ref{eq:6}), the choice of the neighborhoods $U(z)$
and the definition of H\"older potentials make that
Lemma~\ref{lem:l1rds13b} still holds. This gives us an important
control of the distortion which is what is needed in the rest of
the paper rather than the condition (\ref{eq:6}) leading to it. The
hypothesis (\ref{eq:6}) is only used in the proof of
Lemma~\ref{lem:l1rds13b}.

\section{Back to the original system}

In this section we explain how to get the Thermodynamic Formalism for the
original system.  

With the preceeding notations, for the expanding induced map $\ov T$
the Thermodynamical Formalism of Chapter \ref{ch:RPFmain}
and, in particular, the Theorems \ref{thm:Gib50A} and ~\ref{thm:Gib50B}
do apply. We denote by $\ov \nu _x$, $\ov \mu_x$ and $\ov q_x$, $x\in A$,
the resulting conformal and invariant measures and the 
invariant density respectively for $\ov T$. We now explain
how the corresponding objects can be recovered for the original map $T$.
Notice that this is possible since we only induced in the
base system.

First, we consider the case of the conformal measures.
Let $\ov\nu_x$, $x\in A$ be the measure such that
\begin{displaymath}
  \ov\cL_x^*\ov\nu_{\shift_A(x)}=\ov\lambda_x\nu_x.
\end{displaymath}
If $x\in A$ we put $\nu_x=\ov\nu_x$. If $x\notin A$, then by
ergodicity of $\shift$, almost surely there exists $k\in\N$, such that
$\shift^k(x)\in A$ and, for $j=0,\ldots, k-1$, $\shift^j(x)\notin
A$. Then we put
\begin{equation}
  \label{eq:7}
  \nu_x=\frac{(\cL_x^k)^*\nu_{\shift^k(x)}}{\cL_x^k(\1)}.
\end{equation}
Therefore, the family $\{\nu_x\}_{x\in X'}$ is a family of probability
measures well defined for $X$ in a subset $X'$ of $X$ with full
measure. Then, for $x\in X'$, we put
$\lambda_x=\nu_{\shift(x)}(\cL_x\1)$. It follows from (\ref{eq:7})
that
\begin{displaymath}
  \cL_x^*\nu_{\shift^k(x)}=\lambda_x\nu_x.
\end{displaymath}
It also follows, that $\cE P(\varphi)=\cE P(\ov\varphi)$.

The family $\{\ov{\mu_x}\}_{x\in A}$ of $\ov T$-invariant measures
gives us a family $\{{\mu_x}\}_{x\in X}$ of $T$-invariant measures as
follows. For $x\in A$ and $j=0,\ldots, \tau_A(x)-1$ put
\begin{displaymath}
  \mu_{\shift^j(x)}= \ov{\mu}_x\circ T_x^{-j}.
\end{displaymath}
Then, for $q_{\shift^j(x)}=\cL_x^j(\ov{q}_x)$, we have that 
\begin{displaymath}
  d\mu_{\shift^j(x)}=q_{\shift^j(x)}d\nu_{\shift^j(x)}.
\end{displaymath}

Hence Theorem~\ref{thm:Gib50A} and Theorem~\ref{thm:Gib50B} among with
all statistical consequences hold for the original map. Moreover, since 
$\cE P (\ov\varphi_t)=\cE P(\varphi_t)$ their zeros coincide and consequently
Bowen's Formula and the Multifractal Analysis are also true 
for conformal expanding in the mean random maps.

\section{An example}

Here is an example of an expanding in the mean
random system.  Define
\begin{displaymath}
  f_0(x)=\left\{
    \begin{array}{lcl}
      \frac{1}{2}x+\frac{15}{2}x^2 &\textrm{if } &x\in[0,1/3]\\
      8x-7&\textrm{if } &x\in[7/8,1]\\
    \end{array}
\right.
\end{displaymath}
and
\begin{displaymath}
  f_1(x)=8x (\Mod 1) \textrm{ for }x\in [0,1/8]\cup [7/8,1].
\end{displaymath}
Let $X=\{0,1\}^\Z$, $\theta$ be the shift transformation and $m$ be
the standard Bernoulli measure. For
$x=(\ldots,x_{-1},x_0,x_1,\ldots)\in X$ define 
\begin{displaymath}
  f_x=f_{x_0}\textrm{, }
  f_x^n=f_{\theta^{n-1}(x)}\circ f_{\theta^{n-2}(x)}\circ \ldots \circ
  f_{x}
\end{displaymath}
and
\begin{displaymath}
  \J_x=\bigcap_{n=0}^\infty (f_x^n)^{-1}([0,1]).
\end{displaymath}

If $x_0=0$, then $\gamma_x=1/2$. Otherwise, $\gamma_x=8$. Hence
\begin{displaymath}
  \int \log \gamma_x dm(x)>0.
\end{displaymath}
Note that the size of each component of $f_x^{-n}([0,1])$ is bounded by 
\begin{equation}
  \label{eq:8}
  a_n=8^{-n_1}(1/2)^{-n_0},
\eeq
 where
  $n_i:=\#\{j=0,\ldots,n-1:x_j=i\}$, $i=0,1$.
Since 
\begin{displaymath}
  \lim_{n\to\infty}\frac{n_0}{n}=  \lim_{n\to\infty}\frac{n_1}{n}=\frac{1}{2}
\end{displaymath}
almost surely, we have that $\lim_{n\to\infty}a_n= 0$. Hence, for almost
every $x\in X$, $\J_x$ is a Cantor set. Moreover, by (\ref{eq:8}),
almost surely we have, that,
\begin{displaymath}
  \cE P(t)\leq \lim_{n\to\infty}\frac{1}{n}\log 2^n 8^{-n_1 t}(1/2)^{-n_0 t}
  \leq \log 2 - t \(\lim_{n\to\infty}\frac{n_1}{n}\log 8 - \frac{n_0}{n}\log 2\)=
  \log 2 - t \log 4.
\end{displaymath}
Therefore, by Bowen's Formula,  the Hausdorff dimension of
almost every fiber $\cJ_x$ is smaller than or equal to $1/2$.
Notice however that for some choices of $x\in X$ the fiber $\J_x$
contains open intervals.

%%% Local Variables:
%%% mode: latex
%%% TeX-master: "RDSmain"
%%% End:

%% file: RDSsection8.tex
%\section{Examples}\label{ch:examples}

% In this chapter we continue our investigations of geometric properties
% of conformal random expanding maps. We analyze a rather large class of
% examples. They appear as similarity maps with different contraction
% rates, the so called G-systems, classical random systems, classical
% random conformal systems, complex dynamical systems and, perhaps most
% importantly, Br\"uck and B\"urger polynomial systems. 

\chapter{Classical Expanding Random Systems}
\label{cha:class-expand-rand}
Having treated a very general situation up to here, we now focus on
more concrete random repellers and, in the next section, random maps
that have been considered by Denker and Gordin.  The Cantor example of
Chapter~\ref{sec:random-cantor-set} and random perturbations of
hyperbolic rational functions like the examples considered by Br\"uck
and B\"urger are typical random maps that we consider now. We classify
them into quasi-deterministic and essential systems and analyze then
their fractal geometric properties. Here as a consequence of the
techniques we have developed, we positively answer the question of
Br\"uck and B\"urger (see \cite{BruBug03} and Question~5.4 in
\cite{Bru01}) of whether the Hausdorff dimension of almost all (most)
naturally defined random Julia sets is strictly larger than $1$. We
also show that in this same setting the Hausdorff dimension of almost
all Julia sets is strictly less than $2$.

\section{Definition of Classical Expanding Random Systems}
Let $(Y, \rho)$ be a compact metric space normalized by $diam (Y)=1$
and let $U\subset Y$. A \emph{repeller over $U$} will be a continuous open and surjective map
$T:V_T\to U$ where $\overline{V_T}$, the closure of the domain of $T$, is a subset of $U$.
Let $\ga >1$ and consider
$$\cR =\cR (U,\gamma )=\left\{ T:V_T\to U \;  \;\; \ga \text{--expanding repeller over} \; U\right\}.$$
Concerning the randomness we will consider classical independently and identically distributed (i.i.d.) choices.
 More precisely, we suppose the repellers
\beq \label{rep1}  T_{x_{0}}, T_{x_{1}}, ... , T_{x_{n}},...\eeq
are chosen i.i.d.  with respect to some arbitrary probability
space $(I,\cF_0,m_0)$.
 This gives rise to a \emph{random repeller} $T_{x_0}^n= T_{x_{n-1}}\circ ...\circ T_{x_{0}}$, $n\geq 1$.
 The natural associated Julia set is
$$ \cJ_{x} =\bigcap_{n\geq 1 } T_{x_{0}}^{-n} (U) \quad \text{where} \;\; x=(x_0,x_1,...)\,.$$
Notice that compactness of $Y$ together with the expanding assumption, we recall that $\ga$-expanding means that the distance of
all points $z_1,z_2$ with $\rho (z_1,z_2)\leq \eta _T$ is expanded by the factor $\gamma$,
implies that $\cJ_{x}$ is compact and also that the maps $T\in \cR$ are of bounded degree.
A random repeller is therefore the most classical form of a uniformly expanding random system.

The link with the setting of the preceding sections goes via natural extension. Set $X=I^{\Z}$, take the
Bernoulli measure $m=m_0^{\Z}$ and let the ergodic invariant map $\theta$ be the shift map $\sigma : I^{\Z} \to I^{\Z}$.
If $\pi :X\to I$ is the projection on the $0^{th}$ coordinate and if $x\mapsto T_x$ is a map from $I$
to $\cR$ then the repeller \eqref{rep1} is given by the skew-product
\beq \label{rep2}
T(x,z) = \big(\sigma (x) , T_{\pi (x)} (z)\big) \quad , \quad (x,z)\in \cJ = \bigcup _{x\in X} \{x\} \times \cJ _{x}\, .
\eeq
The particularity of such a map is that the mappings $T_x$ do only
depend on the $0^{th}$ coordinate. It is natural to make the same assumption for the potentials
i.e. $\ph_x = \ph _{\pi (x)}$. We furthermore consider the following  continuity assumptions:
\begin{itemize}
\item[(T0)] $I$ is a bounded metric space.
  \item[(T1)] $(x,z)\mapsto T_x^{-1} (z)$ is continuous from $\cJ$ to $\cK (U)$, the space of all non-empty
   compact subsets of $U$ equipped with the Hausdorff distance.
  \item[(T2)] For every $z\in U$, the map $x\mapsto \ph _x (z)$ is continuous.
\end{itemize}

\emph{A classical expanding random system} is a random repeller together with
a potential depending only on the $0^{th}$--coordinate such that the
conditions (T0), (T1) and (T2) hold.

\begin{exe}
Suppose $V,U$ are open subsets of $\C$ with $V$ compactly contained in $U$ and consider the set $\cR (V,U)$ of all holomorphic repellers
$T:V_T \to U$ having uniformly bounded degree and a domain $V_T\subset V$.
This space has natural topologies, for example the one induced by the distance
$$\rho \big( T_1,T_2\big) = d_H \big(V_{T_1} , V_{T_2} \big) +\|(T_1 -T_2)_{|V_{T_1} \cap V_{T_2}} \|_\infty\, ,$$
 where $d_H$ denotes the
  Hausdorff metric.
Taking then geometric potentials $-t\log |T'|$ we get one of the most natural example of classical expanding random system.
\end{exe}

\bprop \label{press cont}
The pressure function $x\mapsto P_x(\ph )$ of a classical expanding random system is continous.
\eprop

\bpf
We have to show that $x\mapsto \lam_x$ is continuous and since $\cL_x^n\1 (y)/\cL_{x_1}^{n-1} \1(y)$ converges
uniformly to $\lam_x$ for every $y\in U$ (see Lemma~\ref{lem:L590}) it suffices to show that
$x\mapsto \cL^n_x\1 (y)$ does depend continuously on $x\in X$. In order to do so, we first show that
condition (T1) implies continuity of the
  function $(x,y)\mapsto \# T^{-1}_x(y)$.

Let $(x,y)\in X\times U$ and fix $0<\xi'<\xi$ such that $B(w_1,\xi')\cap B(w_2,\xi')=\emptyset $
  for all disjoint $w_1,w_2\in T^{-1}_x(y)$. From (T1) follows that there exists $\delta >0$ such that
  \begin{displaymath} \label{eq hdist..}
    d_H(T^{-1}_{x}(y),T^{-1}_{x'}(y'))\leq \frac{\xi}{2} \quad , \;\; \text{whenever}\;\; \varrho\big((x,y),(x',y')\big)\leq\delta \, .
  \end{displaymath}
 But this implies that for every $w\in T_x^{-1} (y)$ there exists at least one preimage $w'\in T^{-1}_{x'}(y')\cap B(w,\xi')$.
 Consequently $\# T^{-1}_{x'}(y') \geq \# T^{-1}_x(y)$.
  Equality follows since $T_{x'}$ is injective on every ball of radius $\xi'$,
  a consequence of the expanding condition.

  Let $x\in X$, let $W$ be a neighborhood of $x$ and let $y\in U$. From what was proved before we
  have that for every $w\in T_x^{-1}(y)$, there exists a continuous
  function $x'\mapsto z_{w}(x')$ defined on $W$ such that $T_{x'}(z_{w}(x'))=y$,
  $z_w(x)=w$ and
  \begin{displaymath}
    T_{x'}^{-1}(y)=\{z_w(x'):w\in T_x^{-1}(y)\}.
  \end{displaymath}
  The proposition follows now from the continuity of $\varphi_x$, i.e. from (T2).
\epf

%*****************************************************************************************************************************

We say that a function $g:I^\Z\to\R$ is past independent if
$g(\om)=g(\tau)$ for any $\om,\tau\in I^\Z$ with
$\om|_0^\infty=\tau|_0^\infty$. Fix $\ka\in (0,1)$ and for every
function $g:I^\Z\to\R$ set
$$
v_\ka(g)=\sup_{n\ge 0}\{v_{\ka,n}(g)\},
$$
where
$$
v_{\ka,n}(g)=\ka^{-n}\sup\{|g(\om)-g(\tau)|:\om|_0^n=\tau|_0^n\}.
$$
Denote by $\H_\ka$ the space of all bounded Borel measurable functions
$g:I^\Z\to\R$ for which $v_\ka(g)<+\infty$. Note that all functions in
$\H_\ka$ are past independent. Let $\Z_-$ be the set of negative integers.
If $I$ is a metrizable space and $d$ is a bounded metric on $I$, then the formula
$$
d_+(\om,\tau)=\sum_{n=0}^\infty2^{-n}d(\om_n,\tau_n)
$$
defines a pseudo-metric on $I^\Z$, and for every $\tau\in I^\Z$, the
pseudo-metric $d_+$ restricted to $\{\tau\}\times\N$, becomes a metric
which induces the product (Tychonoff) topology on $\{\tau\}\times\N$.

\bthm
\label{t1rds197}
Suppose that $T:\J\to \J$ and $\phi:\J\to\R$ form a classical
expanding random system. Let $\lam:I^\Z\to(0,+\infty)$ be the corresponding
function coming from Theorem~\ref{thm:Gib50A}. Then both functions
$\lam$ and $P(\phi)$ belong to $\H_\ka$ with some $\ka\in (0,1)$, and
both are continuous with respect to the pseudo-metric $d_+$.  \ethm

\begin{proof}
  Let $y\in U$ be any point. Fix $n\ge 0$ and $\om,\tau\in I^\Z$
  with $\om|_0^n=\tau|_0^n$. By Lemma~\ref{lem:L590}, we have
$$
\lt|{\pfm_\om^{n+1}\1(y)\over \pfm_{\sg(\om)}^n\1(y)}-\lam_\om\rt|\le A\ka^n \
\text{{\rm and }} \
\lt|{\pfm_\tau^{n+1}\1(y)\over \pfm_{\sg(\tau)}^n\1(y)}-\lam_\tau\rt|\le A\ka^n
$$
with some constants $A>0$ and $\ka\in (0,1)$. Since, by our
assumptions, $\pfm_\om^{n+1}\1(y)=\pfm_\tau^{n+1}\1(y)$ and
$\pfm_{\sg(\om)}^n\1(y)=\pfm_{\sg(\tau)}^n\1(y)$, we conclude that
$|\lam_\om-\lam_\tau|\le 2A\ka^n$. So,
$$
v_\ka(\lam)\le 2A.
$$
Since, by Proposition~\ref{press cont}, the function $\lam:I^\Z\to(0,+\infty)$
is continuous, it is therefore bounded above and separated from
zero. In conclusion, both functions $\lam$ and $P(\phi)$ belong to
$\H_\ka$ with some $\ka\in (0,1)$, and both are continuous with
respect to the pseudo-metric $d_+$.
\end{proof}

\bcor\label{c1rds199} Suppose that $T:\J\to \J$ and $\phi:I^\Z\to\R$
form a classical expanding random system.  Then the number (asymptotic
variance of $P(\phi)$)
$$
\sg^2(P(\phi))=\lim_{n\to\infty}{1\over n}\int\Big(S_n(P(\phi))-n\cE
P(\phi)\Big)^2dm\geq 0
$$
exists, and the Law of Iterated Logarithm holds, i.e. $m$-a.e we have
$$
    -\sqrt{2 \sg^2(P(\phi))}=
   \liminf_{n\to\infty}{P_x^n-n\cE P(\phi)\over \sqrt{n\log\log n}}
\le  \limsup_{n\to\infty}{P_x^n(\phi)-n\cE P(\phi)\over \sqrt{n\log\log n}}
=    \sqrt{2\sg^2(P(\phi))}.
$$
\ecor

\begin{proof}
  Let $\pi :I^\Z\to I$ be the canonical projection onto the $0$th
  coordinate and let $\cG=\pi ^{-1}(\Bc)$, where $\Bc$ is the
  $\sg$-algebra of Borel sets of $I$.  We want to apply Theorem~1.11.1
  from \cite{PrzUrbXX}. Condition (1.11.6) is satisfied with the
  function $\phi$ (object being here as in Theorem~1.11.1 and by no
  means our potential!) identically equal to zero since $|m(A\cap B)
  -m(A)m(B)|=0$ for every $A\in
  \cG_0^m:=\cG\cap\sg^{-1}(\cG)\cap\ld\sg^{-m}(\cG)$ and $B\in
  \cG_n^\infty=\bi_{j=n}^{+\infty}\sg^{-j}(\cG)$, whenever $n>m$. The
  integral $\int|P(\phi)|^{2+\d}dm$ is finite (for every $\d>0$)
  since, by Theorem~\ref{t1rds197}, the pressure function $P(\phi)$ is
  bounded. This then implies that for all $n\ge 1$,
  $|P(\phi)(\om)-\cE\, (P(\phi)|\cG_0^n)(\om)|\le
  v_\ka(P(\phi))\ka^n$, where $v_\ka(P(\phi))<+\infty$. Therefore,
  \begin{displaymath}
    \int|P(\phi)-\cE\,(P(\phi)|\cG_0^n)|dm\le v_\ka(P(\phi))\ka^n,
  \end{displaymath}
  whence condition (1.11.7) from \cite{PrzUrbXX} holds. Finally,
  $P(\phi)$ is $\cG_0^\infty$-measurable, since $P(\phi)$ belonging to
  $\H_\ka$ is past independent. We have thus checked all the
  assumptions of Theorem~1.11.1 from \cite{PrzUrbXX} and, its
  application yields the existence of the asymptotic variance of
  $P(\phi)$ and the required Law of Iterated Logarithm to hold.
\end{proof}

\bprop\label{p1rds183}
Let $g\in\Ha_\ka$. Then $\sg^2(g)=0$ if and only if there exists $u\in
C((\supp(m_0))^\Z)$ such that $g-m(g)=u-u\circ\sg$ holds throughout
$(\supp(m_0))^\Z$.
\eprop

\begin{proof} Denote the topological support of $m_0$ by $S$. The
  implication that the cohomology equation implies vanishing of
  $\sg^2$ is obvious. In order to prove the other implication, assume
  without loss of generality that $m(g)=0$. Because of Theorem 2.51
  from \cite{IbrLin71}) there exists $u\in L_2(m)$ independent of the
  past (as so is $g$) such that \beq\label{1rds183} g=u-u\circ\sg \eeq
  in the space $L_2(m)$. Our goal now is to show that $u$ has a
  continuous version and (\ref{1rds183}) holds at all points of
  $S^\Z$. In view of Lusin's Theorem there exists a compact set $K\sbt
  S^Z$ such that $m(K)>1/2$ and the function $u|_K$ is continuous. So,
  in view of Birkhoff's Ergodic Theorem there exists a Borel set
  $B\sbt S^\Z$ such that $m(B)=1$, for every $\om\in B$,
  $\sg^{-n}(\om)\in K$ with asymptotic frequency $>1/2$, $u$ is
  well-defined on $\bu_{n=-\infty}^{+\infty}\sg^{-n}(B)$, and
  (\ref{1rds183}) holds on $\bu_{n=-\infty}^{+\infty}\sg^{-n}(B)$. Let
  $\Z_-=\{-1,-2,\ld\}$ and let $\{m_\tau\}_{\tau\in I^{\Z_-}}$ be the
  canonical system of conditional measures for the partition
  $\{\{\tau\}\times I^\N\}_{\tau\in I^{\Z_-}}$ with respect to the
  measure $m$. Clearly, each measure $m_\tau$, projected to $I^\N$,
  coincides with $m_+$. Since $m(B)=1$, there exists a Borel set
  $F\sbt S^{\Z_-}$ such that $m_-(F)=1$ and
  $m_\tau(B\cap(\{\tau\}\times I^\N))=1$ for all $\tau\in F$, where
  $m_-$ is the infinite product measure on $S^{\Z_-}$. Fix $\tau\in F$
  and set $Z=p_\N(B\cap(\{\tau\}\times I^\N))$, where $p_\N:I^\Z\to
  I^\N$ is the natural projection from $I^\Z$ to $I^\N$. The property
  that $m_\tau(B\cap(\{\tau\}\times I^\N))=1$ implies that $\ov
  Z=S^\N$. Now, it immediately follows from the definitions of $Z$ and
  $B$ that for all $x,y\in Z$ there exists an increasing sequence
  $(n_k)_{k=1}^\infty$ of positive integers such that $\sg^{-n_k}(\tau
  x),\sg^{-n_k}(\tau y) \in K$ for all $k\ge 1$. For every $0<q\le
  n_k$ we have from (\ref{1rds183}) that
  \begin{equation*}\begin{aligned}
    \sum_{j=0}^{n_k-q}
    \big(g(\sg^j(\sg^{-n_k}(\tau y))) - g(\sg^j(\sg^{-n_k}(\tau x)))\)
    +\sum_{j=n_k-q+1}^{n_k}
    \(g(\sg^j(\sg^{-n_k}(\tau y))) - g(\sg^j(\sg^{-n_k}(\tau x)))\) \\
    =(u(\sg^{-n_k}(\tau y))-u(\sg^{-n_k}(\tau x))+(u(\tau x)-u(\tau y)).
  \end{aligned}\end{equation*}
  Since $g\in H_\ka$, we have
\begin{equation*}
\begin{aligned}
    \sum_{j=0}^{n_k-q}
        \(g(\sg^j(\sg^{-n_k}(\tau y)))
- g(\sg^j(\sg^{-n_k}(\tau y)))\)
&\le \sum_{j=0}^{n_k-q}
         |g(\sg^j(\sg^{-n_k}(\tau y))) - g(\sg^j(\sg^{-n_k}(\tau y)))|\\
&\le \sum_{j=0}^{n_k-q}v_\ka(g)\ka^{n_k-j} \le v_\ka(g)(1-\ka)^{-1}\ka^q.
\end{aligned}
\end{equation*}
Now, fix $\e>0$. Take $q\ge 1$ so large that
$
v_\ka(g)(1-\ka)^{-1}\ka^q<\e/2.
$
Since the function $g:I^\Z\to\R$ is uniformly continuous with respect
to the pseudometric $d$, there exists $\d>0$ such that $|g(b)-g(a)|<{\e\over
  2q}$ whenever $d(a,b)<\d$. Assume that $d(x,y)<\d$ (so
$d(\sg^{-i}(\tau x),\sg^{-i}(\tau y))<\d$ for all $i\ge 0$) It
follows now that for every $k\ge 1$
we have
$$
\begin{aligned}
|u(\tau x)-u(\tau y)|
&\le v_\ka(g)(1-\ka)^{-1}\ka^q+q{\e\over 2q}+
     |u(\sg^{-n_k}(\tau y))-u(\sg^{-n_k}(\tau x))| \\
&\le {\e\over 2}+{\e\over 2}+|u(\sg^{-n_k}(\tau y))-u(\sg^{-n_k}(\tau
        x))| \\
&=\e+{\e\over 2}+|u(\sg^{-n_k}(\tau y))-u(\sg^{-n_k}(\tau x))|.
\end{aligned}
$$
Since $\sg^{-n_k}(\tau x),\sg^{-n_k}(\tau y)\in K$ for all $k\ge 1$,
since $\lim_{k\to\infty}d(\sg^{-n_k}(\tau x),\sg^{-n_k}(\tau y))=0$,
and since the function $u$, restricted to $K$, is uniformly
continuous, we conclude that
$$\lim_{k\to\infty}|u(\sg^{-n_k}(\tau
y))-u(\sg^{-n_k}(\tau x))|=0\, .$$
 We therefore get that
$|u(\tau x)-u(\tau y)|<\e$ and this shows that the function $u$ is uniformly
continuous (with respect to the metric $d$) on the set
$$
W=\bu_{\tau\in F}B\cap(\{\tau\}\times I^\N)
$$
Since $\ov W=S^\Z$ (as $m(W)=1$) and since $u$ is
independent of the past, we conclude that $u$ extends continuously to
$S^\Z$. Since both sides of (\ref{1rds183}) are continuous functions,
and the equality in (\ref{1rds183}) holds on the dense set $W\cap
\sg^{-1}(W)$, we are done.
\end{proof}

\section{Classical Conformal Expanding Random Systems}
If a classical system is conformal in the sense of Definition~\ref{conformal RDS}
and if the potential is of the form $\ph = -t \log |f'|$ for some $t\in \R$
then we will call it \emph{classical conformal expanding random system}\index{classical conformal expanding random systems}

\bthm\label{t1crcs} Suppose $f:\J\to \J$ is a classical conformal
expanding random system. Then the following hold.
\begin{itemize}
\item[(a)]  The asymptotic variance $\sg^2(P(h))$ exists.
\item[(b)]  If $\sg^2(P(h))>0$, then the system $f:\J\to \J$ is
  essential, $\Ha^h(\J_x)=0$ and $\Pa^h(\J_x)=+\infty$
for $m$-a.e. $x\in I^\Z$.
\item[(c)]  If, on the other hand, $\sg^2(P(h))=0$, then the system
  $f:\J\to \J$, reduced in the base to the topological support of $m$ (equal
  to $\supp(m_0)^\Z$), is quasi-deterministic, and then for every
  $x\in\supp(m )$, we have:
\begin{itemize}
\item[(c1)] $\nu_x^h$ is a geometric measure with exponent $h$.
\item[(c2)] The measures $\nu_x^h$, $\Ha^h|_{\J_x}$, and $\Pa^h|_{\J_x}$
  are all mutually
equivalent with Radon-Nikodym derivatives separated away from zero and infinity
independently of $x\in I^\Z$ and $y\in \J_x$.
\item[(c3)] $0<\Ha^h(\J_x),\Pa^h(\J_x)<+\infty$ and $\HD(\J_x)=h$.
\end{itemize}
\end{itemize}
\end{thm}
\begin{proof} It follows from Corollary~\ref{c1rds199} that the asymptotic variance
  $\sg^2(P(h))$ exists. Combining this corollary (the Law of Iterated Logarithm)
with Remark~\ref{r3rds193}, we conclude that the system $f:\J\to \J$ is
  essential. Hence, item (b) follows from
  Theorem~\ref{t2rds109}(a). If, on the
  other hand, $\sg^2(P(h))=0$, then the system
  $f:\J\to \J$, reduced in the base to the topological support of $m$ (equal
  to $\supp(m_0)^\Z$), is quasi-deterministic because of
  Proposition~\ref{p1rds183}, Theorem~\ref{t1rds197} ($P(h)\in\Ha_\ka$), and
  Remark~\ref{r4rds193}. Items (c1)-(c4) follow now from
  Theorem~\ref{t2rds109}(b1)-(b4). We are done.
\end{proof}

As a consequence of this theorem we get the following.

\bthm\label{t1dtr} Suppose $f:\J\to \J$ is a classical conformal
expanding random system. Then the following hold.
\begin{itemize}
\item[(a)] Suppose that for every $x\in I^\Z$, the fiber $\J_x$ is
  connected. If there exists at least one $w\in \supp(m)$ such that
  $\HD(\J_w)>1$, then $\HD(\J_x)>1$ for $m$-a.e. $x\in I^\Z$.
\item[(b)] Let $d$ be the dimension of the ambient Riemannian space
  $Y$. If there exists at least one $w\in \supp(m)$ such that
  $\HD(\J_w)<d$, then $\HD(\J_x)<d$ for $m$-a.e. $x\in I^\Z$.
\end{itemize}
\ethm

\begin{proof} Let us proof first item (a). By
  Theorem~\ref{t1crcs}(a) the asymptotic variance
  $\sg^2(P(h))$ exists. If $\sg^2(P(h))>0$, then by
  Theorem~\ref{t1crcs}(a) the system $f:\J\to \J$ is essential. Thus
  the proof is concluded in exactly the same way as the proof of
  Theorem~\ref{cbilipscitz}(3). If, on the other hand,
  $\sg^2(P(h))=0$, then the assertion of (a) follows from
  Theorem~\ref{t1crcs}(c4) and the fact that $\HD(\J_w)>1$ and
  $w\in \supp(m)$.

Let us now prove item (b). If $\sg^2(P(h))>0$, then, as in the proof
of item (a), the claim is proved in exactly the same way as the proof of
  Theorem~\ref{cbilipscitz}(4). If, on the other hand,
  $\sg^2(P(h))=0$, then the assertion of (b) follows from
  Theorem~\ref{t1crcs}(c4) and the fact that $\HD(\J_w)<d$ and
  $w\in\supp(m)$. We are done.
\end{proof}

\section{Complex Dynamics and Br\"uck and B\"urger Polynomial
  Systems}
We now want to describe some classes of examples coming from complex
dynamics. They will be classical conformal expanding random systems as
well as $G$-systems defined later in this section. Indeed, having a
sequence of rational functions $F=\{f_n\}_{n=0}^\infty$ on the Riemann
sphere $\oc$ we say that a point $z\in\oc$ is a member of the Fatou
set of this sequence if and only if there exists an open set $U_z$
containing $z$ such that the family of maps
$\{f_n|_{U_z}\}_{n=0}^\infty$ is normal in the sense of Montel. The
Julia set $\J(F)$ is defined to be the complement (in $\oc$) of the
Fatou set of $F$. For every $k\ge 0$ put
$F_k=\{f_{k+n}\}_{n=0}^\infty$ and observe that
\begin{equation}\label{1rds141}
\J(F_{k+1})=f_k(\J(F_k)).
\end{equation}
Now, consider the maps
\begin{displaymath}
  f_c(z)=f_{d,c}(z)=z^d+c\textrm{, }\quad d\ge 2.
\end{displaymath}
Notice that for every $\e>0$ there exists $\d_\e>0$ such that if
$|c|\le\d_\e$, then
\begin{displaymath}
  f_c(\ov B(0,\e))\sbt\ov B(0,\e).
\end{displaymath}
Consequently, if $\om\in \ov B(0,\e)^{\Z}$, then
$
\J(\{f_{\om_n}\}_{n=0}^\infty)\sbt\{z\in\C:|z|\ge\e\}
$
and
\beq\label{1rds189}
|f_{\om_k}'(z)|\ge d\e^{d-1}
\eeq
for all $z\in \J(\{f_{\om_{k+n}}\}_{n=0}^\infty)$. Let
$
\d(d)=\sup\Big\{\d_\e:\e>\sqrt[d-1]{1/d}\Big\} .
$
Fix $0<\d<\d(d)$. Then there exists $\e>\sqrt[d-1]{1/d}$ such that
$\d<\d_\e$. Therefore, by (\ref{1rds189}), \beq\label{1ards189}
|f_{\om_k}'(z)|\ge d\e^{d-1} \eeq for all $\om\in \ov B(0,\d)^\Z$, all
$k\ge 0$ and all $z\in \J(\{f_{\om_{k+n}}\}_{n=0}^\infty)$. A straight
calculation (\cite{Bru01}, p. 349) shows that $\d(2)=1/4$. Keep
$0<\d<\d(d)$ fixed. Let
$$
\Fa_{d,\d}=\{f_{d,c}:c\in \ov B(0,\d)\}.
$$
Consider an arbitrary ergodic measure-preserving transformation
$\th:X\to X$. Let $m$ be the corresponding invariant probability
measure. Let also $H:X\to\Fa_{d,\d}$ be an arbitrary measurable
function. Set $f_{d,x}=H(x)$ for all $x\in X$. For every $x\in X$ let
$\J_x$ be the Julia set of the sequence
$\{f_{\th^n(x)}\}_{n=0}^\infty$, and then $\J=\bu_{x\in X}\J_x$. Note
that, because of (\ref{1rds141}),
$
  f_{d,x}(\J_x)=\J_{\th(x)}.
$
Thus, the map
\begin{equation}
  \label{3rds141} f_{d,\d,\th,H}(x,y)=(\th(x),f_{d,x}(y)) \
  x\in X,\, y\in \J_x,
\end{equation}
defines a skew product map in the sense of
Chapter~\ref{sec:preliminaries} of our paper. In view of
(\ref{3rds141}), when $\th:X\to X$ is invertible, $f_{d,\d,\th,H}$ is
a distance expanding random system, and, since all the maps $f_x$ are
conformal, $f_{d,\d,\th,H}$ is a conformal measurably expanding system
in the sense of Definition~\ref{conformal RDS}. As an immediate
consequence of Theorem~\ref{thm:Hau100} we get the following.

\bthm\label{t1rds141} Let $\th:X\to X$ be an invertible measurable map
preserving a probability measure $m$. Fix an integer $d\ge 1$ and
$0<\d<\d(d)$.  Let $H:X\to\Fa_{d,\d}$ be an arbitrary measurable
function. Finally, let $f_{d,\d,\th,H}$ be the distance expanding
random system defined by formula (\ref{3rds141}). Then for
almost all $x\in X$ the Hausdorff dimension of the Julia set $\J_x$ is
equal to the unique zero of the expected value of the pressure
function.  \ethm

\bthm\label{t1rds143}
For the conformal measurably expanding systems
$f_{d,\d,\th,H}$ defined in Theorem~\ref{t1rds141} the multifractal theorem,
Theorem~\ref{thm:MA99} holds.
\ethm

We now define and deal with Br\"uck and B\"urger polynomial
systems. We still keep $d\ge 2$ and $0<\d<\d(d)$ fixed. Let
$X=B(0,\d)^\Z$ and let
\begin{displaymath}
  \th:B(0,\d)^\Z\to B(0,\d)^\Z
\end{displaymath}
to be the shift map denoted in the sequel by $\sigma$. Consider any
Borel probability measure $m_0$ on $\ov B(0,\d)$ which is different
from $\d_0$, the Dirac $\d$ measure supported at $0$. Define
$H:X\to\Fa_{d,\d}$ by the formula $H(\om)=f_{d,\om_0}$. The
corresponding skew-product map $f_{d,\d}:\J\to\J$ is then given by the
formula
$$
f_{d,\d}(\om,z)=(\sg(\om),f_{d,\om_0}(z))=(\sg(\om),z^d+\om_0),
$$
and
$
  f_{d,\d,\om}(z)=z^d+\om_0
$
acts from $\J_\om$ to $\J_{\sg(\om)}$, where
$
  \J_\om=\J((f_{d,\om_n})_{n=0}^\infty).
$
Then $f:\J\to \J$ is called\index{Br\"uck and B\"urger polynomial
  systems} \emph{Br\"uck and B\"urger polynomial systems}. Clearly,
$f:\J\to \J$ is a classical conformal expanding random system.

In \cite{Bru01} Br\"uck speculated on page 365 that if $\delta<1/4$
and $m_0$ is the normalized Lebesgue measure on $\ov B(0,\d)$, then
$\HD(\J_\om)>1$ for $m_+$-a.e. $\om\in\ov B(0,\d)^\N$ with respect to
the skew-product map
\begin{displaymath}
  (\om,z)\mapsto (\sg(\om),z^2+\om_0).
\end{displaymath}
In \cite{BruBug03} this problem was explicitly formulated by Br\"uck
and B\"urger as Question~5.4.  Below (Theorem~\ref{t1rds201}) we prove
a more general result (with regard the measure on $\ov B(0,\d)$ and
the integer $d\ge 2$ being arbitrary), which contains the positive
answer to the Br\"uck and B\"urger question as a special case. In
\cite{Bru01} Br\"uck also proved that if $\d<1/4$ and the above skew
product is considered then $\lam_2(\J_\om)=0$ for all $\om\in\ov
B(0,\d)^\N$, where $\lam_2$ denotes the planar Lebesgue measure on
$\C$. As a special case of Theorem~\ref{t1rds201} below we get a
partial strengthening of Br\"uck's result saying that $\HD(\J_\om)<2$
for $m_+$-a.e.  $\om\in\ov B(0,\d)^\N$. Our results are formulated for
the product measure $m$ on $\ov B(0,\d)^\Z$, but as $m_+$ is the
projection from $\ov B(0,\d)^\Z$ to $\ov B(0,\d)^\N$ and as the Julia
sets $\J_\om$, $\om\in\ov B(0,\d)^\Z$ depend only on
$\om|_0^{+\infty}$, i.e. on the future of $\om$, the analogous results
for $m_+$ and $\ov B(0,\d)^\N$ follow immediately.  Proving what we
have just announced, note that if $\om_0\in\supp(m_0)\sms \{0\}$, then
\begin{displaymath}
  \HD(\J_{\om_0^\infty})) =\HD(\J(f_{\om_0}))\in (1,2)
\end{displaymath}
(the equality holds already on the level of sets:
$\J_{\om_0^\infty}=\J(f_{\om_0})$), and by \cite{BruBug03}, all the
sets $\J_\om$, $\om\in \ov B(0,\d)^\Z$, are Jordan curves. Hence,
since $f:\J\to \J$ is a classical conformal expanding random system,
as an immediate application of Theorem~\ref{t1dtr} we get the
following.

\bthm\label{t1rds201} If $d\ge 2$ is an integer, $0<\d<\d(d)$, the
skew-product map $f_{d,\d}:\J\to\J$ is given by the formula
$$
f_{d,\d}(\om,z)=(\sg(\om),f_{d,\om_0}(z))=(\sg(\om),z^d+\om_0),
$$
and $m_0$ is an arbitrary Borel probability measure on $\ov B(0,\d)$, different from $\d_0$,
the Dirac $\d$ measure supported at $0$, then for $m$-almost every
$\om\in\ov B(0,\d)^\Z$ we have $1<\HD(\J_\om)<2$.
\ethm

\section{\index{DG-system}Denker-Gordin Systems} We now want to discuss
another class of expanding random maps. This is the
setting from \cite{DenGor99}. In order to describe this setting
 suppose that $X_0$ and $Z_0$ are compact metric spaces and that
$\th_0:X_0\to X_0$ and $T_0:Z_0\to Z_0$ are open topologically exact
distance expanding maps in the sense as in \cite{PrzUrbXX}.  We assume
that $T_0$ is a skew-product over $Z_0$, i.e. for every $x\in X_0$ there
exists a compact metric space $\J_x$ such that $Z_0=\bu_{X\in
  X_0}\{x\}\times \J_x$ and the following diagram commutes
\begin{diagram}
  Z_0 & \rTo^{T_0} & Z_0\\
  \dTo_{\pi} & & \dTo_{\pi} \\
  X_0 & \rTo^{\th_0} & X_0 \\
\end{diagram}
where $\pi(x,y)=x$ and the projection $\pi:Z_0\to X_0$ is an open
map.  Additionally, we assume that there exists $L$ such
that
\begin{equation}
  \label{eq:13}
  d_{X_0}(\shift_0(x),\shift_0(x'))\leq Ld_X(x,x')
\end{equation}
for all $x\in X$ and that there exists $\xi_1>0$ such that, for all
$x,x'$ satisfying $d_{X_0}(x,x')<\xi_1$ there exist $y,y'$ such that
\begin{equation}
  \label{eq:14}
  d\big((x,y),(x',y')\big)<\xi.
\end{equation}
We then refer to $T_0:Z_0\to Z_0$ and $\th_0:X_0\to X_0$ as a DG-system. Note that
\begin{displaymath}
  T_0(\{x\}\times \J_x) \sbt \{\th_0(x)\}\times
  \J_{\th_0(x)}
\end{displaymath}
and this gives rise to the map $T_x:\J_x\to \J_{\th_0(x)}$.

Since $T_0$ is distance expanding, conditions uniform openness,
measurably expanding measurability of the degree, topological
exactness (see Chapter~\ref{ch:GREM}) hold with some constants
$\g_x\ge\g>1$, $\deg (T_x)\le N_1<+\infty$ and the number $n_r=n_r(x)$
in fact independent of $x$. Scrutinizing the proof of Remark~2.9 in
\cite{DenGor99} one sees that Lipschitz continuity (Denker and Gordin
assume differentiability) suffices for it to go through and Lipschitz
continuity is incorporated in the definition of expanding maps in
\cite{PrzUrbXX}. Now assume that $\phi:Z\to\R$ is a H\"older
continuous map. Then the hypothesis of Theorems~2.10, 3.1, and 3.2
from \cite{DenGor99} are satisfied. Their claims are summarized in the
following.

\begin{thm}\label{t1rds103}
Suppose that $T_0:Z_0\to Z_0$ and $\th_0:X_0\to X_0$ form a DG system and that
$\phi:Z\to\R$ is a H\"older continuous potential. Then there exists a H\"older continuous
function $P(\phi):X_0\to\R$, a measurable collection $\{\nu_x\}_{x\in X_0}$ and a
continuous function $q:Z_0\to [0,+\infty)$ such that
\begin{itemize}
\item[(a)] $\nu_{\th_0(x)}(A)=\exp\(P_x(\phi)\)\int_Ae^{-\phi_x}d\nu_x$ for all $x\in X_0$ and all
Borel sets $A\sbt \J_x$ such that $T_x|_A$ is one-to-one.
\item[(b)] $\int_{\J_x}q_xd\nu_x=1$ for all $x\in X_0$.
\item[(c)] Denoting for every $x\in X_0$ by $\mu_x$ the measure
  $q_x\nu_x$ we have
$$\sum_{w\in\th_0^{-1}(x)}\mu_w(T_w^{-1}(A))=\mu_x(A) \qquad \text{for every Borel set} \quad A\sbt \J_x\; .$$
\end{itemize}
\end{thm}

This would mean that we got all the objects produced in Chapter~\ref{ch:RPFmain} of
our paper. However, the map $\th_0:X_0\to X_0$ need not be, and apart
from the case when $X_0$ is finite, is not invertible. But to remedy
this situation is easy. We consider the projective limit (Rokhlin's
natural extension) $\th:X\to X$ of $\th_0:X_0\to X_0$. Precisely,
$$
X=\{(x_n)_{n\le 0}:\th_0(x_n)=x_{n+1}\, \forall n\le -1\}
$$
and
$$
\th\((x_n)_{n\le 0}\)=(\th_0(x_n))_{n\le 0}.
$$
Then $\th:X\to X$ becomes invertible and the diagram
\begin{equation}
  \label{diag:1rds105}
  \begin{diagram}
    X & \rTo^{\th} & X\\
    \dTo_{p} & & \dTo_{p} \\
    X_0 & \rTo^{\th_0} & X_0 \\
  \end{diagram}
\end{equation}
commutes, where $p\((x_n)_n\le 0\)=x_0$. If in addition, as we assume
from now on, the space $X$ is endowed with a Borel probability
$\th_0$-invariant ergodic measure $m_0$, then there exists a unique
$\th$-invariant probability measure measure $m$ such that
$m\circ\pi^{-1}=m_0$. Let
$$
Z:=\bu_{x\in X}\{x\}\times \J_{x_0}.
$$
We define the map $T:Z\to Z$ by the formula
$
T(x,y)=(\th(x),T_{x_0}(y))
$
and the potential $X\ni x\mapsto \phi(x_0)$ from $X$ to $\R$. We keep
for it the same symbol $\phi$. Clearly the quadruple $(T,\th,m,\phi)$
is a H\"older fiber system as defined in Chapter~\ref{ch:GREM} of our paper. It
follows from Theorem~\ref{t1rds103} along with the definition of $\th$
a commutativity of the diagram (\ref{diag:1rds105}) for $x\in X$ all
the objects $P_x(\phi)=P_{x_0}(\phi)$, $\lambda_x=\exp(P_x(\phi))$,
$q_x=q_{x_0}$, $\nu_x= \nu_{x_0}$, and $\mu_x=\mu_{x_0}$ enjoy all the
properties required in Theorem~\ref{thm:Gib50A} and
Theorem~~\ref{thm:Gib50B}; in particular they are unique. From now on
we assume that the measure $m$ is a Gibbs state of a H\"older
continuous potential on $X$ (having nothing to do with $\phi$ or
$P(\phi)$; it is only needed for the Law of Iterated Logarithm to
hold). We call the quadruple $(T,\th,m,\phi)$\index{DG*-system} DG*-system.

The following H\"older continuity theorem appeared in the paper
\cite{DenGor99}. We provide here an alternative proof under weaker
assumptions.
\begin{thm}
  \label{thm:16}
  If $d_X(x,x')<\xi$, then $|\lambda_x-\lambda_{x'}|\leq H
  d_X^{\alpha}(x,x')$.
\end{thm}
\begin{proof}
  Let $n$ be such that
  \begin{equation}
    \label{eq:16}
    d_X(\shift^{2n-1}(x),\shift^{2n-1}(x'))<\xi_1\textrm{ and }
    d_X(\shift^{2n}(x),\shift^{2n}(x'))\geq\xi_1.
  \end{equation}
  Let $z\in T^{-2n+1}(y)$ and $z'\in T^{-2n+1}(y')$.
  Then for all $k=0,\ldots,n-1$
  \begin{displaymath}
    |\varphi(T^k(z))-\varphi(T^k(z'))|\leq Cd^{\alpha}(T^k(z), T^k(z'))\leq
    C\gamma^{-\alpha n}\gamma^{-\alpha(n-k-1)}\xi.
  \end{displaymath}
  Then
  \begin{displaymath}
    |S_n\varphi(z)-S_n\varphi(z')|\leq \frac{C\xi\gamma^{-\alpha n}}
    {1-\gamma^{-\alpha}}.
  \end{displaymath}
  Put $C':=C\xi/({1-\gamma^{-\alpha}})$.  Then
  \begin{displaymath}
    \Big|\log\frac{\cL_x^n\1(w)}{\cL_{x'}^n\1(w')}\Big|
    \leq C'\gamma^{-\alpha n}
 \qquad
  \text{and}
 \qquad
    \Big|\log\frac{\cL_{\shift(x)}^{n-1}\1(w)}{\cL_{\shift(x')}^{n-1}\1(w')}\Big|
    \leq C'\gamma^{-\alpha n}.
  \end{displaymath}
  Then
  \begin{equation}
    \label{eq:15}
    \Big|\log\frac{\cL_x^n\1(w)}{\cL_{\shift(x)}^{n-1}\1(w)}
    -\log\frac{\cL_{x'}^n\1(w')}{\cL_{\shift(x')}^{n-1}\1(w')}\Big|\leq
    2C'\gamma^{-\alpha n}.
  \end{equation}
  Let $\alpha':=(\alpha\log\gamma)/(2\log L)$. Then by \eqref{eq:16}
  \begin{displaymath}
    \gamma^{-n\alpha}=L^{-2n\alpha'}\leq
    \frac{(d(\shift^{2n}(x),\shift^{2n}(x')))^{\alpha'}}
    {\xi_1^{\alpha'}L^{-2n\alpha'}}\leq
    \frac{(d(x,x'))^{\alpha'}}{\xi_1^{\alpha'}}.
  \end{displaymath}
  Then \eqref{eq:15} finishes the proof.
\end{proof}

Since the map $\th_0:X_0\to
X_0$ is expanding, since $m$ is a Gibbs state, and since
$P(\phi):X_0\to\R$ is H\"older continuous, it is well-known (see
\cite{PrzUrbXX} for example) that the following asymptotic variance
exists
$$
\sg^2(P(\phi))=\lim_{n\to\infty}\frac{1}{n}
\int\Big(S_n(P(\phi))-n\cE P(\phi)\Big)^2dm.
$$
The following theorem of Livsic flavor is (by now) well-known (see \cite{PrzUrbXX}).

\begin{thm}\label{t1rds107}
Suppose $(T,\th,m,\phi)$ is a DG*-system. Then the following are
equivalent.
\begin{itemize}
\item[(a)] $\sg^2(P(\phi))=0$.
\item[(b)] The function $P(\phi)$ is cohomologous to a constant in the
  class of real-valued continuous functions on $X$ (resp. $X_0$),
  meaning that there exists a continuous function $u:X \to\R$
  (resp. $u:X_0\to\R$) such that
  \begin{displaymath}
    P(\phi)-(u-u\circ \th)
    \quad\textrm{(resp. }P(\phi)-(u-u\circ \th_0)\textrm{)}
  \end{displaymath}
   is a constant.
\item[(c)] The function $P(\phi)$ is cohomologous to a constant in the
  class of real-valued H\"older continuous functions on $X$
  (resp. $X_0$), meaning that there exists a H\"older continuous
  function $u:X\to\R$ (resp. $u:X\to\R$) such that
  \begin{displaymath}
    P(\phi)-(u-u\circ
    \th)  \quad\textrm{(resp. }P(\phi)-(u-u\circ \th_0)\textrm{)}
  \end{displaymath}
   is a constant.
\item[(d)] There exists $R\in\R$ such that $P_x^n(\phi)=nR$ for all
  $n\ge 1$ and all periodic points $x\in X$ (resp. $X_0$).
\end{itemize}
\end{thm}

\fr As a matter of fact such theorem is formulated in \cite{PrzUrbXX} for
non-invertible ($\th_0$) maps only but it also holds for the Rokhlin's
natural extension $\th$. The following theorem follows directly from
\cite{PrzUrbXX} and Theorem~\ref{t1rds103} (H\"older continuity of
$P(\phi)$).

\begin{thm}\label{t2grds107}
  (the Law of Iterated Logarithm) If $(T,\th,m,\phi)$ is a DG*-system
  and if $\sg^2 (P(\phi))>0$, then $m$-a.e. we have
  $$
  -\sqrt{2 \sg^2(P(\phi))}=
  \liminf_{n\to\infty}\frac{P_x^n(\phi)-n\cE P(\phi)}{\sqrt{n\log\log n}}
  \le  \limsup_{n\to\infty}\frac{P_x^n(\phi)-n\cE P(\phi)}{\sqrt{n\log\log n}}
  =    \sqrt{2\sg^2(P(\phi))}.
  $$
\end{thm}

\section{\index{conformal!DG*-systems}Conformal DG*-Systems} Now we
turn to geometry. This section dealing with, below defined, conformal
DG*-systems is a continuation of the previous one in the setting of
conformal systems. We shall show that these systems naturally split
into essential and quasi-deterministic, and will establish their
fractal and geometric properties. Suppose that $(f_0,\th_0)$ is a
DG-system endowed with a Gibbs measure $m_0$ at the base. Suppose also
that this system is a random conformal expanding repeller in the sense
of Chapter~\ref{ch:section 5} and that the function $\phi:Z\to\R$ given by the
formula
$$
\phi(x,y)=-\log|f_x'(y)|,
$$
is H\"older continuous.

\bdfn\label{Gibbs system} The corresponding system $(f,\th,m)=(f, \th,
m,\phi)$ (with $\th$ the Rokhlin natural extension of $\th_0$ as
described above) is called conformal DG*-system.  \edfn

For every $t\in\R$ the potential $\phi_t=t\phi$, considered in
Chapter~\ref{ch:section 5}, is also H\"older continuous. As in
Chapter~\ref{ch:section 5} denote its topological pressure by $P(t)$.
Recall that $h$ is a unique solution to the equation $\cE P(t)=0$. By
Theorem~\ref{thm:Hau100} (Bowen's Formula) $\HD(\J_x)=h$ for
$m$-a.e. $x\in X$.  As an immediate consequence of
Theorem~\ref{t2rds109}, Theorem~\ref{t2grds107}, and
Remark~\ref{r4rds193}, we get the following.

\begin{thm}\label{t2grds109}
  Suppose $(f,\th,m)=(f,\th,m,\phi)$ is a random conformal DG*-system.
\begin{itemize}
\item[(a)]\ If $\sg^2(P(h))>0$, then the system $(f,\th,m)$ is
  essential, and then
  \begin{displaymath}
    \Ha^h(\J_x)=0\quad\textrm{ and }\quad\Pa^h(\J_x)=+\infty.
  \end{displaymath}
\item[(b)]\ If, on the other hand, $\sg^2(P(h))=0$, then
  $(f,\th,m)=(f,\th,m,\phi)$ is quasi-deterministic, and then for
  every $x\in X$, we have that $\nu_x^h$ is a geometric measure with exponent $h$
  and, consequently, the geometric properties (\textsc{gm}1)-(\textsc{gm}3) hold.
\end{itemize}
\end{thm}

Exactly as Corollary~\ref{cbilipscitz} is a consequence of Theorem~\ref{t2rds109},
the following corollary is a consequence of Theorem~\ref{t2grds109}.

\bcor\label{cgbilipscitz} Suppose $(f,\th,m)=(f,\th,m,\phi)$ is a
conformal DG*-system and $\sg^2(P(h))>0$. Then the system $(f,\th,m)$
is essential, and for $m$-a.e. $x\in X$ the following hold.
\begin{enumerate}
\item The fiber $\J_x$ is not bi-Lipschitz equivalent to any
deterministic nor quasi-deterministic self-conformal set.
\item $\J_x$ is not a geometric circle nor even a piecewise smooth curve.
\item If $\J_x$ has a non-degenerate connected component (for
  example if $\J_x$ is connected), then
  \begin{displaymath}
    h=\HD(\J_x)>1.
  \end{displaymath}
\item Let $d$ be the dimension of the ambient Riemannian space
  $Y$. Then
 $
    \HD(\J_x)<d.
$
\end{enumerate}
\ecor

Now, in the same way as Theorem~\ref{t1dtr} is a consequence of Theorem~\ref{t1crcs},
Corollary~\ref{cgbilipscitz} yields the following.

\bthm\label{t1gdtr}
Suppose $(f,\th,m)=(f,\th,m,\phi)$ is a conformal DG*-system. Then the
following hold.
\begin{itemize}
\item[(a)] Suppose that for every $x\in X$, the fiber $\J_x$ is
  connected. If there exists at least one $w\in \supp(m)$ such that
  $\HD(\J_w)>1$, then
  \begin{displaymath}
    \HD(\J_x)>1
 \qquad
  \text{for m-a.e.} \quad x\in I^\Z\; . \end{displaymath}
\item[(b)] Let $d$ be the dimension of the ambient Riemannian space
  $Y$. If there exists at least one $w\in X$ such that
  $\HD(\J_w)<d$, then $\HD(\J_x)<d$ for $m$-a.e. $x\in X$.
\end{itemize}
\ethm

We end this subsection and the entire section with a concrete
example of a conformal DG*-system. In particular, the three above
results apply to it.  Let
\begin{displaymath}
  X:=S_{\d_d}^1=\{z\in\C:|z|=\d\}.
\end{displaymath}
Fix an integer $k\ge 2$. Define the map $\th_0:X\to X$ by the formula
$
\th_0(x)=\d^{1-k}x^k.
$
Then
$
  \th_0'(x)=k\d^{1-k}x^{k-1}
$
and therefore
$
  |\th_0'(x)|=k\ge 2
$
for all $x\in X$. The normalized Lebesgue measure $\lambda_0$ on $X$
is invariant under $\th_0$. Define the map $H:X\to\Fa_d$ by setting
$H(x)=f_x$. Then
$$
f_{\th_0,H,0}(x,y)=(k\d^{1-k}x^{k-1},g^d+x).
$$
Note that $\(f_{\th_0,H,0},\th_0,\lambda_0)$ is a uniformly conformal
DG-system and let $(f_{\th,H},\th,\lambda)$ be the corresponding random
conformal $G$-system, both in the sense of Chapter~\ref{ch:section 5}.
Theorem~\ref{t2grds109}, Theorem~\ref{t1gdtr}, and Corollary~\ref{cgbilipscitz}
apply.

%\bthm\label{t1rds145} If $\(f_{\th_0,H,0},\th_0,\lambda_0)$ is the
%random conformal DG*-system described above, then
%\begin{itemize}
%\item[(f)] $\Ha^h(\J_x)=0$ and $\Pa^h(\J_x)=+\infty$ for
%  $\lambda$-a.e. $x\in X$ if only $\(f_{\th_0,H,0},\th_0,\lambda_0)$
%  is essential. In consequence, for $\lambda$-a.e. $x\in X$ the Julia
%  set $\J_x$ is not bi-Lipschitz equivalent to any deterministic
%  self-conformal set. Furthermore, $\J_x$ is not a geometric circle nor
%  even a piecewise smooth curve. In fact, if $\J_x$ is connected (it
%  suffices to have a non-degenerate connected component) then
%  \begin{displaymath}
%    h=\HD(\J_x)>1.
%  \end{displaymath}
%\item[(g)] If, on the other hand, $\(f_{\th_0,H,0},\th_0,\lambda_0)$
%  is quasi-deterministic, then for every $x\in X$,
%\begin{itemize}
%\item[(g1)]\quad $\nu_x^h$ is a geometric measure with exponent $h$.
%\item[(g2)]\quad The measures $\nu_x^h$, $\Ha^h|_{\J_x}$, and $\Pa^h|_{\J_x}$
%  are all mutually equivalent with Radon-Nikodym derivatives separated
%  away from zero and infinity independently of $x\in X$ and $y\in
%  \J_x$.
%\item[(g3)]\quad $0<\Ha^h(\J_x),\Pa^h(\J_x)<+\infty$.
%\item[(g4)]\quad $\HD(\J_x)=h$.
%\end{itemize}
%\end{itemize}
%\ethm

%**************************************************** Stationary measure *****************
\section{Random expanding maps on smooth manifold}
We now complete the previous examples with some remarks on random maps
on smooth manifolds.
Let $(M, \rho)$ be a smooth compact Riemannian manifold. We recall that a
differentiable endomorphism $f:M\to M$ is expanding if there
exists $\g>1$ such that
$$
||f'_x(v)||\ge \g||v||  \qquad \text{for all} \;\; x\in M \text{ and all } v\in T_xM \,   .
$$
The largest constant $\g>1$ enjoying this property is denoted by $\g(f)$.
If $\g>1$, we denote by $\cE_\g(M)$ the set of all expanding
endomorphisms of $M$ for which $\g(f)\ge \g$. We also set
$$
\cE(M)=\bu_{\g>1}\cE_\g(M),
$$
i.e. $\cE(M)$ is the set of all expanding endomorphisms of $M$.

\section{Topological exactness}
We
shall prove the following.

\bprop\label{p1rds237}
Suppose that for each $n\ge 1$, $f_n\in\cE(M)$ and
$lim_{n\to\infty}\Pi_{j=1}^n\g(f_j)=+\infty$. If $U$ is a non-empty
open subset of $M$, then there exists $k\ge 1$ such that $F_k(U)=M$,
where $F_k=f_k\circ f_{k-1}\circ\ld\circ f_1$. If there exists $\g>1$
such that $f_n\in\cE_\g(M)$ for all $n\ge 1$, then for every $r>0$
there exists $k_k\ge 1$ such that for every $x\in M$, we have
$$
F_{k_r}(B(x,r))=M.
$$
\eprop

{\bf Proof.} Let $\^M$ be the universal cover of $M$ and let
$\pi:\^M\to M$ be the corresponding projection. Let $q=\dim(M)$
be the dimension of $M$. It is well known that $\^M$ is diffeomorphic
to $\R^q$ and that $\^M$ can be canonically endowed with a Riemannian
metric $\^\rho$ such that for every $x\in\^M$, the derivative
$D_x\pi:T_X\^M\to T_{\pi(x)}M$ is an isometry and all covering maps of
$\^M$ are isometries too. Fix a point $w\in\^M$. Since
$$
\bu_{R>0}\pi(B(w,R))
=\pi\lt(\bu_{R>0}B(w,R)\rt)
=\pi(\^M)
=M,
$$
since the sets $\pi(B(w,R)$ form an ascending family and since all of
them are open, it follows from compactness of $M$ that there exists
$R>0$ such that
\beq\label{1rds239}
\pi(B(w,R))=M.
\eeq
Since the the group of deck transformations of $M$ acts transitively
on each fiber, $\pi^{-1}(x)$, where $x\in M$, and since each deck
mapping is an isometry, we conclude from \eqref{1rds239} that
\beq\label{2rds239}
\pi(B(z,2R))=M \qquad \text{for all } z\in\^M \, .
\eeq
 Now, each map $f_n:M\to M$ has a lift
$\^f_n:\^M\to\^M$ such that $f_n\circ\pi=\pi\circ\^f_n$. Clearly
$\^f_n:\^M\to\^M$ is a diffeomorphism and
$$
\^\rho(\^f_n(x),\^f_n(y))\ge \g(f_n)\^\rho(x,y) \qquad \text{for all } \;\;x,y\in\^M\,.
$$
 Consequently,
$$
\^\rho(\^F_n(x),\^F_n(y))\ge \Pi_{j=1}^n\g(f_j)\^\rho(x,y),
$$
where $\^F_k=\^f_k\circ \^f_{k-1}\circ\ld\circ \^f_1$ and
\beq\label{3rds239}
\pi(B(z,r))\spt B\lt(\^F_n(z),r\Pi_{j=1}^n\g(f_j)\rt)
\eeq
for all $z\in\^M$ and all $r>0$. Now, since $U$ is a non-empty open
subset of $M$, the set $\pi^{-1}(U)$ is open in $\^M$. Thus, there
exists $r>0$ such that
\beq\label{4rds239}
B(x,r)\sbt\pi^{-1}(U).
\eeq
Take $n_r\ge 1$  so large that $r\Pi_{j=1}^n\g(f_j)\ge 2R$. It then
follows from this, \eqref{4rds239}, \eqref{3rds239}, and
\eqref{4rds239}, that
$$
F_{n_r}(U)
=\pi\lt(\^F_{n_r}(\pi^{-1}(U))\rt)
\spt\pi\lt(\^F_{n_r}(B(x,r))\rt)
\spt\pi\lt(B\lt(\^F_{n_r}(x),2R\rt)\rt)
=M.
$$
The first assertion of our proposition is thus proved. The seccond
assertion also follws from this proof by taking
$k_r=E\lt({\log(2R)\over \log r}\rt)+1$.
\endpf

\section{Stationary measures}
Let $M$ be an $n$- dimensional compact Riemannian manifold and let $I$
be a set equipped with a probabilistic measure $m_0$. With every $a\in
I$ we associate a differentiable expanding transformation $f_a$ of $M$
into itself. Put $X=I^\Z$ and let $m$ be the product measure induced
by $m_0$. For $x=\ldots a_{-1}a_0a_1\ldots$ consider $\varphi_x:=-\log
|\det f_{a_0}'|$. We assume that all our assumption are
satisfied. Then the measure $\nu= vol_M$ (where $vol_M$ is the
normalized Riemannian volume on $M$) is the fixed point of
the operator $\cL_{x,\varphi}^*$ with $\lambda_x=1$. Let $q_x$ be the
function given by Theorem~\ref{thm:Gib50A}, and let $\mu_x$ be the
measure determined by $d\mu_x/d\nu_x=q_x$.

We write $I^\Z=I^{-\N}\times I^\N$ where points from $I^{-\N}$ we
denote by $x^-=\ldots a_{-2}a_{-1}$ and from $I^\N$ by
$x^+=a_0a_1\ldots$. Then $x^-x^+$ means $x=\ldots
a_{-1}a_0a_1\ldots$. Note that $q_x$ does not depend on $x^+$, since
nor does $\cL_{x_{-n}}^n\1(y)$. Then we can
write $q_{x^-}:=q_x$ and $\mu_{x^-}:=\mu_x$. Since $\mu_x(g\circ
f_{a_0})=\mu_{\shift(x)}$,
\begin{equation}
  \label{eq:exp12}
  \mu_{x^-}(g\circ f_{a})=\mu_{x^-a}(g)
\end{equation}
for every $a\in I$.

Define a measure $\mu^*$ by $d\mu^*=d\mu_{x^-}dm^-(x^-)$ where $m^-$
is the product measure on $I^{-\N}$. Then by (\ref{eq:exp12})
\begin{displaymath}
  \begin{aligned}
    \int\mu^*(g\circ f_{a}) dm_0(a)&=\int \mu_{x^-}(g\circ
    f_{a})dm^-(x^-) \\
    &=\int\int\mu_{x^-a}(g)dm^-(x^-)dm_0(a)=\mu^*(g).
  \end{aligned}
\end{displaymath}
Therefore, $\mu^*$ is a stationary measure (see for example \cite{Via97}).
%**************************************************** End Stationary measure *****************

%%% Local Variables:
%%% mode: latex
%%% TeX-master: "RDSmain"
%%% End:

%% file: RDSappendix.tex
\chapter{Real Analyticity of Pressure}

\section{The pressure as a function of a parameter}

Here, we will have a careful close look at the measurable bounds
obtained in Chapter~\ref{ch:RPFmain} from which we deduce that the
theorems from that section can be proved to hold for every parameter
and almost every $x$ (common for all parameters).

In this section we only assume that $T:\cJ \to \cJ$ is a measurable
expanding random map.
Let $\varphi^{(1)},\varphi^{(2)} \in \cH_m (\cJ )$ and let $t=(t_1,t_2)\in\R^2$. Put
\begin{equation}
  \label{eq:20}
  |t|:=\max\{|t_1|,|t_2|\} \ \text{ and } t^*:=\max\{1,|t|\}.
\end{equation}
Set
$
  \ph_t=:t_1\varphi^{(1)}+t_2\varphi^{(2)}
$
and
\begin{equation}
  \label{eq:defphi}
  \ph:=|\varphi^{(1)}|+|\varphi^{(2)}|.
\end{equation}
Fix $\a>0$ and a measurable log-integrable function
$H:X\to[0,+\infty)$ such that
$\varphi^{(1)}, \varphi^{(2)}\in\cH^\alpha_m(\J,H)$. Then for all $x\in X$
and all $y_1,y_2\in\J_x$, we have
$$
\aligned
|\ph_{t,x}(y_2)-\ph_{t,x}(y_1)|
\le H_x|t_1|\rho_x^\a(y_2,y_1)+H_x|t_2|\rho_x^\a(y_2,y_1)
\le 2|t|H_x\rho_x^\a(y_2,y_1)
\endaligned
$$
Therefore
$
\ph_t\in \cH^\alpha_m(\J,2|t|H)\sbt \cH^\alpha_m(\J,2t^*H)
$.
Also, for all $x\in X$ and all $y\in \J_x$, we have
$$
\aligned
|S_n\ph_{t,x}(y)|
\le |t_1||S_n\ph^{(1)}_x(y)|+|t_2||S_n\ph^{(2)}_x(y)|
\le |t||S_n\ph_x(y)|
\le |t|||S_n\ph_x||_\infty.
\endaligned
$$
This implies
\beq\label{1RDS208}
||S_n\ph_{t,x}||_\infty
\le |t|||S_n\ph_x||_\infty
\le t^*||S_n\ph_x||_\infty.
\eeq
Concerning the potential $\ph$, we get
$$
\aligned
|\ph_x(y_2)-\ph_x(y_1)|
\le \lt||\varphi_x^{(1)}(y_2)-\varphi_x^{(1)}(y_1)\rt|
        +\lt|\varphi_x^{(2)}(y_2)-\varphi_x^{(2)}(y_1)\rt|
\le 2H_x\rho_x^\a(y_2,y_1).
\endaligned
$$
Thus \beq\label{2RDS208} \ph\in \cH^\alpha_m(\J,2H).  \eeq Denote by
$C_t$, $C_{t,\max}$, $C_{t,\min}$, $D_{\xi,t}$ and $\b_t(s)$, the
respective functions associated to the potential $\ph_t$ as in
Chapter~\ref{sec:FUC}. If the index $t$ is missing, these numbers, as
usually, refer to the potential $\ph$ given by
\eqref{eq:defphi}. Using \eqref{1RDS208} and \eqref{2RDS208}, we then
immediately get
\beq\label{1rds209}
D_{\xi,t}(x)\ge D_{\xi,\ph}^{t^*},
\eeq
\beq\label{1rds211}
\aligned C_t(x) &\le
\exp\(Q_x(2t^*H)\)\max_{0\le k\le j}
\lt\{\exp\(2t^*||S_k\ph_{x_{-k}}||_\infty\)\rt\}\\
&\le \lt(\exp\(Q_x(2H)\)
\max_{0\le k\le j}\lt\{\exp\(2||S_k\ph_{x_{-k}}||_\infty\)\rt\}\rt)^{t^*}
= C_\ph^{t^*},
\endaligned
\eeq
\beq\label{3rds211}
C_{t,\min}(x)
\ge \exp\(-Q_x(2t^*H)\)\exp\(-2t^*||S_n\ph_x||_\infty\)
=C_{\min}(x)^{t^*},
\eeq
\beq\label{2rds211}
C_{t,\max}(x)
= \exp\(Q_x(2t^*H)\)\deg(T_x^n)\exp\(2t^*||S_n\ph_x||_\infty\)
\le C_{\max}(x)^{t^*},
\eeq
and therefore,
\begin{displaymath}
  \label{eq:P10}
\aligned
  \beta_{t,x}(s)
&\ge \lt(\frac{C_{\min}(x)}{C_\varphi(x)}\rt)^{t_*}
        \frac{(s-1)2t^*H_{x_{-1}}\gamma_{x_{-1}}^{-\alpha}}{4t^*sQ_{x}}
=   \lt(\frac{C_{\min}(x)}{C_\varphi(x)}\rt)^{t_*}
        \frac{(s-1)H_{x_{-1}}\gamma_{x_{-1}}^{-\alpha}}{2sQ_{x}}\\
&\ge \lt(\frac{C_{\min}(x)}{C_\varphi(x)}\rt)^{t_*}
        \lt(\frac{(s-1)H_{x_{-1}}\gamma_{x_{-1}}^{-\alpha}}{2sQ_{x}}\rt)^{t^*}
= \beta_x^{t_*}(s).
\endaligned
\end{displaymath}
Finally we are going to look at the function $A(x)$ and the constant $B$ obtained in
Proposition~\ref{4.16}. We fix the set
\begin{displaymath}
  G:=\{x:\beta_x\geq M \textrm{ and } j(x)\leq J\}
\end{displaymath}
as defined by (\ref{eq:defG}). Note that by (\ref{eq:P10}), for $x\in
G$ we have,
$\beta_{x,t}\geq M^{t_*}
$.
Denote by $G_-'$ the corresponding visiting set for backward iterates
of $\shift$, and by $(n_k)_1^\infty$ the corresponding visiting
sequence. In particular
$
  \lim_{k\to\infty}\frac{k}{n_k}\geq \frac{3}{4J}.
$
Putting
$
B_t=\sqrt[4J^2]{1-M^{t_*}}
$
and
\begin{displaymath}
  A_t(x):=\max\{2C_{\max}^{t_*}(x)B_t^{-Jk^*_x},
  C_\varphi^{t_*}(x)+ C_{\max}^{t_*}(x)\},
\end{displaymath}
as an immediate consequence of Proposition~\ref{4.16} and its proof
along with our estimates above, we obtain the following.

\begin{prop}
  \label{prop:P51}
  For every $t\in\R^2$, for every $x\in G_-'$, and every $g_x\in
\Lambda_{t,x}^s$
  $$
\| \tcL ^n_{x_{-n},t} g_{x_{-n}} -q_{t,x} \|_\infty \leq A_t(x)
B_t^n.
$$
  More generally, if $g_x\in \cH^\alpha(\J_x)$, then
  \begin{multline*}
    \Big|\Big|\hcL_{t,x}^n g_x -
    \Big(\int g_xd\mu_{t,x}\Big)\1\Big|\Big|_\infty
    \leq
    C^{t_*}_{\vp } (\shift ^n (x))\Big(\int |g_x| \, d\mu_{t,x}
     +4\frac{v_\alpha(g_xq_{t,x})}{t_*Q_x}\Big)A_{t_*}(\shift ^n(x))B_{t_*}^n.
  \end{multline*}
In here and in the sequel, by $q_{t,x}$, $\Lambda_{t,x}^s$ and
$\cL_{t,x}$ we  denote the respective objects for the potential $\varphi_{t}$.
\end{prop}

\begin{rem}\label{rem:P55}
  It follows from the estimates of all involved measurable functions, that,
  for $R>0$ and $t\in \R$ such that $|t|\leq R$, the functions $A_t$
  and $B_t$ in Proposition~\ref{prop:P51} can be replaced by $A_{\max\{R,1\}}$
  and $B_{\max\{R,1\}}$ respectively.
\end{rem}

Now, let us look at Proposition~\ref{4.19NEW}. Similarly as with the
set $G$,
we consider the set $X_A$ defined by (\ref{eq:defXA}) with $A(x)$
generated by $\ph$. So, if $x\in X_A$, then $A_t(x)\leq \cA_t$ for
some finite number $\cA_t$ which depends on $t$. Denote by $X_{A,+}'$
the corresponding visiting set intersected with $G_+'$. Therefore,
the following is a consequence of the proof of
Proposition~\ref{4.19NEW} and the formula (\ref{eq:dDirc}).

\begin{prop}\label{prop:P67}
  For every  $R>0$, every $x\in X_{A,+}'$, and every
  $g_x\in\cC(\J_x)$ we have that
\begin{displaymath}
  \lim_{n\to\infty}\sup_{|t|\le R}\lt\{\Big|\Big|\hcL_{t,x}^n g_x -
  \Big(\int g_xd\mu_{t,x}\Big)\1_{\shift^n(x)}\Big|\Big|_\infty\rt\}
  =0.
  \end{displaymath}
 % and, for all $x\in X_{A,+}'$ and all $t\in\R^2$,
 % \begin{equation}
 %   \nu_{x,t,n}\xrightarrow[n\to\infty ]{ }\nu_{t,x}
 % \end{equation}
 % in the weak* topology, where
 % \begin{displaymath}
 %   \nu_{x,t,n}:=\frac{\sum_{y\in T_x^{-n}(w_n)}e^{S_n\ph_{t,x}(y)}\delta_y}
 %   {\cL_{t,x}^n \1 (w_n)}
 % \end{displaymath}
 % and, as usually, $\delta_y$ is the Dirac measure supported at $y$.
\end{prop}

\fr Moreover, we obtain the following consequence of
Lemma~\ref{lem:Gibbs} and \eqref{1rds209}.
\begin{lem}
  \label{lem:PGibbs}
  There exist a set $X'\sbt X$ of full measure, and a measurable function
  $X\ni x\mapsto D_1(x)$ with the following property. Let $x\in X'$, let $w\in
  \cJ_x$ and let $n\geq 0$. Put $y=(x,w)$. Then
  \begin{displaymath}
    (D_1(\shift^n(x)))^{-t_*}
    \leq
    \frac{\nu_{t,x}(T^{-n}_y(B(T^n(y),\xi)))}
    {\exp(S_n\varphi_t(y)-S_n P_{x}(\varphi_t))}
    \leq (D_1(\shift^n(x)))^{t_*}
  \end{displaymath}
for all $t\in \R^2$.
\end{lem}
For all $t\in \R^2$ set
$$
\cE P(t):=\cE P(\varphi_t).
$$
We now shall prove the following.

\begin{lem}
  \label{lem:pre4.25}
  The function $\cE P:\R^2\to \R$ is convex, and therefore,
  continuous. There exists a measurable set $X_\cE'$ such that $m(X_\cE')=1$ and for
  all $x\in X_\cE'$ and all $t\in\R^2$, the limit
\beq\label{1rds215}
\lim_{n\to\infty}\frac{1}{n}\log\cL_{t,x}^n\1(w_n)
\eeq
exists, and is equal to $\cE P(t)$.
\end{lem}
\begin{proof} By Lemma~\ref{lem:FormOfPressPart} and
  Lemma~\ref{lem:11} we know that for every $t\in\R^2$ there exists a
  measurable $X_t'$ with $m(X_t')=1$ and such that \beq\label{1rds227}
  \lim_{n\to\infty}\frac{1}{n}\log\cL_{t,x}^n\1 (w_n)
  =\lim_{n\to\infty}\frac{1}{n}\log\lambda_{t,x}^n =\cE P(t) \eeq for
  all $x\in X_t'$. Fix $\lambda\in [0,1)$ and let $t=(t_1,t_2)$ and
  $t'=(t_1',t_2')\in\R^2$. H\"older's inequality implies that
  all the functions $\R^2\ni t\mapsto\frac{1}{n}
\log\cL_{t,x}^n\1 (w_n)$, $n\ge 1$, are convex. It thus follows from
\eqref{1rds227}, that the function $\R^2\ni t\mapsto\cE P(t)$ is
convex, whence continuous. Let
$$
X_\cE'=\bi_{t\in \Q^2}X_t'.
$$
Since the set $\Q^2$ is countable, we have that $m(X_\cE')=1$. Along with
\eqref{1rds227}, and density of $\Q^2$ in $\R^2$, the convexity of the
functions $\R^2\ni t\mapsto\frac{1}{n}\log\cL_{t,x}^n\1 (w_n)$ implies
that for all $x\in X_\cE'$ and all $t\in\R^2$, the limit
$
\lim_{n\to\infty}\frac{1}{n}\log\cL_{t,x}^n\1 (w_n)
$
exists and represents a convex function, whence continuous. Since for
all $t\in \Q^2$ this continuous function is equal to the continuous
function $\cE P$, we conclude that for all $x\in
X_\cE'$ and all $t\in\R^2$, we have
$$
\lim_{n\to\infty}\frac{1}{n}\log\cL_{t,x}^n\1 (w_n)=\cE P(t).
$$
We are done.
\end{proof}

\begin{lem}
  \label{4.25} Fix $t_2\in\R$ and assume that there exist measurable
  functions $L:X\ni x\mapsto L_x\in\R$ and $c: X\ni x \mapsto c_x>0$
  such that
  \begin{equation}
    \label{eq:Pre100}
    S_n\varphi_{x,1} (z) \leq - n c_x +L_x \quad \text{for every} \;\;z\in \J_x \;\;\text{and}\;\; n\ge 1\, .
  \end{equation}
   Then the function $\R\ni
  t_1\mapsto\cE P(t_1,t_2)\in \R$ is strictly decreasing and
  \begin{equation}
    \label{4.23}
    \lim_{t_1\to+\infty}\cE P(t_1,t_2)=-\infty\quad and \quad
    \lim_{t_1\to-\infty}\cE P(t_1,t_2)=+\infty \quad\;\; m-a.e.
  \end{equation}
\end{lem}
\begin{proof}
  Fix $x\in X_\cE'$.  Let $t_1< t_1'$.  Then by (\ref{eq:Pre100})
$$
\aligned
\sum_{z\in T^{-n}_x(w_n)}\exp&\(S_n\ph_{(t_1,t_2)}(z)\)=
  \sum_{z\in T^{-n}_x(w_n)}\exp\(t_1S_n\varphi_1(z)\)\exp\(t_2S_n\varphi_2(z)\)\\
&=  \sum_{z\in T^{-n}_x(w_n)}\exp\(t_1'S_n\varphi_1(z)\)\exp\(t_2S_n\varphi_2(z)\)
         \exp\((t_1-t_1')S_n\varphi_1(z)\)\\
&\ge\sum_{z\in T^{-n}_x(w_n)}
    \exp\(t_1'S_n\varphi_2(z)\)\exp\(t_2S_n\varphi_2(z)\)
    \exp\((t_1-t_1')(L_x-nc_x)\) \\
&=  \sum_{z\in T^{-n}_x(w_n)}\exp\(S_n\ph_{(t_1',t_2)}(z)\)
         \exp\((t_1'-t_1)(nc_x-L_x)\)
\endaligned
$$
  Therefore,
$$
\aligned
  \frac{1}{n}\log \cL_{t,x}^n\1 (w_n)
  \geq \frac{1}{n}\log \cL_{(t_1',t_2),x}^n\1 (w_n)
  + (t_1'-t_1)(c_x-L_x/n)
  \, .
\endaligned
$$
Hence, letting
  $n\to\infty$, we get from Lemma~\ref{lem:pre4.25} that
    $\cE P(t_1,t_2) \geq
  \cE P(t_1',t_2) + (t_1'-t_1)c_x$.
  It directly follows from this
  inequality that the function $t_1\mapsto \cE P(t_1,t_2)$ is strictly
  decreasing, that
  $ \lim_{t_1\to+\infty}\cE P(t_1,t_2)=-\infty$ and that $\lim_{t_1\to-\infty}\cE P(t_1,t_2)=+\infty.$
\end{proof}

\section{Real cones}
\label{sec:real-cones}
We adapt the approach of Rugh \cite{Rug07} based on complex cones and
establish real analyticity of the pressure function.  Via Legendre
transformation, this completes the proof of real analyticity of the
multifractal spectrum (see Chapter~\ref{cha:mult-analys}).

Let $\cH_{x}:=\cH_{\R,x}:=\holderx$ and let $\cH_{\C,x}:=\cH_{\R,x}\oplus
i \cH_{\R,x}$ its complexification.
\begin{equation}
  \label{eq:21}
  \cC_{x}^s:=\cC_{\R,x}^s:=\{g\in \cH_x:g(w_1)\leq
  e^{sQ_x\varrho^\alpha(w_1,w_2)}g(w_2)
  \textrm{ if }\varrho(w_1,w_2)\leq \xi\}.
\end{equation}
Whenever it is clear what we mean by $s$, we also denote this cone by
$\cC_x$.

By $\cC_x^+$ we denote the subset of all non-zero functions from
$\cC_{x}^s$.  For $l\in (\cH_x)^*$, the dual space of $\cH_x$, we
define
\begin{displaymath}
  K(\cC_{x}^s,l):=\sup_{g\in \cC^+_{x}}\frac{||l||_\alpha||g||_\alpha}
  {|\langle l,g\rangle|}.
\end{displaymath}
Then the \emph{aperture} of $\cC_x^s$ is
\begin{displaymath}
  K(\cC_x^s):=\inf\{K(\cC_{x}^s,l):l\in (\cH_x)^*,l\neq 0\}.
\end{displaymath}

\begin{lem}
  $K(\cC_x^s)<\infty$. This property of a cone is called an
  \emph{\index{outer regularity}outer regularity}.
\end{lem}
\begin{proof}
  Let $w_k\in \J_x$, $k=0,\ldots,N$ be such that
$
  \bigcup_{k=1}^{L_x}B(w_k,\xi)=\J_x
$.
Define
\begin{equation}
  \label{eq:R120}
  l_0(g):=\sum_{k=1}^{L_x}g(w_k).
\end{equation}
Then by Lemma~\ref{4.11TXE} we have
\begin{eqnarray*}
  ||g||_\alpha\leq \Big(sQ_x(\exp(sQ_x \xi^\alpha))+1\Big)||g||_\infty
  \leq
  \Big(sQ_x(\exp(sQ_x \xi^\alpha))+1\Big)\exp(sQ_x \xi^\alpha)l_0(g).
\end{eqnarray*}
Note that $\|l_0\|_\alpha= L_x$, since $l_0(g)\leq L_x
||g||_\infty\leq L_x||g||_\alpha$ and $l_0(\1)= L_x =L_x ||\1||_\alpha$.  Hence
\begin{equation}
  \label{eq:R130}
  \frac{\|l_0\|_\alpha\|g\|_\alpha}{\langle l_0,g\rangle}\leq K_x'
  :=L_x\Big(sQ_x(\exp(sQ_x \xi^\alpha))+1\Big)\exp(sQ_x \xi^\alpha).
\end{equation}
\end{proof}

Let
\begin{displaymath}
  s_x':=\frac{sQ_{x_{-1}}\gamma_{x_{-1}}^{-\alpha}+
    H_{x_{-1}}\gamma_{x_{-1}}^{-\alpha}}{Q_x}.
\end{displaymath}
By (\ref{eq:eQ1-1}) for $s>1$, $s_x'<s$. Moreover, like in
(\ref{eq:KeyToExpConv}) we have the following.
\begin{lem}
  \label{lem:vg}
  Let $g\in \cC_x^s$ and let $w_1,w_2\in \J_{\shift (x)}$ with
  $\vr(w_1,w_2)\leq\xi$.  Then, for $y\in T^{-1}_x(w_1)$
  \begin{equation}
    \label{eq:vg}
    \frac{e^{\varphi(y)}}{e^{\varphi(T^{-1}_y(w_2))}}\frac{g(y)}{g(T^{-1}_y(w_2))}
    \leq\exp \left\{ s_{\shift(x)}'
      Q_{\shift(x)} \vr ^\al (w_1,w_2 )\right\}.
  \end{equation}
  Consequently
  $$\frac{\cL_x g(w_1)}{\cL_x g(w_2)} \leq \exp \left\{ s_{\shift(x)}'
    Q_{\shift(x)} \vr ^\al (w_1,w_2 )\right\},$$
\end{lem}

\begin{lem}
  \label{R4.14} There is a measurable function $C_{R} : X\to
  (0,\infty )$ such that
  $$\frac{\cL ^i _{x_{-i}} g(w)}{\cL ^i _{x_{-i}} g(z)}
  \leq C_{R} (x)
  \quad \textrm{for every}\; i\geq j(x) \; and \; g\in \cC_x^s.$$
\end{lem}

\begin{proof}
  First, let $i=j(x)$. Let $a\in T_{x_{-i}}^{-i}(z)$ be such that
  \begin{displaymath}
    e^{S_i\varphi(a)}g(a)
    =\sup_{y\in T_{x_{-i}}^{-i}(z)}e^{S_i\varphi(y)}g(y).
  \end{displaymath}
  By definition of $j(x)$, for any point $w\in \J_x$ there exists
  $b\in T_{x_{-i}}^{-i}(w)\cap B(a,\xi)$. Therefore
  \begin{displaymath}
    \begin{aligned}
          \cL_{x_{-i}}^ig (w)&\geq e^{S_i\varphi_{x_{-i}}(b)}g(z)
    \geq
    \exp(S_i\varphi_{x_{-i}}(b)-S_i\varphi_{x_{-i}}(a))
    e^{S_i\varphi_{x_{-i}}(a)}e^{-sQ_x}g(a)\\
    &\geq
    \frac{\exp(-2\|S_{j(x)}\vp _{x_{-j(x)}}\|_\infty-sQ_x)}
    {\deg(T_{x_{-j}}^j)}\cL^i _{x_{-i}}
    g(z)\geq (C_{R}(x))^{-1}\cL ^i _{x_{-i}} g(z)
    \end{aligned}
  \end{displaymath}
  where
  \begin{equation}
    \label{eq:defCR}
    C_{R}(x):=\left(\frac{
        \exp\Big(-sQ_x-2\|S_{j(x)}\vp _{x_{-j(x)}}\|_\infty\Big)}
    {\deg(T_{x_{-j}}^{j(x)})}\right)^{-1}\geq 1.
  \end{equation}
  The case $i > j(x)$ follows from the previous one, since
  $\cL_{x_{-i}}^{i-j(x)}g_{x_{-i}}\in \cC_{x_{-j(x)}}^s$.
\end{proof}

  Let $s>1$ and $s'<s$.
Define
\begin{equation}
  \label{eq:Rbetax}
  \tau_x:=\tau_{x,s,s'}:=\sup_{r\in(0,\xi]}\frac{1
    -\exp\big(-(s+s')Q_xr^\alpha\big)}{1
    -\exp\big(-(s-s')Q_{x}r^\alpha\big)}\leq \frac{s+s'}{s-s'}.
\end{equation}

\begin{lem}
  \label{lem:R4.15} For
  $g_{x},f_{x}\in \cC_{x}^{s'}$,
  $$\tau_x\frac{\sup_{y\in \J_x}|g_x(y)|}{\inf_{y\in \J_x}|f_x(y)|}
  f_x- g_x\in \cC^{s}_{\R,x}.$$
\end{lem}
\begin{proof}
For all $w,z\in \J_x$ with $\varrho_x(z,w)<\xi$,
$$\begin{aligned}
  \tau_x\|g_x/f_x\|_\infty&
  \Big(\exp\big(sQ_{x}\varrho_x^\alpha(z,w)\big)f_x(z)-f_x(w)\Big)\\
  &\geq
    \tau_x\|g_x/f_x\|_\infty
    \Big(\exp\big(sQ_{x}\varrho_x^\alpha(z,w)\big)
    -\exp\big(s'Q_{x}\varrho_x^\alpha(z,w)\big)\Big)f_{x}(z)\\
  &\geq\Big(\exp\big(sQ_x\varrho_x^\alpha(z,w)\big)
  -\exp\big(-s'Q_x\varrho_x^\alpha(z,w)\big)\Big)
  g_{x}(z)\\
  &\geq\exp\big(sQ_{x}\varrho_x^\alpha(z,w)\big)g_x(z)-g_x(w).
\end{aligned}$$
Then
$
  \exp\big(sQ_{x}\varrho_x^\alpha(z,w)\big)
  \Big(\tau_x\|g/f\|_\infty f_x(z)- g_x(z)\Big)
  \geq \tau_x\|g/f\|_\infty f_x(w)   -g_x(w)
$.
\end{proof}

We say that $g_x\in \cC_x^{s}$ is \emph{balanced}\index{balanced} if
\begin{equation}
  \label{eq:12}
  \frac{f_x(y_1)}{f_x(y_2)}\leq C_R(x) \quad \text{for all} \;\; y_1,y_2\in \J_x.
\end{equation}
Let $g_x,f_x\in \cC_x^{s}$. Put
$
  \beta_{x,s}(f_x,g_x):=\inf\{\tau>0:\tau f_x-g_x\in\cC_x^{s}\}
$
and define the\index{Hilbert projective distance}\index{projective
  distance} \emph{Hilbert projective distance} $\pdist:\cC_x^{s}\times
\cC_x^{s}\to\R$ by the formula
\begin{displaymath}
  \pdist_{x}(f_x,g_x):=\pdist_{x,s}(f_x,g_x):=\log(\beta_{x,s}(f_x,g_x)\cdot\beta_{x,s}(g_x,f_x)).
\end{displaymath}
Let
\begin{displaymath}
  \Delta_{x}:=\diam_{\cC_{x,\R}^s}(\cL ^j _{x_{-j}}(\cC^{s}_{x_{-j},\R})),
\end{displaymath}
where $\diam_{\cC_{x,\R}^s}$ is the diameter with respect to the
projective distance and $j=j(x)$. Then by Lemma~\ref{lem:vg}, Lemma~\ref{R4.14} and
Lemma~\ref{lem:R4.15} we get the following.
\begin{lem}
  \label{lem:21}
  If $g_{x},f_{x}\in \cC_{x}^{s'}$ are balanced, then
  \begin{displaymath}
    \pdist_{x}(f_x,g_x)\leq 2\log\Big(\frac{s+s'}{s-s'}\cdot C_R(x)\Big)
  \end{displaymath}
 and, consequently,
  \begin{displaymath}
    \Delta_{x}\leq 2\log\Big(\frac{s+s'}{s-s'}\cdot C_R(x)\Big).
  \end{displaymath}
\end{lem}

\section{Canonical complexification}
Following the ideas of Rugh \cite{Rug07} we now extend real cones
to complex ones. Define
$
  \cC_x^*:=\{l\in (\cH_x)^*:l|_{\cC_x}\geq 0\}
$ and
\begin{displaymath}
  \cC_{\C,x}^s:=\{g\in \cH_{\C,x}:\forall_{l_1,l_2\in\cC_x^*}
  \re \langle l_1,g\rangle\overline{\langle l_2,g\rangle}\geq 0\}.
\end{displaymath}
Denote also by $\cC_{\C,x}^+$ the set of all $g\in\cC_{\C,x}^s$ such that
$g\not\equiv 0$.
There are other equivalent definitions of $\cC_{\C,x}^s$. The first one
is called \emph{polarization identity} by Rugh in \cite[Proposition
5.2]{Rug07}.
\begin{prop}[Polarization identity]\index{polarization identity}
  \label{prop:polarization}
  \begin{displaymath}
    \cC_{\C,x}^s=\{a(f^*+ig^*):f^*\pm g^*\in \cC_{\R,x}^+\textrm{ and }a\in\C\}.
  \end{displaymath}
\end{prop}

In our case we can also define $\cC_{\C,x}^s$ as follows.  Let
$\varrho(w,w')<\xi$. Define
\begin{displaymath}
  l_{w,w'}(g):=g(w)-e^{-sQ_x\varrho^\alpha(w,w')}g(w')
\end{displaymath}
and
\begin{displaymath}
  F_x:=\{l_{w,w'}:\varrho(w,w')<\xi\}\subset\cC_{x}^*.
\end{displaymath}
Then
\begin{displaymath}
  \cC_x^s=\{g\in \cH_x:\forall_{l\in F_x}l(g)\geq 0\}.
\end{displaymath}

Later in this section we use the following two facts about geometry of complex
numbers. The first one is obvious and the second is Lemma 9.3 from
\cite{Rug07}.
\begin{lem}
  \label{lem:18}
  Given $c_1,c_2>0$ there exist $p_1, p_2>0$ such that if $s_0:=c_1
  p_2$ and
  \begin{displaymath}
    Z\in\{re^{iu}:1\leq 1+s_0^2\textrm{, }|u|\leq 2p_1+2s_0\},
  \end{displaymath}
  then there exist $\alpha, \beta, \gamma>0$ such that
  $    \re Z\geq \alpha$, $\re Z\leq \beta$, $\im Z\leq \gamma$ and $\gamma c_2<\alpha$.
\end{lem}

\begin{lem}
  \label{lem:rugh9.3}
  Let $z_1,z_2\in\C$ be such that $\re z_1>\re z_2$ and define $u\in
  \C$ though
  \begin{displaymath}
    e^{i\im z_1}u\equiv \frac{e^{z_1}-e^{z_2}}{e^{\re z_1}-e^{\re z_2}}.
  \end{displaymath}
  Then
  \begin{displaymath}
    |\Arg u|\leq \frac{|\im(z_1-z_2)|}{\re(z_1-z_2)} \;
    \text{ and } \;
    1\leq |u^2|\leq 1+\Big(
    \frac{\im(z_1-z_2)}{\re(z_1-z_2)}\Big)^2.
  \end{displaymath}
\end{lem}

Let $\varphi=\re \varphi + i\im \varphi$ be such that $\re \varphi,
\im \varphi\in \cH^\alpha(\J)$. We now consider the corresponding
complex Perron-Frobenius operators $\cL_{x,\varphi}$ defined by
\begin{displaymath}
  \cL_{x,\vp}g_x (w) = \sum_{T_x(z)=w}
  e^{\vp _x(z)}g_x(z) ,\quad w\in \cJ_{\shift (x)}.
\end{displaymath}
\begin{lem}
  \label{lem:20}
  Let $w,w',z,z'\in\J_x$ such that $\vr(w,w')<\xi$ and
  $\vr(z,z')<\xi$. Then, for all $g_1,g_2\in \cC_{x,\R}^s$,
  \begin{displaymath}
    \frac{l_{w,w'}(\cL_{x,\varphi}g_1)\overline{l_{z,z'}(\cL_{x,\varphi}g_2)}}
    {l_{w,w'}(\cL_{x,\re\varphi}g_1)l_{z,z'}(\cL_{x,\re\varphi}g_2)}=Z
  \end{displaymath}
  where
  \begin{equation}
    \label{eq:R210}
    Z\in A_x
    :=\{re^{iu}:1\leq r\leq 1+s_0^2, |u|\leq 2||\im \varphi||_\infty+2s_0\}.
  \end{equation}
  and
  \begin{equation}
    \label{eq:R200}
    s_0 :=\frac{v_{\alpha}(\im \varphi)\gamma_x^{-\alpha}}
    {(s-s_{\shift(x)}')Q_{\shift(x)} }
  \end{equation}
\end{lem}
\begin{proof}
  For $y\in T^{-1}_x(w)$, by $y'$ we denote $T^{-1}_y(w')$. Then for $g\in \cC_x$
\begin{displaymath}
  \begin{aligned}
    l_{w,w'}(\cL_{x,\varphi}g)&:=\cL_{x,\varphi}
    g(w)-e^{-sQ_x\varrho^\alpha(w,w')}
    \cL_{x,\varphi}g(w')\\&=\sum_{y\in T^{-1}_x(w)}e^{\varphi(y)}g(y)
    -e^{-sQ_x\varrho^\alpha(w,w')}e^{\varphi(y')}g(y')
    =\sum_{y\in T^{-1}_x(w)}n_{y}(\varphi,g),
  \end{aligned}
\end{displaymath}
where
\begin{displaymath}
  n_{y}(\varphi,g):=e^{\varphi(y)}g(y)
  -e^{-sQ_x\varrho^\alpha(w,w')}e^{\varphi(y')}g(y').
\end{displaymath}
Define implicitly $u_y$ so that
$
  n_y(\re\varphi,g) e^{i\im \varphi (y)}u_y = n_{y}(\varphi,g).
$
Put $z_1:=\varphi(y)+\log g(y)$ and
$z_2:=-sQ_x\varrho^\alpha(w,w')+\varphi(y')+\log g(y')$.
Then
\begin{displaymath}
  e^{i\im z_1}u_y=\frac{e^{z_1}-e^{z_2}}{e^{\re z_1}- e^{\re z_2}}.
\end{displaymath}
By (\ref{eq:vg})
\begin{displaymath}
   \re\varphi(y)- \log g(y)-(\re\varphi(y')+\log g(y'))\geq
    -s_{\shift(x)}'
      Q_{\shift(x)} \vr ^\al (w_1,w_2 ).
\end{displaymath}
Hence
\begin{displaymath}
  \re (z_1-z_2)\geq (s-s_{\shift(x)}')
      Q_{\shift(x)} \vr ^\al (w_1,w_2 ).
\end{displaymath}
We also have that
\begin{displaymath}
  |\im(z_1 -z_2)|\leq v_{\alpha}(\im \varphi)\gamma_x^{-\alpha}\vr^\alpha(w_1,w_2),
\end{displaymath}
since $\im(z_1 -z_2)=\im \varphi(y)-\im \varphi(y')$.
Therefore, by Lemma~\ref{lem:rugh9.3}
\begin{displaymath}
  |\Arg u_y|\leq s_0 :=\frac{v_{\alpha}(\im \varphi)\gamma_x^{-\alpha}}
  {(s-s_{\shift(x)}')
      Q_{\shift(x)} }
\quad
\text{and}
\quad
  1\leq |u_y|^2\leq 1+ s_0^2.
\end{displaymath}
Since
\begin{displaymath}
  l_{w,w'}(\cL_{x,\varphi}g)=\sum_{y\in T^{-1}_x(w)}n_{y}(\varphi,g)=
  \sum_{y\in T^{-1}_x(w)}e^{i\im \varphi (y)}u_y n_y(\re\varphi,g),
\end{displaymath}
\begin{displaymath}
  \frac{l_{w,w'}(\cL_{x,\varphi}g)}{l_{w,w'}(\cL_{x,\re\varphi}g)}=Z
\end{displaymath}
where
\begin{displaymath}
  Z\in A_x
  :=\{re^{iu}:1\leq r\leq 1+s_0^2, |u|\leq 2||\im \varphi||_\infty+2s_0\}.
\end{displaymath}
Similarly
\begin{displaymath}
  \frac{l_{w,w'}(\cL_{x,\varphi}g_1)\overline{l_{z,z'}(\cL_{x,\varphi}g_2)}}
  {l_{w,w'}(\cL_{x,\re\varphi}g_1)l_{z,z'}(\cL_{x,\re\varphi}g_2)}=Z
\end{displaymath}
for possibly another $Z \in A_x$.
\end{proof}

Let $p_1,p_2$ be the real numbers given by Lemma~\ref{lem:18} with
\begin{displaymath}
  c_1=\frac{\gamma_x^{-\alpha}}{(s-s_{x}')
    Q_{x}}\textrm{ and }c_2=\cosh\frac{\Delta_x}{2}.
\end{displaymath}

Having Lemma~\ref{lem:20}, Lemma~\ref{lem:18} and Lemma~\ref{lem:21}
the following proposition is a consequence of the proof of Theorem 6.3
in \cite{Rug07}.
\begin{prop}
  \label{prop:19}
  Let $j=j(x)$.  If
  \begin{equation}
    \label{eq:19}
    \|\im S_j\vp_{x_{-j}}\|_\infty\leq p_1\quad
    \textrm{ and }\quad v_\alpha(\im S_j\vp_{x_{-j}})\leq p_2,
  \end{equation}
  then
  \begin{displaymath}
    \cL_{x_{-j}}^j(\cC_{\C,{x_{-j}}}^s)\subset\cC_{\C,x}^{s}.
  \end{displaymath}
\end{prop}

Let $l_0$ (the functional defined by (\ref{eq:R120})). Then by
Lemma 5.3 in \cite{Rug07} we get
\begin{displaymath}
  K:=K(\cC_{\C,x}^s,l_0):=\sup_{g\in \cC^+_{\C,x}}\frac{||l_0||_\alpha||g||_\alpha}
  {|\langle l_0,g\rangle|}\leq K_x:=2\sqrt{2}K'_x
\end{displaymath}
where $K'_x$ is defined by (\ref{eq:R130}). By $l$ we denote the
functional which is a normalized version of $(1/L_x)l_0$. So
$||l||_\alpha=1$. Then, for every $g\in \cC_{\C,x}^s$,
\begin{equation}
  \label{eq:R220}
  1\leq \frac{||g||_\alpha}{\langle l, g\rangle}\leq K_x.
\end{equation}

\section{The pressure is real-analytic}
\label{sec:pres-real-analyt}

%\section{Appendix B: Real Analyticity of the Pressure}

We are now in position to prove the main result of this appendix.
Here, we assume that $T:\J \to \J$ is uniformly expanding random map.
Then there exists $j\in\N$ such that $j(x)=j$ for all $x\in
X$. Without loss of generality we assume that $j=1$.

\begin{thm}
  \label{thm:RAN}
  Let $t_0=(t_1,\ldots,t_n)\in \R^n$, $R>0$ and let
   \begin{displaymath}
      D(t_0, R):=\{z=(z_1,\ldots,z_n)\in\C^n:\forall_{k}\,|z_k-t_k|<R\}.
    \end{displaymath}
     Assume that the following conditions are satisfied.
  \begin{itemize}
  \item[(a)] For every $x\in X$ and every $w\in \J_x$,
    $z\mapsto \varphi_{z,x}(w)$ is
    holomorphic on $D(t_0,R)$.
  \item[(b)] For $z\in \R^n\cap D(t_0,R)$,
    $\varphi_{z,x}\in \cH_{\R,x}$.
  \item[(c)] For all $z\in D(t_0, R)$ and all $x\in
    X$, there exists $H$ such that
   $
      \|\varphi_{z,x}\|_\alpha\leq H.
    $
  \item[(d)] For every $\varepsilon>0$ there exists $\delta >0$ such
    that for all $z\in D(t_0,\delta)$ and all $x\in X$,
    \begin{displaymath}
      \|\im \varphi_{z,x} \|_\alpha\leq \varepsilon.
    \end{displaymath}
  \end{itemize}
  Then the function $D(t_0,R)\cap \R^n\ni z\mapsto
  \cE P(\varphi_{z})$ is real-analytic.
\end{thm}

\begin{proof}
  Since we assume that the measurable constants are uniform for $x\in
  X$ we get that from Proposition~\ref{prop:19} and condition
  (d) that there exists $r>0$ such that, for all $z\in
  D(t_0,r)$ and all $x\in X$,
  \begin{displaymath}
    \cL_{z,x_{-1}}(\cC_{\C,x_{-1}}^s)\subset \cC_{\C,x}^s.
  \end{displaymath}
  Then by (\ref{eq:R220}),
  \begin{displaymath}
    \frac{||\cL_{z,x_{-n}}^{n}(\1)||_\alpha}{l_x(\cL_{z,x_{-n}}^{n}(\1))}\leq K.
  \end{displaymath}
  Therefore, by Montel Theorem, the family
  $\frac{\cL_{z,x_{-n}}^{n}(\1)(w)}{l_x(\cL_{z,x_{-n}}^{n}(\1))}$ is
  normal. Since, for all $z\in \R^n\cap D(t_0,r)$ and all $x\in X$ we have
  that
  \begin{displaymath}
    \frac{\cL_{z,x_{-n}}^{n}(\1)(w)}
    {l_x(\cL_{z,x_{-n}}^{n}(\1))}\xrightarrow[n\to\infty]{ }
    \frac{q_{z,x}(w)}{l_x(q_{z,x})},
  \end{displaymath}
  we conclude that there exists an analytic function $z\mapsto
  g_{z,x}(w)$ such that
  \begin{equation}
    \label{eq:R411}
    \frac{\cL_{z,x_{-n}}^{n}(\1)(w)}{l_x(\cL_{z,x_{-n}}^{n}(\1))}
    \xrightarrow[n\to\infty]{ }g_{z,x}(w).
  \end{equation}
  Since, in addition,
  \begin{displaymath}
    \cL_x\Big(\frac{\cL_{z,x_{-n}}^{n}(\1)(w)}
    {l_x(\cL_{z,x_{-n}}^{n}(\1))}\Big)=
    \frac{\cL_{z,x_{-n}}^{n+1}(\1)(w)}
    {l_{x_1}(\cL_{z,x_{-n}}^{n+1}(\1))}\cdot
    l_{x_1}\Big(\cL_{z,x}\Big(\frac{\cL_{z,x_{-n}}^{n}(\1)(w)}
    {l_x(\cL_{z,x_{-n}}^{n}(\1))}
    \Big)\Big),
  \end{displaymath}
  we therefore get that
  \begin{displaymath}
    \cL_x\Big(\frac{\cL_{z,x_{-n}}^{n}(\1)(w)}
    {l_x(\cL_{z,x_{-n}}^{n}(\1))}\Big)\xrightarrow[n\to\infty ]{ }
    l_{x_1}(\cL_x(g_{z,x}))g_{x_1,z}.
  \end{displaymath}
  Thus, using again (\ref{eq:R411}), we obtain
  $
    \cL_{z,x}(g_{z,x})=l_{x_1}(\cL_{z,x}(g_{z,x}))g_{x_1,z}.
  $
  As for all $z\in D(t_0,r) \cap \R^n$,
  \begin{displaymath}
    g_{z,x}=\frac{q_{z,x}}{l_x(q_{z,x})}=\frac{L_x q_{z,x}}
    {\sum_{k=0}^{N}q_{z,x}(w_k)},
  \end{displaymath}
  we conclude that,
  \begin{equation}
    \label{eq:AN230}
    l_{x_1}(\cL_{z,x}g_{z,x})=l_{x_1}(\cL_{z,x}\frac{q_{z,x}}{l_x(q_{z,x})})=
    \lambda_{z,x}\frac{l_{x_1}(q_{x_1,z})}{l_{x}(q_{z,x})}.
  \end{equation}

  By the very definitions
  \begin{displaymath}
    l_{x_1}(\cL_{z,x}g_{z,x})=(1/L_x)\sum_{k=1}^{L_x}\cL_{z,x}g_{z,x}(w_k)
  \end{displaymath}
  and
  \begin{displaymath}
    \cL_{z,x}g_{z,x}(w)=\sum_{y\in T_x^{-1}(w)}e^{\varphi_{z,x}(y)}g_{z,x}(y).
  \end{displaymath}
  Denote $g_{z,x}(w)$ by $F(z)$ and $\varphi_{z,x}(w)$ by
  $G(z)$. Then, for $z=(z_1,\ldots,z_n)\in D(t_0,r/2)$, and
  $\Gamma(u)=z+((r/2)e^{2\pi i u_1},\ldots,(r/2)e^{2\pi i u_n})$,
  where $u=(u_1,\ldots,u_n)\in [0, 2\pi]^n$, by the Cauchy Integral Formula,
  \begin{displaymath}
    \Big|\frac{\partial F}{\partial z_k}(z)\Big|
    =\Big|\frac{1}{(2\pi i)^2}\int_{\Gamma}
    \frac{F(\xi)}{(\xi_1-z_1)\ldots(\xi_k-z_k)^2\ldots(\xi_2-z_2)}
    d\xi\Big|\leq 2K/r
  \end{displaymath}
  for $k=1,\ldots,n$. Similarly we obtain that
    \begin{eqnarray*}
    \Big|\frac{\partial G}{\partial z_k}(z)\Big|\leq 2H/r
  \end{eqnarray*}
  for $k=1,\ldots,n$.
  Then, for $k=1,\ldots,n$,
  \begin{eqnarray*}
    \Big|\frac{\bd e^{\varphi_{z,x}(y)}g_{z,x}(y)}{\bd z_k}\Big|&=
    \Big|\frac{\bd \varphi_{z,t}(w)}{\bd z_k}e^{\varphi_{z,x}(y)}g_{z,x}(y)+
    e^{\varphi_{z,x}(y)}\frac{\bd g_{z,x}(y)}{\bd z_k}\Big|\\ &\leq
    (2H/r)e^{H}K+e^H (2K/r).
  \end{eqnarray*}
  It follows that there exists $C_g$ such that for all $x\in X$,
  \begin{equation}
    \label{eq:R300}
    \Big|\frac{\bd l_{x_1}(\cL_{z,x}g_{z,x})}{\bd z_k}\Big|\leq C_{g}.
  \end{equation}
  Using (\ref{4.18}) we obtain that
  \begin{displaymath}
    C^{-1}_\varphi\leq q_{t_0,x}(y)\leq C_\varphi
  \end{displaymath}
  and then
  \begin{displaymath}
    C^{-1}_\varphi\leq l_x(q_{t_0,x}(y))\leq C_\varphi
  \end{displaymath}
  for all $x\in X$. Moreover, it follows from Lemma~\ref{lem:l1rds18} that
  $
    \lambda_{t_0,x}\geq \exp(-\|\varphi_{t_0,x}\|_\infty).
  $
  Then
  \begin{displaymath}
    z_0:=l_{x_1}(\cL_{t_0,x}g_{t_0,x})=
    \lambda_{t,x}\frac{l_{x_1}(q_{t,x_1})}{l_{x}(q_{t,x})}\geq
    \exp(-\sup_{x\in X}\|\varphi_x\|_\infty)C^{-2}_\varphi>0.
  \end{displaymath}
  Hence, by (\ref{eq:R300}), there exists $r_1>0$ so small that
  \begin{displaymath}
    l_{x_1}(\cL_{z,x}g_{z,x})\in D(z_0, z_0/2)
  \end{displaymath}
  for all $z\in D(t_0,r_1)$. Therefore, for all $x\in X$ we can define
  the function
  \begin{displaymath}
    D(t_0,r_1)\ni z\mapsto \log l_{x_1}(\cL_{z,x}g_{z,x})\in \C.
  \end{displaymath}
  Now consider the holomorphic function
  \begin{displaymath}
    z\mapsto \int \log l_{x_1}(\cL_{z,x}g_{z,x}) dm(x).
  \end{displaymath}
  Since the measure $m$ is $\shift$-invariant, by (\ref{eq:AN230})
  \begin{multline*}
    \int \log l_{x_1}(\cL_{z,x}g_{z,x}) dm(x)=
    \int \log
    \lambda_{z,x}\frac{l_{x_1}(q_{z,x_1})}{l_{x}(q_{z,x})}dm(x)\\
    = \int
    \log \lambda_{z,x} dm +\int l_{x_1}(q_{z,x_1}) dm - \int
    l_{x}(q_{z,x})dm(x)= \int \log \lambda_{z,x} dm = \cE P(\varphi_t)
  \end{multline*}
  for $z\in D(t_0,r_1)\cap\R^n$. Therefore the function
  $D(t_0,r_1)\cap\R^n\ni z\mapsto \cE P(\varphi_z)$ is real-analytic.
\end{proof}

\section{Derivative of the Pressure}
\label{sec:derivative}

Now, let $T:\J\to \J$ be uniformly expanding random map. Throughout
the section, we assume that $\ph \in \cH_m (\cJ )$ is a potential such
that there exist measurable functions $L:X\ni x\mapsto L_x\in\R$ and
$c: X\ni x \mapsto c_x>0$ such that
\begin{equation}
  \label{eq:Pre102}
  S_n\varphi_x (z) \leq - n c_x +L_x
\end{equation}
for every $z\in \J_x$ and $n$ and $\psi\in \cH_m (\cJ )$.  For $t\in
\R$, define
\begin{displaymath}
  \varphi_t:=t\varphi + \psi.
\end{displaymath}

Let $R>0$ and let $|t_0|\leq R/2$. Since we are in the uniform case,
it follows from Remark~\ref{rem:P55} that there exist constants $A_R$
and $B_R$ such that, for $t\in [-R,R]$,
\begin{equation}
  \label{eq:der80}
  \Big|\Big|\frac{\tcL_{t,x}^n g_x}{q_{\shift^n(x)}} -
  \Big(\int g_xd\nu_{t,x}\Big)\Big|\Big|_\infty\leq
  \Big(\|g_x\|_\infty +2\frac{v(g_x)}{Q}\Big)A_R
  B_R^n.
\end{equation}

\begin{prop}
  \label{prop:derivative}
  \begin{displaymath}
    \frac{d\cE P(t)}{dt}
    =\int \varphi_x d\mu_x^t dm(x)=\int \varphi d\mu^t.
  \end{displaymath}
\end{prop}
\begin{proof}
  Assume without loss of generality that $|t|\leq R/2$ for some $R>0.$
  Let $x\mapsto y(x)\in Y_x$ be a measurable function and let
  \begin{displaymath}
    \cE P(t,n):=\int \frac{1}{n}\log\cL_{t,x}^{n} \1_{x}(y(x_n)) dm(x).
  \end{displaymath}
  Then $\lim_{n\to\infty}\cE P(t,n)=\cE P(t)$ by
  Lemma~\ref{lem:FormOfPressPart}.  Fix $x\in X$ and put
  $y_n:=y(x_n)$. Observe that
  \begin{displaymath}
    \begin{split}
      \frac{d\cL_{t,x}^{n} \1_{x}(y_{n})}{dt}&=\sum_{y\in
        T_x^{-n}(y_n)}e^{S_n(\varphi_x^t)(y)}S_n\varphi_x (y)\\&=
      \sum_{j=0}^{n-1} \sum_{y\in
        T_x^{-n}(y_n)}e^{S_n(\varphi_x^t)(y)}\varphi_{x_j}(T_x^jy)
      =\sum_{j=0}^{n-1} \cL_{t,x}^{n}(\varphi_{x_j}\circ T_x^j)(y_{n}).
    \end{split}
  \end{displaymath}
  Since
 $
    S_n(\varphi_x^t)(y)=S_j(\varphi_x^t)(y)+
    S_{n-j}(\varphi_{x_j}^t)(T_x^jy)
 $
  we have that
  \begin{displaymath}
    \cL_{t,x}^{n}(\varphi_{x_j}\circ T_x^j)(y(x_{n}))
    =\cL_{t,x_j}^{n-j}(\varphi_{x_j}\cL_{t,x}^{j}\1_x)(y(x_{n})).
  \end{displaymath}
  Then by a version of Leibniz integral rule (see for example
  \cite{Mal95}, Proposition 7.8.4 p. 40)
  \begin{displaymath}
    \frac{d\cE P(t,n)}{dt}=\int \frac{\frac{1}{n}\sum_{j=0}^{n-1}
      \cL_{x_{j},t}^{n-j}(\varphi_{x_{j}}\cL_{t,x}^{j}\1_x)(y(x_n))}
    {\cL_{t,x}^{n}
      \1_{{x}}(y_n)} dm(x).
  \end{displaymath}
Since
  $  \cL_{t,x_j}^{n-j}(\varphi_{x_j}\cL_{t,x}^{j}\1)(y_{n})
    =\lambda_{x}^{n}\tcL_{t,x_j}^{n-j}
    \Big(\varphi_{t,x_j}\tcL_{t,x}^{j}\1_x\Big)
    (y_{n})
 $
  and
  $
    \cL_{t,x}^{n}
      \1_{{x}}(y_n)=\lambda_{x}^{n}\tcL_{t,x}^{n}
      \1_{{x}}(y_n)
 $
 we have that
  \begin{equation}
    \label{eq:der110}
     \frac{\cL_{t,x}^{n}(\varphi_{x_j}\circ T_x^j)(y_{n})} {\cL_{t,x}^{n}
      \1_{x}(y_{n})}=
    \frac{\tcL_{t,x_j}^{n-j}
      \Big(\varphi_{x_j}\tcL_{t,x}^{j}\1_x\Big)
      (y_{n})}{\tcL_{t,x}^{n}\1_{x}(y_{n})} .
  \end{equation}
  The function $\varphi_{x_j}\tcL_{t,x}^{j}\1_x$ is uniformly
  bounded. So does its H\"older variation. Therefore it follows from
  (\ref{eq:der80}), that there exists a constant $A_R$ and $B_R$ such
  that
  \begin{displaymath}
    \Big\|\tcL_{t,x_j}^{n-j}
    \Big({\varphi_{x_j}\tcL_{t,x}^{j}\1_x}\Big)
    (y_{n})/q_{x_n} -
    \Big(\int \varphi_{x_j}\tcL_{t,x}^{j}\1_{x}
    d\nu_{x_j}^t\Big)\Big\|_\infty\leq A_R B_{R}^{n-j}
  \end{displaymath}
  and
  \begin{displaymath}
    \Big\|{\tcL_{t,x}^{n}
      (\1_x)(y_{n})}/q_{x_n}-\1_{x_n}\Big\|_\infty\leq A_R B_R^{n},
  \end{displaymath}
  From this by (\ref{eq:der110}) it follows that
  \begin{displaymath}
    \frac{\int \varphi_{x_j}\tcL_{t,x}^{j}\1_{x}
      d\nu_{x_j}^t - A_R B_R^{n-j}}{1 +
      A_R B_R^{n}}\leq
    \frac{\cL_{t,x}^{n}(\varphi_{x_j}\circ T_x^j)(y_{n})}
    {\cL_{x}^{n} \1_{Y_x}(y_{n})}\leq
    \frac{\int \varphi_{x_j}\tcL_{t,x}^{j}\1_{x}
      d\nu_{x_j}^t + A_R B_R^{n-j}}{1 - A_R
      B_R^{n}},
  \end{displaymath}
  Since $m$ is $\shift$-invariant, we have that
  \begin{displaymath}
    \int \int \varphi_{x_j}\tcL_{t,x}y^{j}\1_{x}
    d\nu_{x_j}^t dm(x)=\int \int \varphi_{x}\tcL_{x_{-j},t}^{j}\1_{x_{-j}}
    d\nu_x^t dm(x).
  \end{displaymath}
  Hence, for large $n$,
  \begin{multline*}
    \frac{\int\int
      \varphi_{x}\Big(\frac{1}{n}\sum_{j=0}^{n-1}\tcL_{x_{-j},t}^{j}
      \1_{x_{-j}}\Big) d\nu_{x}^tdm(x) - \frac{1}{n}\sum_{j=0}^{n-1}
      (A_R B_R^{n-j})}{1 + A_R B_R^{n}} \leq\frac{d\cE
      P(\varphi^t,n)}{dt} \\ \leq \frac{\int\int
      \varphi_{x}\Big(\frac{1}{n}\sum_{j=0}^{n-1}\tcL_{x_{-j},t}^{j}
      \1_{x_{-j}}\Big) d\nu_{x}^tdm(x) - \frac{1}{n}\sum_{j=0}^{n-1}
      (A_R B_R^{n-j})}{1 - A_R B_R^{n}}.
  \end{multline*}
  Therefore
  \begin{displaymath}
    \lim_{n\to\infty}\frac{d\cE P(t,n)}{dt}=\int \varphi_x d\mu_x^t dm(x)
  \end{displaymath}
  uniformly for $t\in [-R,R]$.
\end{proof}

%%% Local Variables:
%%% mode: latex
%%% TeX-master: "RDSmain"
%%% End: